\documentclass[12pt]{amsart}
\usepackage{amsmath,times,epsfig,amssymb,mathtools,braket,amsbsy,amscd,amsfonts,amstext,color,bm,mathrsfs}
\usepackage[margin=1in]{geometry}
\usepackage{enumerate}
\usepackage{placeins}
\usepackage{array}
\usepackage{adjustbox}
\usepackage{tabulary}
\usepackage{multirow}
\usepackage{graphicx}
\usepackage{hyperref}
\usepackage{hhline}
\usepackage{dashbox}
\usepackage{soul}
\usepackage{multirow}
\usepackage[table, svgnames, dvipsnames]{xcolor}
\usepackage{makecell, cellspace, caption}
\usepackage{subcaption}
\usepackage{todonotes}
\usepackage[color,matrix,arrow]{xypic}
\usepackage{tikz}
        \usetikzlibrary{arrows.meta}
        \usetikzlibrary{arrows}
        \usetikzlibrary{quotes}
        \usetikzlibrary{graphs}
        \usetikzlibrary{positioning,calc}
\usepackage{tikz-cd}
\hypersetup{
  colorlinks   = true, %Colours links instead of ugly boxes
  urlcolor     = blue, %Colour for external hyperlinks
  linkcolor    = blue, %Colour of internal links
  citecolor   = red %Colour of citations
}

\usetikzlibrary{shapes}

\usetikzlibrary{decorations}
 
\tikzstyle{v} = [circle, draw, inner sep=2pt, minimum size=3pt, fill=black]
\tikzstyle{l} = [rectangle, draw, rounded corners]
\tikzstyle{c} = [circle, draw, inner sep=0pt, minimum size=14pt, text height=1.5ex, text depth=.25ex, font=\small]
\tikzstyle{label1} = [black, thick]
\tikzstyle{label2} = [blue, thick]
\tikzstyle{label3} = [teal, thick]
\tikzstyle{label4} = [orange, thick]
\tikzstyle{label5} = [magenta, thick]

\theoremstyle{plain}
\newtheorem{theorem}{Theorem}[section]
\newtheorem{corollary}[theorem]{Corollary}
\newtheorem{lemma}[theorem]{Lemma}

\newtheorem{proposition}[theorem]{Proposition}

\makeatletter
    
    \@addtoreset{equation}{section}
  \makeatother

\theoremstyle{definition}
\newtheorem{definition}[theorem]{Definition}

\newtheorem{def-prop}[theorem]{Definition \& Proposition}

\newtheorem{example}[theorem]{Example}

\newtheorem{remark}[theorem]{Remark}

\newcommand{\D}{{\mathcal{D}}}

\newcommand{\tbf}{\textbf}

\newcommand{\rev}{\operatorname{rev}}
\newcommand{\rich}{\operatorname{rich}}

\DeclareMathOperator{\rk}{rank}
\DeclareMathOperator{\id}{id}

\DeclareMathOperator{\SpMAT}{\mathscr{G}}
\DeclareMathOperator{\SpRV}{\mathscr{V}}
\DeclareMathOperator{\SpASPD}{\mathscr{D}}
\DeclareMathOperator{\SpSplit}{\mathsf{Sp}}

\newcommand{\symdiff}{\mathbin{\triangle}}

\newcolumntype{K}[1]{>{\centering\arraybackslash}p{#1}}

\allowdisplaybreaks

\makeatletter
\@namedef{subjclassname@2020}{%
  \textup{2020} Mathematics Subject Classification}
\makeatother

\begin{document}

\title[Splitting and merging of combinatorial species]{An axiomatic framework from splitting and merging in MAT-labeled graphs, vines, and single-peaked domains}

\begin{abstract}
In recent work (Forum Math.~Sigma, 2024), we established a correspondence between MAT-labeled graphs arising from hyperplane arrangements and regular vines from probability theory. 
In this paper, we extend this connection to Arrow's single-peaked domains in social choice theory. 
We show that MAT-labeled complete graphs, regular vines, and maximal Arrow's single-peaked domains arise from the same recursive combinatorial structure.

Our main result gives an axiomatic characterization of these objects using the language of combinatorial species. 
At the heart of this characterization are two fundamental operations, called splitting and merging, together with natural compatibility conditions that uniquely determine the structures. 
As consequences, we obtain explicit correspondences between maximal Arrow's single-peaked domains, MAT-labeled complete graphs, and regular vines, thereby providing new combinatorial and axiomatic characterizations of these domains.

 We further show that regular vines are equivalent to $(n,3)$-extremal lattices from formal concept analysis, and that these lattices are in turn equivalent to extremal binary matrices with no triangles from combinatorial matrix theory. Consequently, these lattices and matrices also fit naturally into the same splitting-and-merging framework, providing further examples from different areas unified by our axiomatic characterization.
\end{abstract}

 \author{Hung Manh Tran}
\address{Hung Manh Tran, Faculty of Fundamental Sciences, Phenikaa University, Hanoi 12116, Vietnam.}
\email{hung.tranmanh@phenikaa-uni.edu.vn}

\author{Tan Nhat Tran}
\address{Tan Nhat Tran, Department of Mathematics and Statistics, Binghamton University (SUNY), Binghamton, NY 13902-6000, USA.}
\email{tnhattran@binghamton.edu}

\author{Shuhei Tsujie}
\address{Shuhei Tsujie, Department of Mathematics, Hokkaido University of Education, Asahikawa, Hokkaido 070-8621, Japan}
\email{tsujie.shuhei@a.hokkyodai.ac.jp}


\date{\today}
\maketitle

\tableofcontents

%********************************************************************************************************
\section{Introduction}
\label{sec:intro}

%******************************************************************************** 
\subsection{Background}
\label{subsec:BG}

The motivation of this work comes from an unexpected interaction between several combinatorial structures arising in different areas of mathematics. 
In our previous work \cite{TTT24}, we established a correspondence between \emph{MAT-labeled graphs}, which originate in the theory of hyperplane arrangements, and \emph{locally regular vines}, which appear in probability theory. 
In the present paper we show that this connection extends further and naturally relates to \emph{single-peaked domains} in social choice theory.

\medskip

The first main concept in this paper is the notion of MAT-labeled graphs.

 \begin{definition}[MAT-labeled graphs {\cite{TT23,TTT24}}]
\label{definition MAT-labeling}
Let $G=(A,E)$ be a finite simple graph with vertex set $A$ and edge set $E$. 
Let $ \lambda \colon E \longrightarrow \mathbb{Z}_{>0} $ be a map.
For $k>0$, let $\pi_k$, $\pi_{\le k}$, and $\pi_{<k}$ denote the sets of edges with label exactly $k$, at most $k$, and less than $k$, respectively.
A map $ \lambda \colon E \longrightarrow \mathbb{Z}_{>0} $ is called an \textbf{MAT-labeling} of $G$ if the following conditions hold for every $ k >0$. 
\begin{enumerate}[(1)]
\item\label{definition MAT-labeling cycle} 
Edges in $\pi_{k}$ do not form a cycle with an edge in $\pi_{\leq k}$. 
\item\label{definition MAT-labeling triangle} 
Every edge in $\pi_k$ forms exactly $k-1$ triangles with edges in $\pi_{<k}$.
\end{enumerate}

An edge-labeled graph $(G,\lambda)$ is an \tbf{MAT-labeled graph} if $\lambda$ is an MAT-labeling of $G$. 
If $G$ is a complete graph, we refer to $(G,\lambda)$ as an MAT-labeled complete graph.
\end{definition}

\begin{example}
The figure below illustrates an MAT-labeled complete graph on the vertex set $A = \{a,b,c,d\}$.
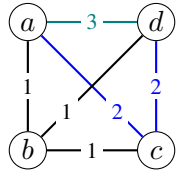
\begin{center}
\begin{tikzpicture}[baseline=(current bounding box.center)]
\draw (135: 12mm) node[c](0){$a$};
\draw (225: 12mm) node[c](1){$b$};
\draw (315: 12mm) node[c](2){$c$};
\draw (405: 12mm) node[c](3){$d$};
\draw[label1] (0) --node[midway, fill=white, inner sep=2pt] {\scalebox{.7}{1}} (1);
\draw[label2] (0) --node[pos=0.75, fill=white, inner sep=2pt] {\scalebox{.7}{2}} (2);
\draw[label3] (0) --node[midway, fill=white, inner sep=2pt] {\scalebox{.7}{3}} (3);
\draw[label1] (1) --node[midway, fill=white, inner sep=2pt] {\scalebox{.7}{1}} (2);
\draw[label1] (1) --node[pos=0.25, fill=white, inner sep=2pt] {\scalebox{.7}{1}} (3);
\draw[label2] (2) --node[midway, fill=white, inner sep=2pt] {\scalebox{.7}{2}} (3);
\end{tikzpicture}
\end{center}
\end{example}
\vskip .5em
 
The study of MAT-labelings originates from a question of Cuntz--M{\"u}cksch \cite{CM20} concerning \emph{free hyperplane arrangements}. 
A hyperplane arrangement is a finite collection of linear hyperplanes in a vector space. 
Such an arrangement is called \emph{free} if its module of logarithmic derivations is free \cite{T80,OT92}. 
Freeness has been a central topic in the theory of hyperplane arrangements for several decades, and a major direction in the area is to understand this algebraic property through combinatorial structures associated with the arrangement.

One important notion in this direction is \emph{MAT-freeness}, introduced by Abe--Barakat--Cuntz--Hoge--Terao \cite{ABCHT16}. 
An arrangement is called MAT-free if its hyperplanes admit an \emph{MAT-partition}, namely a partition satisfying certain combinatorial conditions. 
MAT-freeness implies freeness and played a crucial role in the proof of the Sommers--Tymoczko conjecture \cite{ST06} on the freeness of ideal subarrangements of Weyl arrangements.

Intuitively, an MAT-partition organizes the hyperplanes into layers. 
On the other hand, the motivating examples coming from root systems also carry a natural partial order given by the root poset. 
This leads naturally to the question of whether MAT-partitions can be described through a suitable poset structure extending the classical root poset. 
Such a question was posed by Cuntz--M{\"u}cksch \cite[Problem~47]{CM20} and motivated our earlier work \cite{TTT24}, where we answered this question for graphic arrangements.

Subarrangements of a type~$A$ Weyl arrangement are in one-to-one correspondence with graphic arrangements, and their MAT-partitions correspond precisely to MAT-labelings of the underlying graphs. 
It is well known that a graphic arrangement is free if and only if the underlying graph is chordal (e.g., \cite[Theorem~3.3]{ER94}). 
Consequently, every graph admitting an MAT-labeling must be chordal. 
Moreover, it was shown in \cite{TT23} that a graph admits an MAT-labeling if and only if it is strongly chordal. 
Thus, although MAT-freeness originated in the theory of hyperplane arrangements, MAT-labelings also arise naturally as a graph-theoretic property.

\medskip

The second structure appearing in this story is that of  regular vines. 
\begin{definition}[Regular vines \cite{BC02}]\label{def:regular vine; graph}
A sequence $\mathcal{V} = (T_{1}, \dots, T_{n})$ is a \textbf{regular vine} on a finite set $A$ with $n$ elements if the following conditions hold. 
\begin{enumerate}[(1)]
\item $T_{1}$ is a tree on $A$. 
\item $T_{i}$ is a spanning tree of the line graph of $T_{i-1}$ for each $i \in \{2,\dots, n\}$. 
\end{enumerate}
\end{definition}

The nested incidence relations among the vertices and edges of a regular vine naturally define a poset structure. 
This leads to an equivalent poset description of regular vines, which will be used throughout the paper.

\begin{definition}[Poset definition of regular vines {\cite[Definition 3.14 and Proposition 4.8]{TTT24}}]\label{def:regular vine; poset}
An induced subposet $\mathcal{V}$ of the Boolean lattice $(2^{A}, \subseteq)$ is called a \textbf{regular vine} on an $n$-element set $A$ if the following conditions hold: 
\begin{enumerate}[(1)]
\item All maximal chains have length $n-1$. 
Hence, $\mathcal{V}$ is graded. 
We assume that every minimal element of $\mathcal{V}$ has rank $1$. 
\item The number of minimal elements of $\mathcal{V}$ equals $n$. 
Consequently, the minimal elements of $\mathcal{V}$ are the singletons $\{a\}$ for all $a \in A$, and the maximal element of $\mathcal{V}$ is $A$. 
Let $\mathcal{V}(i)$ denote the set of elements of $\mathcal{V}$ with rank (equivalently, cardinality) $i$.
\item Every non-minimal element covers exactly two elements. 
\item For each $1 \leq i \leq n-1$, the graph on $\mathcal{V}(i)$ obtained by viewing each element of $\mathcal{V}(i+1)$ as an edge joining the two elements it covers is a tree, called the \textbf{$i$-th associated tree}. 
\item (\textbf{Proximity}) If two elements in $\mathcal{V}(i)$ with $i \geq 2$ are covered by a common element, then they cover a common element. 
\end{enumerate}
An ideal (i.e.,~ a downward-closed subset) of a regular vine is called a \textbf{locally regular vine}, or equivalently, an \textbf{m-saturated vine}; 
see \cite[Theorem 6.13]{TTT24} and \cite[Definition 4.2]{KC06}.
\end{definition}

\begin{example}
The left-hand figure below illustrates a regular vine on $A = \{a,b,c,d\}$ in the sense of Definition \ref{def:regular vine; graph}. 
By declaring each edge in a tree of the regular vine to cover its endpoints, we obtain the poset shown in the right-hand figure.

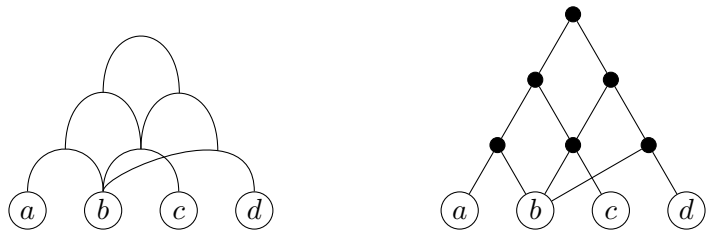
\begin{center}
\begin{tikzpicture}[baseline=0]
\draw (0,0) node[c](a){$a$};
\draw (1,0) node[c](b){$b$};
\draw (2,0) node[c](c){$c$};
\draw (3,0) node[c](d){$d$};
\draw (a) .. controls +(0,1) and +(0,1) .. coordinate[pos=0.5] (ab) (b);
\draw (b) .. controls +(0,1) and +(0,1) .. coordinate[pos=0.5] (bc) (c);
\draw (b) .. controls +(0,0.5) and +(0,1.35) .. coordinate[pos=0.68] (bd) (d);
\draw (ab) .. controls +(0,1) and +(0,1) .. coordinate[pos=0.5](abc) (bc);
\draw (bc) .. controls +(0,1) and +(0,1) .. coordinate[pos=0.5](bcd) (bd);
\draw (abc) .. controls +(0,1) and +(0,1) .. (bcd);
\end{tikzpicture}
\hspace{20mm}
\begin{tikzpicture}[baseline=0]
\draw (0, 0) node[c](0){$a$};
\draw (1, 0) node[c](1){$b$};
\draw (2, 0) node[c](2){$c$};
\draw (3, 0) node[c](3){$d$};
\draw (0.5, 0.865) node[v](01){};
\draw (0)--(01);
\draw (1)--(01);
\draw (1.5, 0.865) node[v](12){};
\draw (1)--(12);
\draw (2)--(12);
\draw (2.5, 0.865) node[v](13){};
\draw (1)--(13);
\draw (3)--(13);
\draw (1.0, 1.730) node[v](012){};
\draw (01)--(012);
\draw (12)--(012);
\draw (2.0, 1.730) node[v](123){};
\draw (12)--(123);
\draw (13)--(123);
\draw (1.5, 2.595) node[v](0123){};
\draw (012)--(0123);
\draw (123)--(0123);
\end{tikzpicture}
\end{center}
\vskip 1em
Identifying each element of this poset with the subset of $A$ that it dominates, we may realize it as an induced subposet of the Boolean lattice $(2^{A}, \subseteq)$. 
Conversely, by taking the associated trees of a regular vine in the sense of Definition \ref{def:regular vine; poset}, we recover the regular vine in the original sense. 
\end{example}

In \cite{TTT24} we proved that the categories of MAT-labeled graphs and locally regular vines are equivalent.
In particular, MAT-labeled complete graphs correspond exactly to regular vines. 
This result provides a positive answer to the question of Cuntz--M{\"u}cksch in the special case of graphic arrangements.

Regular vines were originally introduced in probability theory as models for describing  dependence structures in multivariate distributions. 
They have since become an important tool in probability, uncertainty analysis, and related fields. 
For a comprehensive account of vine models and their applications, we refer the reader to \cite{KJ11}.

A key observation behind the correspondence established in \cite{TTT24} was that the numbers of non-isomorphic MAT-labeled complete graphs with $n$ vertices and of non-isomorphic regular vines on $n$ elements coincide for $n\le 8$.
Such numerical coincidences frequently hint at deeper structural relationships, and in this case it ultimately led to the equivalence between the two objects.

\medskip

 In the present work, the same sequence also arises in the enumeration of non-isomorphic \emph{maximal Arrow's single-peaked domains} \cite{Karpov25}. 
This suggests that these structures may admit a common combinatorial framework.

We now recall the notion of Arrow's single-peaked domains.
Let $A$ be a finite set with $n$ elements (the alternatives) and let $\mathcal{L}(A)$ denote the set of linear orders (the preferences) on $A$.
For $\mathcal{D} \subseteq \mathcal{L}(A)$ and $T \subseteq A$, let $\mathcal{D}_{T}$ denote the set of linear orders on $T$ obtained by restricting the linear orders in $\mathcal{D}$ to $T$. 

\begin{definition}[Arrow's single-peaked domains {\cite{A63,I64}}]
A subset $\mathcal{D} \subseteq \mathcal{L}(A)$ is called an \textbf{Arrow's single-peaked domain (ASPD)} if for every triple $T\subseteq A$ there exists $x\in T$ such that $x$ is never ranked last in the restriction $\mathcal{D}_T$. 
An ASPD is \textbf{maximal} if adding any additional preference destroys this property.
\end{definition}

\begin{example}
The table below illustrates an example of a maximal ASPD on $A = \{a,b,c,d\}$. 
Each column represents a linear order, written from top to bottom.
\begin{align*}
\begin{array}{|cccc|cccc|}
\hline
a & b  & b & c & b &  c & b & d \\
b & a & c & b & c &  b & d & b \\
c & c & a & a & d &  d & c & c \\
d & d & d & d & a & a & a & a \\
\hline
\end{array}
\end{align*}
\end{example}
\vskip .5em

Given a collection of voters, each with a preference on $A$, the \emph{majority relation} declares that $x$ is socially preferred to $y$ whenever a strict majority of voters rank $x$ above $y$. 
When the number of voters is odd, this relation is well defined but may fail to be transitive, a phenomenon known as the \emph{Condorcet paradox}.

A domain $\mathcal D \subseteq \mathcal L(A)$ is called a \textbf{Condorcet domain} if for every odd preference profile whose preferences lie in $\mathcal D$, the induced majority relation is always a linear order, thereby avoiding the Condorcet paradox. 
In this case the majority rule always admits a \emph{Condorcet winner}, an alternative that defeats every other alternative in pairwise majority comparisons.

Condorcet domains are classical objects in social choice theory, tracing back to the work of Borda and Condorcet in the late 18th century. 
Understanding their structure is a major problem in the theory and remains largely open. 
Among the known examples, ASPDs form one of the largest and most tractable families, and maximal ASPDs provide important extremal instances within this class.

\medskip

Several combinatorial frameworks for maximal ASPDs have been proposed in the literature. 
For instance, Liversidge \cite{Liver20} described them using Hamiltonian directed paths, Karpov--Slinko \cite{KS23} introduced the concatenation and shuffle construction, and Slinko \cite{Slinko24} developed an approach based on generalized arrangements of pseudolines.
Recently, Karpov \cite{Karpov25} gave a characterization and enumeration of maximal ASPDs via binary matrices.

%******************************************************************************** 
\subsection{Main results}
\label{subsec:main-result}

This paper uncovers a common structure underlying three combinatorial objects that arise in different areas: maximal ASPDs, MAT-labeled complete graphs, and regular vines. 
Despite their different definitions and origins in social choice theory, hyperplane arrangements, and probability theory, we show that these objects admit the same recursive description.

Our approach is to view each of these structures as a combinatorial species and to identify two fundamental operations, called \emph{splitting} and \emph{merging}, that govern their recursive construction. 
The splitting operation leads to a natural \emph{proximity} condition, while merging gives rise to a corresponding \emph{mergeability} condition (Definition \ref{def:prox-merge}). 
We prove that these two axioms uniquely determine the structure, yielding an axiomatic characterization that applies uniformly to all three families. 
This is the central result of the paper and is stated in Theorem~\ref{thm:AxCh}.

The framework has two main advantages.  
First, it explains why these objects share similar structural and enumerative properties.  
Second, it provides a systematic viewpoint for relating and analyzing further examples arising in different areas.  
In particular, we will see in the applications that extremal lattices from formal concept analysis and extremal binary matrices from combinatorial matrix theory also fit naturally into this framework.

As a consequence of Theorem~\ref{thm:AxCh}, isomorphisms between these structures can be constructed inductively from the splitting and merging operations. 
In Theorems~\ref{thm:graph-ASPD} and~\ref{thm:vine-ASPD} we give explicit constructions of these isomorphisms, relating maximal ASPDs to MAT-labeled complete graphs and to regular vines, respectively. 
These constructions make the correspondences transparent and allow structural properties to be translated between the different settings.

\subsection{Applications}
\label{subsec:appl}

Our results lead to two main lines of applications.

The first concerns social choice theory. 
The axiomatic characterization of maximal ASPDs, together with the explicit correspondences with MAT-labeled complete graphs and regular vines, provides two concrete combinatorial models for maximal ASPDs---one graph-theoretic and one poset-theoretic. 
The representation via regular vines also yields an explicit formula for the number of non-isomorphic maximal ASPDs (Corollary~\ref{cor:key-seq}).

The correspondence with regular vines further yields several applications in social choice theory, which are developed in Section \ref{sec:appl-sct}. 
First, maximal Black's single-peaked domains correspond precisely to D-vines. 
Second, the distribution of first-ranked alternatives in a maximal ASPD can be computed via maximal chains of the corresponding regular vine, leading to a rule analogous to Pascal's triangle. 
Third, we obtain a characterization of the richness property in terms of the combinatorial structure of regular vines.

The second line of applications illustrates the broader scope of our axiomatic framework.
Using a direct combinatorial proof (Theorem~\ref{thm:EL-vine}), we show that regular vines are equivalent to $(n,3)$-extremal lattices arising in formal concept analysis, a field that studies data through object--attribute relationships.
We further show (Proposition~\ref{prop:EL-EM}) that these extremal lattices are equivalent to extremal binary matrices with no triangles from combinatorial matrix theory.
This equivalence recovers the characterization of maximal ASPDs via binary matrices recently obtained by Karpov~\cite{Karpov25}.
Consequently, both the extremal lattices and the extremal binary matrices admit the same splitting-and-merging structure and therefore fit naturally into our framework.

Furthermore, a recursive formula for the number of non-isomorphic $(n,3)$-extremal lattices is already known. 
Via the correspondence established here, this immediately yields a recursive formula for the number of non-isomorphic regular vines, MAT-labeled complete graphs, and maximal ASPDs (Corollary~\ref{cor:recursive}).

\medskip

We summarize the relationships between the main structures and the resulting applications in Figure \ref{fig:summary}.

\newsavebox{\boxMAT}
\savebox{\boxMAT}{
\begin{tikzpicture}[scale=0.9, transform shape]
\draw (135: 10mm) node[c](0){$a$};
\draw (225: 10mm) node[c](1){$b$};
\draw (315: 10mm) node[c](2){$c$};
\draw (405: 10mm) node[c](3){$d$};
\draw[label1] (0) --node[midway, fill=white, inner sep=2pt] {\scalebox{.7}{1}} (1);
\draw[label2] (0) --node[pos=0.75, fill=white, inner sep=2pt] {\scalebox{.7}{2}} (2);
\draw[label3] (0) --node[midway, fill=white, inner sep=2pt] {\scalebox{.7}{3}} (3);
\draw[label1] (1) --node[midway, fill=white, inner sep=2pt] {\scalebox{.7}{1}} (2);
\draw[label1] (1) --node[pos=0.25, fill=white, inner sep=2pt] {\scalebox{.7}{1}} (3);
\draw[label2] (2) --node[midway, fill=white, inner sep=2pt] {\scalebox{.7}{2}} (3);
\end{tikzpicture}
}

\newsavebox{\boxRV}
\savebox{\boxRV}{
\begin{tikzpicture}[scale=0.8, transform shape]
\draw (0, 0) node[c](0){$a$};
\draw (1, 0) node[c](1){$b$};
\draw (2, 0) node[c](2){$c$};
\draw (3, 0) node[c](3){$d$};
\draw (0.5, 0.865) node[v](01){};
\draw (0)--(01);
\draw (1)--(01);
\draw (1.5, 0.865) node[v](12){};
\draw (1)--(12);
\draw (2)--(12);
\draw (2.5, 0.865) node[v](13){};
\draw (1)--(13);
\draw (3)--(13);
\draw (1.0, 1.730) node[v](012){};
\draw (01)--(012);
\draw (12)--(012);
\draw (2.0, 1.730) node[v](123){};
\draw (12)--(123);
\draw (13)--(123);
\draw (1.5, 2.595) node[v](0123){};
\draw (012)--(0123);
\draw (123)--(0123);
\end{tikzpicture}
}

\newsavebox{\boxASPD}
\savebox{\boxASPD}{
\begin{tikzpicture}[scale=0.75, transform shape]
\draw (0,0) node{$\begin{array}{|cccc|cccc|}
\hline
a & b  & b & c & b &  c & b & d \\
b & a & c & b & c &  b & d & b \\
c & c & a & a & d &  d & c & c \\
d & d & d & d & a & a & a & a \\
\hline
\end{array}$};
\end{tikzpicture}
}

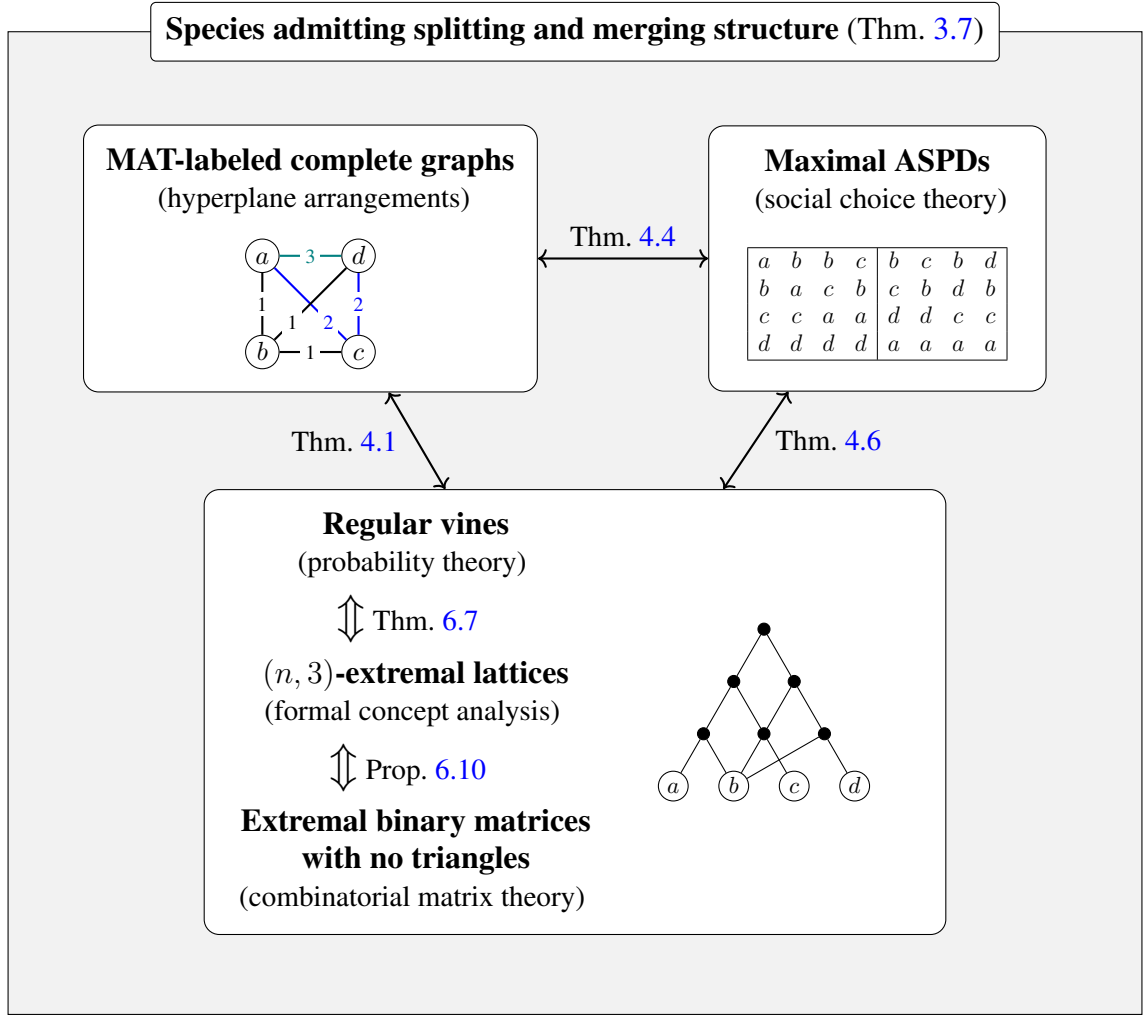
\begin{figure}[h]
\centering
\begin{tikzpicture}
\draw[fill=gray!10] (0,-1) rectangle (15,12);

\draw (7.5,12) node[draw, fill=white, rounded corners = 2pt, inner sep = 2mm]{\textbf{Species admitting splitting and merging structure} (Thm.~\ref{thm:AxCh})};

\draw (4,9) node[draw, fill=white, rectangle, rounded corners = 2mm, align=center, inner sep = 3mm](G){\textbf{MAT-labeled complete graphs} \\ \begin{small}
(hyperplane arrangements)
\end{small} \\[3mm]
\usebox{\boxMAT}}; 

\draw (11.5,9) node[draw, fill=white, rectangle, rounded corners = 2mm, align=center, inner sep = 3mm](D){\textbf{Maximal ASPDs} \\ \begin{small}
(social choice theory)
\end{small} \\[3mm] 
\usebox{\boxASPD}}; 

\draw (7.5,3) node[draw, fill=white, rectangle, rounded corners = 2mm, align=center, inner sep = 3mm](V){
\begin{minipage}[]{50mm}
\centering
\textbf{Regular vines} \\
\begin{small}
(probability theory)
\end{small}

\vspace{2mm}
{
{\Large $\Updownarrow$}  
{\small Thm.~\ref{thm:EL-vine}}
}
\vspace{2mm}

\textbf{$(n,3)$-extremal lattices} \\
\begin{small}
(formal concept analysis)
\end{small}

\vspace{2mm}
{
{\Large $\Updownarrow$}  
{\small Prop.~\ref{prop:EL-EM}}
}
\vspace{2mm}

\textbf{Extremal binary matrices with no triangles} \\
\begin{small}
(combinatorial matrix theory)
\end{small}
\end{minipage}
\begin{minipage}[]{40mm}
\centering
\usebox{\boxRV}
\end{minipage}
};

\draw[<->, thick] (G) -- node[above]{\small Thm.~\ref{thm:graph-ASPD}} (D);
\draw[<->, thick] (G) -- node[left=3pt]{\small Thm.~\ref{thm:graph-vine}} (V);
\draw[<->, thick] (D) -- node[right=3pt]{\small Thm.~\ref{thm:vine-ASPD}} (V);
\end{tikzpicture}
\caption{Relationships between the combinatorial structures considered in this paper.}
\label{fig:summary}
\end{figure}
 
 \vskip 1em
\noindent
\textbf{Acknowledgements.} 
The authors thank A.~V.~Karpov for informing us, after the first draft of this paper appeared, of his work \cite{Karpov25} on the characterization of maximal ASPDs via extremal binary matrices with no triangles. 
This led us to observe the equivalence between $(n,3)$-extremal lattices and extremal binary matrices with no triangles established in Proposition~\ref{prop:EL-EM}.
S.~Tsujie was supported by JSPS KAKENHI Grant Numbers JP23H00081 and JP26K16955.  

%********************************************************************************************************
 \section{Preliminaries}
\label{sec:Preliminaries}
%******************************************************************************** 
 \subsection{MAT-labeled graphs}
\label{subsec:graphs}

In this subsection, we recall several basic properties of MAT-labeled graphs that will be used later. 
All graphs in this paper are finite, undirected, and simple.
Let $G=(A,E)$ be a graph and let $\lambda \colon E \longrightarrow \mathbb{Z}_{>0}$ be an edge-labeling. 
For simplicity of notation, we write
\[
\lambda(a,b)\coloneqq\lambda(\{a,b\})
\]
for the label of an edge $\{a,b\}\in E$. 

\begin{definition}[MAT-simplicial vertices {\cite[Definition 5.1]{TT23}}]
\label{definition MAT-simplicial}
Let $(G,\lambda)$ be an edge-labeled graph. 
A vertex $a\in A$ is called \textbf{MAT-simplicial} if the following conditions hold:
\begin{enumerate}[(1)]
\item $a$ is simplicial in $G$, that is, the neighborhood of $a$ forms a clique;
\item the edges of $G$ incident to $a$ have labels $1,2,\dots,\deg_G(a)$, where $\deg_G(a)$ denotes the degree of $a$;
\item for any distinct vertices $b,c$ adjacent to $a$,
\[
\lambda(b,c) < \max\{\lambda(b,a),\lambda(c,a)\}.
\]
\end{enumerate}
\end{definition}

\begin{definition}[MAT-perfect elimination orderings {\cite[Definition 5.4]{TT23}}]
\label{def:MAT-PEO}
Let $(G,\lambda)$ be an edge-labeled graph on $n$ vertices. 
An ordering $(a_1,\dots,a_n)$ of the vertices of $G$ is called an \textbf{MAT-perfect elimination ordering (MAT-PEO)} if, for each $i$, the vertex $a_i$ is MAT-simplicial in the induced subgraph of $G$ on $\{a_{1}, \dots, a_{i}\}$ equipped with the restriction of $\lambda$. 
\end{definition}

\begin{theorem}[{\cite[Theorem 5.5]{TT23}}]
\label{thm:MAT-PEO}
An edge-labeled graph $(G,\lambda)$ is MAT-labeled if and only if it admits an MAT-PEO.
\end{theorem}

Let  $K_A$ denote the complete graph on vertex set $A$.
In this paper, we will be particularly interested in MAT-labeled complete graphs. 
We recall two structural properties that will play a central role in our framework.

\begin{lemma}[Splitting MAT-labeled complete graphs {\cite[Lemma 5.2]{TT23}}]
\label{lem:splitting-MAT}
Let $(K_A,\lambda)$ be an MAT-labeled complete graph with $|A|\ge2$. 
Then it has exactly two MAT-simplicial vertices $a_1,a_2$, namely the endpoints of the edge with largest label. 
Let
\[
G_i \coloneqq K_{A\setminus \{a_i\}}, \qquad
\lambda_i \coloneqq \lambda|_{E_{G_i}} \quad (i=1,2),
\]
and
\[
G' \coloneqq G_{1} \cap G_{2} = K_{A\setminus\{a_{1},a_{2}\}}, \qquad
\lambda' \coloneqq \lambda|_{E_{G'}}.
\]
Then $(G_i,\lambda_i)$ and $(G',\lambda')$ are MAT-labeled complete graphs.
\end{lemma}

\begin{lemma}[Merging MAT-labeled complete graphs {\cite[Lemma 5.7]{TT23}}]
\label{lem:merge-MAT}
Let $A$ be a finite set and let $a_1,a_2\in A$ be distinct elements. 
For each $i\in\{1,2\}$, let $(G_i,\lambda_i)$ be an MAT-labeled complete graph with vertex set $A\setminus\{a_i\}$. 
Let
\[
G' \coloneqq G_{1} \cap G_{2} = K_{A\setminus\{a_{1},a_{2}\}},
\]
and assume that
\[
\lambda_{1}\vert_{E_{G^{\prime}}} = \lambda_{2}\vert_{E_{G^{\prime}}} \eqqcolon \lambda',
\]
where $\lambda^{\prime}$ is an MAT-labeling of $G^{\prime}$. 

Define a labeling
\[
\lambda:E_{K_A}\longrightarrow\mathbb Z_{>0}
\]
by
\[
\lambda(e)=
\begin{cases}
\lambda_i(e), & e\in E_{G_i},\\
|A|-1, & e=\{a_1,a_2\}.
\end{cases}
\]
Then $(K_A,\lambda)$ is an MAT-labeled complete graph.
\end{lemma}

\begin{example}
See Figure~\ref{Fig:ex_sp_mg_MAT} for an example of splitting and merging MAT-labeled complete graphs.
The graph $G$ splits into $G_{1}, G_{2}$, and $G^{\prime}$, each equipped with the restricted labeling.
Conversely, merging $G_{1}$ and $G_{2}$ along $G^{\prime}$ recovers $G$. 
\begin{figure}[htbp]
\centering
\begin{tikzpicture}[baseline=0]
\draw (126: 15mm) node[c](0){$a$};
\draw (198: 15mm) node[c](1){$b$};
\draw (270: 15mm) node[c](2){$c$};
\draw (342: 15mm) node[c](3){$d$};
\draw (414: 15mm) node[c](4){$e$};
\draw[label1] (0)--node[midway, fill=white, inner sep=1pt] {\scalebox{.8}{1}} (1) ;
\draw[label2] (0)--node[midway, fill=white, inner sep=1pt] {\scalebox{.8}{2}} (2);
\draw[label3] (0)--node[midway, fill=white, inner sep=1pt] {\scalebox{.8}{3}}(3);
\draw[label4] (0)--node[midway, fill=white, inner sep=1pt] {\scalebox{.8}{4}}(4);
\draw[label1] (1)--node[midway, fill=white, inner sep=1pt] {\scalebox{.8}{1}} (2);
\draw[label2] (1)--node[midway, fill=white, inner sep=1pt] {\scalebox{.8}{2}} (3);
\draw[label3] (1)--node[midway, fill=white, inner sep=1pt] {\scalebox{.8}{3}}(4);
\draw[label1] (2)--node[midway, fill=white, inner sep=1pt] {\scalebox{.8}{1}} (3);
\draw[label1] (2)--node[midway, fill=white, inner sep=1pt] {\scalebox{.8}{1}} (4);
\draw[label2] (3)--node[midway, fill=white, inner sep=1pt] {\scalebox{.8}{2}}(4);
\draw (0,-2.3) node{$(G,\lambda)$}; 
\end{tikzpicture}
\hspace{5mm}
\begin{tikzpicture}[baseline=0]
\draw (126: 15mm) node[c](0){$a$};
\draw (198: 15mm) node[c](1){$b$};
\draw (270: 15mm) node[c](2){$c$};
\draw (342: 15mm) node[c](3){$d$};
\draw[label1] (0)--node[midway, fill=white, inner sep=1pt] {\scalebox{.8}{1}} (1) ;
\draw[label2] (0)--node[midway, fill=white, inner sep=1pt] {\scalebox{.8}{2}} (2);
\draw[label3] (0)--node[midway, fill=white, inner sep=1pt] {\scalebox{.8}{3}}(3);
\draw[label1] (1)--node[midway, fill=white, inner sep=1pt] {\scalebox{.8}{1}} (2);
\draw[label2] (1)--node[midway, fill=white, inner sep=1pt] {\scalebox{.8}{2}} (3);
\draw[label1] (2)--node[midway, fill=white, inner sep=1pt] {\scalebox{.8}{1}} (3);
\draw (0,-2.3) node{$(G_{1},\lambda_{1})$}; 
\end{tikzpicture}
\hspace{5mm}
\begin{tikzpicture}[baseline=0]
\draw (198: 15mm) node[c](1){$b$};
\draw (270: 15mm) node[c](2){$c$};
\draw (342: 15mm) node[c](3){$d$};
\draw (414: 15mm) node[c](4){$e$};
\draw[label1] (1)--node[midway, fill=white, inner sep=1pt] {\scalebox{.8}{1}} (2);
\draw[label2] (1)--node[midway, fill=white, inner sep=1pt] {\scalebox{.8}{2}} (3);
\draw[label3] (1)--node[midway, fill=white, inner sep=1pt] {\scalebox{.8}{3}}(4);
\draw[label1] (2)--node[midway, fill=white, inner sep=1pt] {\scalebox{.8}{1}} (3);
\draw[label1] (2)--node[midway, fill=white, inner sep=1pt] {\scalebox{.8}{1}} (4);
\draw[label2] (3)--node[midway, fill=white, inner sep=1pt] {\scalebox{.8}{2}}(4);
\draw (0,-2.3) node{$(G_{2},\lambda_{2})$}; 
\end{tikzpicture}
\hspace{5mm}
\begin{tikzpicture}[baseline=0]
\draw (198: 15mm) node[c](1){$b$};
\draw (270: 15mm) node[c](2){$c$};
\draw (342: 15mm) node[c](3){$d$};
\draw[label1] (1)--node[midway, fill=white, inner sep=1pt] {\scalebox{.8}{1}} (2);
\draw[label2] (1)--node[midway, fill=white, inner sep=1pt] {\scalebox{.8}{2}} (3);
\draw[label1] (2)--node[midway, fill=white, inner sep=1pt] {\scalebox{.8}{1}} (3);
\draw (0,-2.3) node{$(G^{\prime},\lambda^{\prime})$}; 
\end{tikzpicture}
\caption{Example of splitting and merging of MAT-labeled complete graphs}
\label{Fig:ex_sp_mg_MAT}
\end{figure}
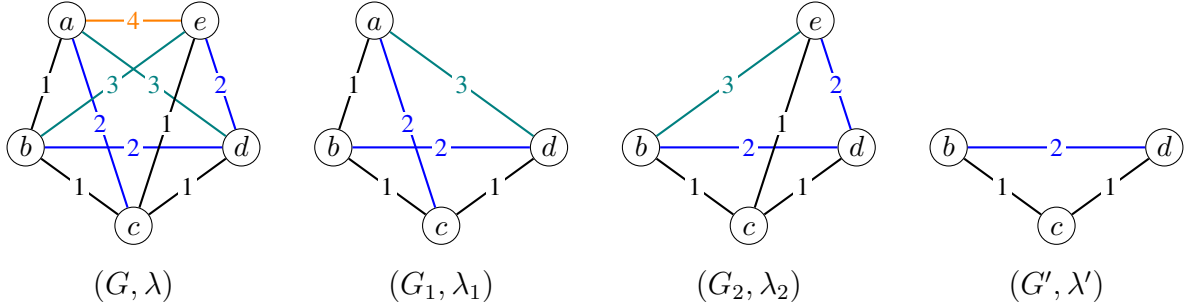
\end{example}

 %********************************************************************************************************
\subsection{Regular vines}
\label{subsec:vine}
  
We recall several structural properties of regular vines established in \cite{TTT24}.

\begin{proposition}[{\cite[Remark 3.15]{TTT24}}]
\label{prop:card-R-vine} 
If $\mathcal{V}$ is a regular vine on a set $A$ with $|A|=n$, then 
$|\mathcal{V}(i)| = n+1-i$ for each $1 \le i \le n$. 
In particular, $A$ is the maximal element of $\mathcal{V}$ and 
$|\mathcal{V}| = n(n+1)/2$. 
\end{proposition}

\begin{proposition}[{\cite[Lemma 5.1]{TTT24}}]
\label{prop:atom}
Every element of rank at least $2$ in a (locally) regular vine is the join of a unique pair of minimal elements.
\end{proposition}

The following are two structural properties of regular vines closely paralleling the splitting and merging properties of MAT-labeled complete graphs.

\begin{lemma}[Splitting regular vines {\cite[Proposition 4.4]{TTT24}}]
\label{lem:splitting-vine}
Let $\mathcal{V}$ be a regular vine on a finite set $A$.
Assume that the maximal element $A$ covers two nodes
$A\setminus\{a_{1}\}$ and $A\setminus\{a_{2}\}$ in $\mathcal{V}$ for distinct
$a_{1},a_{2} \in A$.
Let $\mathcal{V}_i$ be the principal ideal of $\mathcal{V}$ generated by
$A\setminus\{a_i\}$ for $i \in \{1,2\}$, that is, 
\begin{align*}
\mathcal{V}_{i} \coloneqq \Set{S \in \mathcal{V} | S \subseteq A\setminus\{a_{i}\}}. 
\end{align*}
Then $\mathcal{V}_i$ and $\mathcal{V}' \coloneqq \mathcal{V}_1 \cap \mathcal{V}_2$ are regular vines on
$A\setminus\{a_i\}$ and $A\setminus\{a_{1},a_{2}\}$, respectively.
\end{lemma}

\begin{lemma}[Merging regular vines {\cite{CKW15}, \cite[Theorem A.1]{ZK22}, \cite{TTT24}}]
\label{lem:merging-vine}
Let $A$ be a finite set and let $a_1,a_2\in A$ be distinct elements.
For each $i \in \{1,2\}$, let $\mathcal{V}_i$ be a regular vine on
$A\setminus\{a_i\}$.
Assume that $\mathcal{V}' \coloneqq \mathcal{V}_1 \cap \mathcal{V}_2$ is a regular vine on
$A\setminus\{a_1,a_2\}$.
Then $\mathcal{V}_1 \cup \mathcal{V}_2 \cup \{A\}$ is a regular vine on $A$.
\end{lemma}

\begin{example}
See Figure \ref{Fig:ex_sp_mg_vine} for an example of splitting and merging of regular vines. 
The regular vine $\mathcal{V}$ splits into $\mathcal{V}_{1}, \mathcal{V}_{2}$, and $\mathcal{V}^{\prime}$. 
Conversely, $\mathcal{V}$ can be recovered by merging $\mathcal{V}_{1}$ and $\mathcal{V}_{2}$ along $\mathcal{V}^{\prime}$. 
\begin{figure}[htbp]
\centering
\begin{tikzpicture}[baseline=0, scale=0.7, transform shape]
\draw (0, 0) node[c](0){$a$};
\draw (1, 0) node[c](1){$b$};
\draw (2, 0) node[c](2){$c$};
\draw (3, 0) node[c](3){$d$};
\draw (4, 0) node[c](4){$e$};
\draw (0.5, 0.865) node[v](01){};
\draw (0)--(01);
\draw (1)--(01);
\draw (1.5, 0.865) node[v](12){};
\draw (1)--(12);
\draw (2)--(12);
\draw (2.5, 0.865) node[v](23){};
\draw (2)--(23);
\draw (3)--(23);
\draw (3.5, 0.865) node[v](24){};
\draw (2)--(24);
\draw (4)--(24);
\draw (1.0, 1.730) node[v](012){};
\draw (01)--(012);
\draw (12)--(012);
\draw (2.0, 1.730) node[v](123){};
\draw (12)--(123);
\draw (23)--(123);
\draw (3.0, 1.730) node[v](234){};
\draw (23)--(234);
\draw (24)--(234);
\draw (1.5, 2.595) node[v](0123){};
\draw (012)--(0123);
\draw (123)--(0123);
\draw (2.5, 2.595) node[v](1234){};
\draw (123)--(1234);
\draw (234)--(1234);
\draw (2.0, 3.460) node[v](01234){};
\draw (0123)--(01234);
\draw (1234)--(01234);
\draw (2,-1) node{\scalebox{1.428}{$\mathcal{V}$}};
\end{tikzpicture}
\qquad
\begin{tikzpicture}[baseline=0, scale=0.7, transform shape]
\draw (0, 0) node[c](0){$a$};
\draw (1, 0) node[c](1){$b$};
\draw (2, 0) node[c](2){$c$};
\draw (3, 0) node[c](3){$d$};
\draw (0.5, 0.865) node[v](01){};
\draw (0)--(01);
\draw (1)--(01);
\draw (1.5, 0.865) node[v](12){};
\draw (1)--(12);
\draw (2)--(12);
\draw (2.5, 0.865) node[v](23){};
\draw (2)--(23);
\draw (3)--(23);
\draw (1.0, 1.730) node[v](012){};
\draw (01)--(012);
\draw (12)--(012);
\draw (2.0, 1.730) node[v](123){};
\draw (12)--(123);
\draw (23)--(123);
\draw (1.5, 2.595) node[v](0123){};
\draw (012)--(0123);
\draw (123)--(0123);

\draw (1.5,-1) node{\scalebox{1.428}{$\mathcal{V}_{1}$}};
\end{tikzpicture}
\qquad
\begin{tikzpicture}[baseline=0, scale=0.7, transform shape]
\draw (1, 0) node[c](1){$b$};
\draw (2, 0) node[c](2){$c$};
\draw (3, 0) node[c](3){$d$};
\draw (4, 0) node[c](4){$e$};
\draw (1.5, 0.865) node[v](12){};
\draw (1)--(12);
\draw (2)--(12);
\draw (2.5, 0.865) node[v](23){};
\draw (2)--(23);
\draw (3)--(23);
\draw (3.5, 0.865) node[v](24){};
\draw (2)--(24);
\draw (4)--(24);
\draw (2.0, 1.730) node[v](123){};
\draw (12)--(123);
\draw (23)--(123);
\draw (3.0, 1.730) node[v](234){};
\draw (23)--(234);
\draw (24)--(234);

\draw (2.5, 2.595) node[v](1234){};
\draw (123)--(1234);
\draw (234)--(1234);

\draw (2.5,-1) node{\scalebox{1.428}{$\mathcal{V}_{2}$}};
\end{tikzpicture}
\qquad
\begin{tikzpicture}[baseline=0, scale=0.7, transform shape]

\draw (1, 0) node[c](1){$b$};
\draw (2, 0) node[c](2){$c$};
\draw (3, 0) node[c](3){$d$};

\draw (1.5, 0.865) node[v](12){};
\draw (1)--(12);
\draw (2)--(12);
\draw (2.5, 0.865) node[v](23){};
\draw (2)--(23);
\draw (3)--(23);

\draw (2.0, 1.730) node[v](123){};
\draw (12)--(123);
\draw (23)--(123);

\draw (2,-1) node{\scalebox{1.428}{$\mathcal{V}^{\prime}$}};
\end{tikzpicture}
\caption{Example of splitting and merging of regular vines}
\label{Fig:ex_sp_mg_vine}
\end{figure}
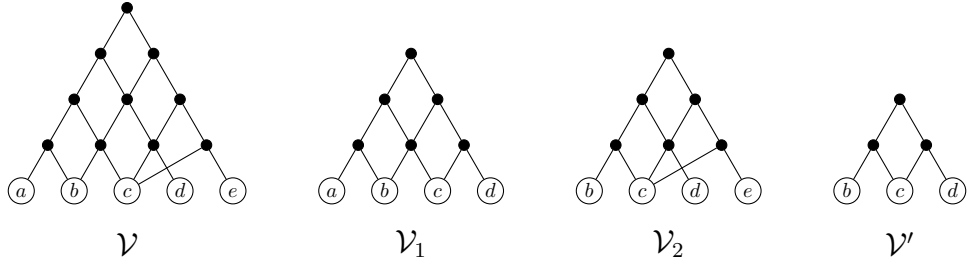
\end{example}

When we consider merging regular vines, the following observation will be useful.

\begin{lemma}\label{lem:v'=v1capv2}
Let $A$ be a finite set with $|A| \geq 3$ and let $a_{1}, a_{2}$ be distinct elements in $A$. 
Assume that $\mathcal{V}_{i}$ is a regular vine on $A\setminus\{a_{i}\}$ for $i \in \{1,2\}$ and $\mathcal{V}^{\prime}$ is a regular vine on $A\setminus\{a_{1},a_{2}\}$ such that $\mathcal{V}^{\prime} \subseteq \mathcal{V}_{1} \cap \mathcal{V}_{2}$. 
Then $\mathcal{V}^{\prime} = \mathcal{V}_{1} \cap \mathcal{V}_{2}$. 
\end{lemma}
\begin{proof}
Let $n \coloneqq |A|$. 
For $i \in \{1,2\}$, by Proposition \ref{prop:card-R-vine}, the top element of $\mathcal{V}_{i}$ is $A\setminus\{i\}$. 
Hence, for each rank, $\mathcal{V}_{i}$ has at least one element containing $a_{3-i}$. 
Therefore, for each rank $j$, 
\begin{align*}
|\mathcal{V}_{1}(j) \cap \mathcal{V}_{2}(j)| 
\leq |\mathcal{V}_{1}(j)|-1 
= n-j 
= |\mathcal{V}^{\prime}(j)|. 
\end{align*}
Therefore, $\mathcal{V}^{\prime} = \mathcal{V}_{1} \cap \mathcal{V}_{2}$. 
\end{proof}

Now we define two important families of regular vines.

\begin{definition}[D-vines and C-vines]
\label{def:DC-vine}
A regular vine is called a \tbf{D-vine} (resp.~\tbf{C-vine}) if each associated tree is a path (resp.~star) graph.
\end{definition}

D-vines and C-vines can be regarded as the extreme cases of regular vines. 

\begin{remark}
\label{rem:Dvine-typeA}
Since the line graph of a path graph is again a path graph of length smaller by one, a D-vine is unique up to isomorphism. 
Hence, after a suitable relabeling, a D-vine $\mathcal{V}$ is given by 
\begin{align*}
\mathcal{V}(k) = \Set{\{a_{i}, a_{i+1}, \dots, a_{i+k-1}\} \mid 1 \leq i \leq n-k+1}
\end{align*}
for each $k \in \{1, \dots, n\}$. 
In particular, a D-vine is isomorphic to the root poset of a type~$A$ root system (see \cite[Remark~4.16]{TTT24}). 
Similarly, a C-vine is also unique up to isomorphism, by the symmetry of star graphs. 
\end{remark}

\begin{proposition}
\label{prop:splitting-Dvine}
If $\mathcal{V}$ is a D-vine (resp. C-vine), then the vines $\mathcal{V}_{1}, \mathcal{V}_{2}$, and $\mathcal{V}'$ from Lemma \ref{lem:splitting-vine} are D-vines (resp. C-vines).
\end{proposition}
\begin{proof}
First, suppose that $\mathcal{V}$ is a D-vine. 
Then its associated trees are paths. 
Therefore, the associated trees of $\mathcal{V}_{1}, \mathcal{V}_{2}$, and $\mathcal{V}'$ are paths. 
Hence, they are D-vines. 

Next, suppose that $\mathcal{V}$ is a C-vine. 
Then each associated tree is a star. 
Therefore, at each rank of $\mathcal{V}$, there exists a unique element that is covered by all elements of rank one higher. 
Hence, $\mathcal{V}_{1}, \mathcal{V}_{2}$, and $\mathcal{V}'$ contains the center of star at each rank. 
Thus, they are C-vines. 
\end{proof}

D-vines can also be characterized by the existence of certain maximal chains.

\begin{proposition}
\label{prop:Dvine-chain}
Suppose that $n = |A| \geq 2$. 
A regular vine $\mathcal{V}$ on $A$ is a D-vine if and only if $\mathcal{V}$ has two maximal chains 
\[
\{c_1\} \subseteq \{c_1,c_2\} \subseteq \cdots \subseteq \{c_{1}, \dots, c_{n}\}
\quad\text{and}\quad
\{c_n\} \subseteq \{c_{n-1},c_n\} \subseteq \cdots \subseteq \{c_{1}, \dots, c_{n}\}
\]
after a suitable relabeling of the elements of $A$.
\end{proposition}

\begin{proof}
By Remark~\ref{rem:Dvine-typeA}, every D-vine has such a pair of maximal chains. 
We prove the converse.

Suppose that a regular vine $\mathcal{V}$ has two maximal chains 
\begin{align*}
\{c_1\} \subseteq \{c_1,c_2\} \subseteq \cdots \subseteq \{c_{1}, \dots, c_{n}\} = A, \\
\{c_n\} \subseteq \{c_{n-1},c_n\} \subseteq \cdots \subseteq \{c_{1}, \dots, c_{n}\} = A. 
\end{align*}

Let $a_{1} \coloneqq c_{1}$ and $a_{2} \coloneqq c_{n}$. 
Then $A$ covers both $A\setminus\{a_{1}\}$ and $A\setminus\{a_{2}\}$. 
Let $\mathcal{V}_{1}$ be the ideal generated by $A\setminus\{a_{1}\}$, and let 
\[
\mathcal{C} \coloneqq \mathcal{V}\setminus\mathcal{V}_{1}. 
\]
By Proposition~\ref{prop:card-R-vine}, the set $\mathcal{C}$ contains exactly one element of each rank. 
Since every non-minimal element of $\mathcal{C}$ contains $a_{1}$ and covers exactly two elements in $\mathcal{V}$, it follows that $\mathcal{C}$ forms a maximal chain of $\mathcal{V}$ with minimal element $\{a_{1}\}$. 

We show by induction on $n$ that $\mathcal{V}$ is a D-vine and that every element of $\mathcal{C}$ is a leaf in the corresponding associated tree. 
The case $n=2$ is immediate, so assume that $n \geq 3$. 

Since every element of $\mathcal{C}$ contains $a_{1}=c_{1}$, the chain $\mathcal{C}$ coincides with 
\[
\{c_1\} \subseteq \{c_1,c_2\} \subseteq \cdots \subseteq \{c_{1}, \dots, c_{n}\} = A. 
\]
Moreover, because every non-minimal element of $\mathcal{C}$ covers exactly two elements, we obtain a maximal chain
\begin{align*}
\{c_{2}\} \subseteq \{c_{2}, c_{3}\} \subseteq \dots \subseteq \{c_{2}, \dots, c_{n}\} = A\setminus\{a_{1}\}
\end{align*}
in $\mathcal{V}_{1}$. 

Therefore, by the induction hypothesis, $\mathcal{V}_{1}$ is a D-vine, and all elements in the above chain are leaves in the associated paths of $\mathcal{V}_{1}$. 
Each element of $\mathcal{C}$ is adjacent to an element of this chain in the associated tree of $\mathcal{V}$. 
Hence, $\mathcal{V}$ is the desired D-vine.
\end{proof}

\subsection{Maximal Arrow's single-peaked domains}
\label{sec:MASPD}

We begin this section by recalling the definitions of Condorcet and single-peaked domains together with some of their basic properties. 
Our presentation is slightly different from, but equivalent to, the one given in the introduction (see, e.g., \cite{M09}).

Let \(A\) be a finite set of \(n\) elements, called \textbf{alternatives} (or \textbf{candidates}), and let \(\mathcal{L}(A)\) denote the set of all bijections from \([n] \coloneqq \{1, \dots, n\}\) onto \(A\).  
Each \(\omega \in \mathcal{L}(A)\) naturally induces a linear order \(>_{\omega}\) on \(A\), defined by
\[
\omega(1) >_{\omega} \omega(2) >_{\omega} \cdots >_{\omega} \omega(n).
\]
An element of \(\mathcal{L}(A)\) is called a \textbf{preference}.  
We will write a preference \(\omega \in \mathcal{L}(A)\) by juxtaposition
\[
\omega = \omega(1)\omega(2)\cdots\omega(n). 
\]
In particular, \(\omega(1)\) and \(\omega(n)\) represent the first-ranked and last-ranked alternatives, respectively.  

A subset \(\mathcal{D} \subseteq \mathcal{L}(A)\) is called a \textbf{domain} of preferences.

\begin{definition}[Condorcet domains]
Let $\mathcal{D} \subseteq \mathcal{L}(A)$ be a domain. 
A triple of preferences $\omega_1,\omega_2,\omega_3 \in \mathcal{D}$  is called a \textbf{Condorcet cycle} if there exists a triple $T=\{a,b,c\}\subseteq A$ of alternatives such that
\[
a >_{\omega_1} b >_{\omega_1} c, \qquad
b >_{\omega_2} c >_{\omega_2} a, \qquad
c >_{\omega_3} a >_{\omega_3} b .
\]
A domain $\mathcal{D}$ is called a \textbf{Condorcet domain} if it contains no Condorcet cycles.
\end{definition}

Characterizing Condorcet domains is a challenging problem and remains open in general. 
Several important subclasses of Condorcet domains are known.

\begin{definition}[{\cite{B48}}]
A domain $\mathcal{D} \subseteq \mathcal{L}(A)$ is called a \textbf{Black's single-peaked domain (BSPD)} if there exists a path graph $P$ with vertex set $A$ such that for every   $\omega \in \mathcal{D}$ and every pair of distinct  $a,b \in A$ we have
\[
a >_{\omega} b
\]
whenever $a$ lies on the unique path in $P$ from $\omega(1)$ to $b$.
\end{definition}

\begin{remark}
\label{rem:BSPD}
The path $P$ is often referred to as the \emph{societal axis} (or \emph{political line}) on $A$. 
Intuitively, one may visualize $P$ as a horizontal axis ordering the alternatives. 
For a given preference $\omega$, the vertical axis represents the ranking position of each alternative. 
The graph of $\omega$ then has a single-peaked shape: starting from the most preferred alternative (the peak), the preference strictly decreases as one moves away from the peak along the axis in either direction (see Figure~\ref{fig:BSPD} for an example). 
\end{remark}

\begin{definition}
Let $S \subseteq A$. 
Each $\omega \in \mathcal{L}(A)$ induces a natural linear order $\omega_S$ on $S$: if $a, b \in S$, then $a >_\omega b$ if and only if $a >_{\omega_S} b$.  
The \textbf{restriction} of the domain $\mathcal{D}$ to $S$ is defined by
\[
\mathcal{D}_S \coloneqq 
\Set{\omega_S \in \mathcal{L}(S) | \omega \in \mathcal{D}} .
\]
\end{definition}

\begin{definition}[{\cite{A63}}]
A domain $\mathcal{D} \subseteq \mathcal{L}(A)$ is called an \textbf{Arrow's single-peaked domain (ASPD)} if for every triple $T\subseteq A$, the restriction $\mathcal{D}_T$ is a BSPD.
\end{definition}

An ASPD is sometimes referred to as a \emph{locally} BSPD.
An alternative is called a \textbf{bottom alternative} of a domain if it is ranked last in at least one preference in the domain.
The following proposition gives a convenient characterization of ASPDs.

\begin{proposition}[Never-bottom condition {\cite[Proposition 3.5]{Slinko19}}]
\label{prop:nbc}
A domain $\mathcal{D} \subseteq \mathcal{L}(A)$ is an ASPD if and only if for every triple $T \subseteq A$ there exists $x \in T$ such that $x$ is not a bottom alternative in the restriction $\mathcal{D}_T$.
\end{proposition}

\begin{theorem}[{\cite{B48,A63,I64}}]
The following strict inclusions hold:
\[
\{\text{BSPDs}\} \subsetneq
\{\text{ASPDs}\}
\subsetneq
\{\text{Condorcet domains}\}.
\]
\end{theorem}

\begin{example}
\label{ex:rk3}
Let $A=\{a,b,c\}$.
The domain
\[
\mathcal{D}_1 = \{abc,\, bca,\, cab\}
\]
is not a Condorcet domain since $\mathcal{D}_1$ itself forms a Condorcet cycle.
The domain
\[
\mathcal{D}_2 = \{abc, acb, cab, cba\}
\]
is a Condorcet domain since no Condorcet cycle exists in $\mathcal{D}_2$. 
However, $\mathcal{D}_2$ is not an ASPD (and hence not a BSPD) because every alternative appears as the bottom alternative in at least one preference (see Proposition \ref{prop:nbc}).
The domain
\[
\mathcal{D}_3 = \{abc,bac,bca,cba\}
\]
is a BSPD (and hence an ASPD) with societal axis $P=(a,b,c)$ (see Figure~\ref{fig:BSPD}). 
An example of an ASPD that is not a BSPD is given in Example~\ref{ex:rk4}.

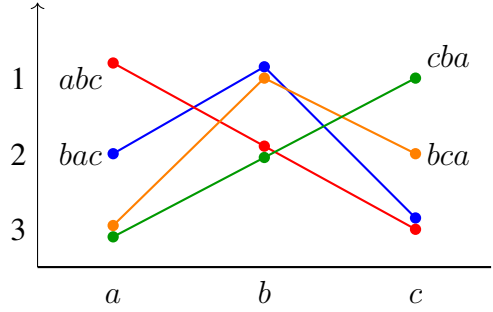
\begin{figure}[h]
\centering
\begin{tikzpicture}[scale=1]
[every node/.style={circle,fill,inner sep=1.5pt}]
% Societal axis
\draw[thick] (-1,0.5) -- (5,0.5);
\node at (0,0.5) [below=4pt] {$a$};
\node at (2,0.6) [below=4pt] {$b$};
\node at (4,0.5) [below=4pt] {$c$};

%\node at (2,-0.8) {Societal axis $P=(a,b,c)$};

% Vertical guide (ranking levels)
\draw[->] (-1,0.5) -- (-1,4);
\node[left] at (-1,3) {1};
\node[left] at (-1,2) {2};
\node[left] at (-1,1) {3};

% bac
\draw[thick,blue]
(0,2) node[circle,fill,inner sep=1.5pt]{} -- 
(2,3.15) node[circle,fill,inner sep=1.5pt]{} -- 
(4,1.15) node[circle,fill,inner sep=1.5pt]{};
\node[left] at (0,2) {$bac$};

% abc
\draw[thick,red]
(0,3.2) node[circle,fill,inner sep=1.5pt]{} -- 
(2,2.1) node[circle,fill,inner sep=1.5pt]{} -- 
(4,1) node[circle,fill,inner sep=1.5pt]{};
\node[left] at (0,3) {$abc$};

% bca
\draw[thick,orange]
(0,1.05) node[circle,fill,inner sep=1.5pt]{} -- 
(2,3) node[circle,fill,inner sep=1.5pt]{} -- 
(4,2) node[circle,fill,inner sep=1.5pt]{};
\node[right] at (4,2) {$bca$};

% cba
\draw[thick,green!60!black]
(0,0.9) node[circle,fill,inner sep=1.5pt]{} -- 
(2,1.95) node[circle,fill,inner sep=1.5pt]{} -- 
(4,3) node[circle,fill,inner sep=1.5pt]{};
\node[above right] at (4,3) {$cba$};

\end{tikzpicture}
\caption{The domain $\mathcal{D}_3=\{abc,bac,bca,cba\}$ as a BSPD with societal axis $P=(a,b,c)$. Each preference is single-peaked along the axis.}
\label{fig:BSPD}
\end{figure}

\end{example}

It is easy to see that every subset of an ASPD (resp.~ BSPD, Condorcet domain) is also an ASPD (resp.~ BSPD, Condorcet domain). 
Thus it is natural to ask when such domains are \emph{maximal}. 
From a voting-theoretic perspective, larger domains allow voters greater freedom in expressing their preferences.

\begin{definition}
An ASPD $\mathcal{D}$ is \textbf{maximal} if
$
\mathcal{D}\cup\{\omega\}
$
is not an ASPD for any $\omega \in \mathcal{L}(A)\setminus \mathcal{D}$.
Maximal BSPDs and maximal Condorcet domains are defined analogously.
\end{definition}

Maximal BSPDs can be characterized as follows.

\begin{proposition}[{\cite[Theorem~1]{Puppe18}, \cite[Corollary~2]{Slinko19}}]
\label{prop:BSPD}
A domain \(\mathcal{D}\) is a maximal BSPD if and only if it is a maximal ASPD and contains two preferences \(\omega, \omega'\) such that \(\omega' = \omega^{\rev}\), where $\omega^{\rev} \coloneqq \omega(n)\omega(n-1)\cdots\omega(1)$ is the reversal of $\omega$. 
\end{proposition}

\begin{example}[{\cite[Theorem 3]{Slinko19}}]
\label{ex:rk4}
Up to isomorphism, $\mathcal{D}_3=\{abc,bac,bca,cba\}$ from Example~\ref{ex:rk3} is the unique maximal ASPD (and also maximal BSPD) on $3$ alternatives. 
There are exactly two maximal ASPDs on $4$ alternatives; see Figure \ref{fig:ASPD-rk4}. 
\end{example}

 \begin{figure}[htbp!]
\begin{subfigure}{.4\textwidth}
   \centering
$$\mathcal{D}_{4,1} =
\begin{array}{|cccc|cccc|}
\hline
a &  b & b & c & b &  c & c & d \\
b &  a & c & b & c &  b & d & c \\
c &  c & a & a & d &  d & b & b \\
d & d & d & d & a & a & a & a \\
\hline
\end{array}$$
\end{subfigure}%
\begin{subfigure}{.4\textwidth}
  \centering
$$\mathcal{D}_{4,2} =
\begin{array}{|cccc|cccc|}
\hline
a & c  & b & a & b &  a & a & d \\
c & a & a & b & a &  b & d & a \\
b & b & c & c & d &  d & b & b \\
d & d & d & d & c & c & c & c \\
\hline
\end{array}$$
\end{subfigure}
\caption{Maximal ASPDs on $4$ alternatives. 
The domain $\mathcal{D}_{4,1}$ is a maximal BSPD, whereas $\mathcal{D}_{4,2}$ is not a BSPD.}
\label{fig:ASPD-rk4}
\end{figure}

The following lemma shows that maximal ASPDs admit a recursive structure.

\begin{lemma}[Splitting maximal ASPDs {\cite{Slinko19}}]\label{Slinko}
Let $\mathcal{D} \subseteq \mathcal{L}(A)$ be a maximal ASPD on a set $A$ with $n$ alternatives. 
Then the following hold:
\begin{enumerate}[(1)]
\item $|\mathcal{D}| = 2^{n-1}$. 

\item If $n \ge 2$, then $\mathcal{D}$ has exactly two bottom alternatives. 

\item Assume $n \ge 3$ and let $a_{1},a_{2} \in A$ be the bottom alternatives of $\mathcal{D}$. 
For $i,j \in \{1,2\}$ with $i\ne j$, define
\begin{align*}
\mathcal{D}_{i} &\coloneqq \Set{\omega_{A\setminus \{a_{i}\}} \in \mathcal{L}(A\setminus \{a_{i}\}) | \omega \in \mathcal{D},\, \omega(n) = a_{i}}, \\
\mathcal{D}_{ij} &\coloneqq \Set{\omega_{A\setminus\{a_{i},a_{j}\}} \in \mathcal{L}(A\setminus\{a_{i},a_{j}\}) | \omega \in \mathcal{D}_{i},\, \omega(n-1) = a_{j}}.
\end{align*}
We represent $\mathcal{D}$ by the following diagrams:
\begin{align*}
\mathcal{D} = 
\begin{tikzpicture}[baseline=25]
\draw (0,0) rectangle (6,2);
\draw (0,.5) -- (6,.5);
\draw (3,0) -- (3,2);
\draw (1.5,0.2) node{$a_{1}$};
\draw (4.5,0.2) node{$a_{2}$};
\draw (1.5, 1.2) node{$\mathcal{D}_{1}$};
\draw (4.5, 1.2) node{$\mathcal{D}_{2}$};
\end{tikzpicture}
=
\begin{tikzpicture}[baseline=25]
\draw (0,0) rectangle (6,2);
\draw (0,.5) -- (6,.5);
\draw (3,0) -- (3,2);
\draw (1.5,0.2) node{$a_{1}$};
\draw (4.5,0.2) node{$a_{2}$};
\draw (1.5,.5) -- (1.5,2);
\draw (4.5,.5) -- (4.5,2);
\draw (0,1) -- (6,1);
\draw (2.25,.7) node{$a_{2}$};
\draw (3.75,.7) node{$a_{1}$};
\draw (2.25, 1.5) node{$\mathcal{D}_{12}$};
\draw (3.75, 1.5) node{$\mathcal{D}_{21}$};
\end{tikzpicture}.
\end{align*}
Then $\mathcal{D}_{i}$ and $\mathcal{D}_{ij}$ are all maximal ASPDs. 
Furthermore, $\mathcal{D}_{12} = \mathcal{D}_{21}$, and we will denote this common domain by 
$
\mathcal{D}^{\prime}\coloneqq\mathcal{D}_{12}=\mathcal{D}_{21}.
$

\end{enumerate}
\end{lemma}

\begin{corollary}
\label{cor:splitting-BSPD}
If \(\mathcal{D}\) is a maximal BSPD, then the domains \(\mathcal{D}_{1}, \mathcal{D}_{2}\) and \(\mathcal{D}'\) from Lemma~\ref{Slinko} are also maximal BSPDs.
\end{corollary}

\begin{proof}
Since $\mathcal{D}$ is a BSPD, there exists a path graph $P$ on $A$ such that every preference in $\mathcal{D}$ is single-peaked with respect to $P$. 
Because $a_{1}$ is a bottom element of $\mathcal{D}$, the vertex $a_{1}$ must be an endpoint of $P$. 
Let $P_{1} \coloneqq P \setminus a_{1}$. 
Then every preference in $\mathcal{D}_{1}$ is single-peaked with respect to $P_{1}$. 
Hence, $\mathcal{D}_{1}$ is a BSPD. 
By Lemma~\ref{Slinko}, $\mathcal{D}_{1}$ is a maximal ASPD, and therefore maximal as a BSPD. 
The same argument applies to $\mathcal{D}_{2}$ and $\mathcal{D}'$.
\end{proof}

We now establish another important structural property of maximal ASPDs.

\begin{lemma}[Merging maximal  ASPDs]\label{merging ASPD}
Let $a_{1},a_{2} \in A$ be distinct alternatives. 
For each $i \in \{1,2\}$, assume that $\mathcal{D}_{i}$ is a maximal ASPD on $A\setminus\{a_{i}\}$ and that $a_{3-i}$ is a bottom alternative of $\mathcal{D}_{i}$. 
We retain the notation of $D_{ij}$ from Lemma~\ref{Slinko}:
\[
\mathcal{D}_{ij} \coloneqq 
\Set{\omega_{A\setminus\{a_{i},a_{j}\}} \in \mathcal{L}(A\setminus\{a_{i},a_{j}\}) 
| \omega \in \mathcal{D}_{i},\, \omega(n-1) = a_{j}}.
\]
Suppose further that $\mathcal{D}_{12}=\mathcal{D}_{21}$, and denote this common domain by 
\[
\mathcal{D}^{\prime}\coloneqq\mathcal{D}_{12}=\mathcal{D}_{21}.
\]
We represent $\mathcal{D}_i \in \mathcal{L}(A\setminus \{a_{i}\})$ by the following diagrams:
\begin{align*}
\mathcal{D}_{1}
&=
\begin{tikzpicture}[baseline=30]
\draw (0,0.5) rectangle (3,2);
\draw (1.5,.5) -- (1.5,2);
\draw (0,1) -- (3,1);
\draw (0.75,.74) node{$b_{1}$};
\draw (2.25,.7) node{$a_{2}$};
\draw (0.75, 1.45) node{$\mathcal{D}_{1}^{\prime}$};
\draw (2.25, 1.5) node{$\mathcal{D}^{\prime}$};
\end{tikzpicture} \ ,
\qquad
\mathcal{D}_{2}
=
\begin{tikzpicture}[baseline=30]
\draw (0,0.5) rectangle (3,2);
\draw (1.5,.5) -- (1.5,2);
\draw (0,1) -- (3,1);
\draw (0.75,.7) node{$a_{1}$};
\draw (2.25,.74) node{$b_{2}$};
\draw (0.75, 1.5) node{$\mathcal{D}^{\prime}$};
\draw (2.25, 1.45) node{$\mathcal{D}_{2}^{\prime}$};
\end{tikzpicture} \ ,
\end{align*}
where $b_1,b_2 \in A\setminus\{a_{1},a_{2}\}$.
Let $T \subseteq A\setminus\{a_{1},a_{2}\}$ be a triple and let $x \in T$. 
Then the following hold:
\begin{enumerate}[(1)]

\item\label{merging ASPD 1}
$x$ is not bottom in $(\mathcal{D}_{1})_{T}$ if and only if $x$ is not bottom in $(\mathcal{D}_{2})_{T}$.

\item\label{merging ASPD 2}
The domain $\mathcal{D}$ on $A$ defined by
\[
\mathcal{D} \coloneqq
\begin{tikzpicture}[baseline=25]
\draw (0,0) rectangle (6,2);
\draw (0,.5) -- (6,.5);
\draw (3,0) -- (3,2);
\draw (1.5,0.2) node{$a_{1}$};
\draw (4.5,0.2) node{$a_{2}$};
\draw (1.5, 1.2) node{$\mathcal{D}_{1}$};
\draw (4.5, 1.2) node{$\mathcal{D}_{2}$};
\end{tikzpicture}
\]
is a maximal ASPD.
\end{enumerate}
\end{lemma}

\begin{proof}
First we prove part~(\ref{merging ASPD 1}) by induction on the number $|A|$ of alternatives. 
By symmetry, it suffices to prove the forward implication.
If $|A| \le 4$, the statement is immediate. 

Suppose $|A| \ge 5$ and let $T \subseteq A\setminus\{a_{1},a_{2}\}$ be a triple. 
Assume that $x \in T$ is not bottom in $(\mathcal{D}_{1})_{T}$. 
Hence $x$ is also not bottom in $(\mathcal{D}^{\prime})_{T}$.
Since $\mathcal{D}_{2}$ is a maximal ASPD, Lemma~\ref{Slinko}(3) implies that $b_{2}$ is a bottom alternative of $\mathcal{D}^{\prime}$. 
Hence $x \neq b_{2}$, because otherwise $x$ would be bottom in $(\mathcal{D}^{\prime})_{T}$.

If $b_{2} \in T$, then $x$ is clearly not bottom in $(\mathcal{D}_{2})_{T}$ since $x \neq b_{2}$. 
Assume now that $b_{2} \notin T$. 
Then necessarily $|A| \ge 6$. 
Again using Lemma~\ref{Slinko}(3), the domains $\mathcal{D}^{\prime}$ and $\mathcal{D}_{2}^{\prime}$ satisfy the same structural conditions as in the statement of the lemma. 
By the induction hypothesis, $x$ is not bottom in $(\mathcal{D}_{2}^{\prime})_{T}$, and therefore also  not bottom in $(\mathcal{D}_{2})_{T}$.

We now prove part~(\ref{merging ASPD 2}). 
Let $T \subseteq A$ be a triple.
First consider the case $a_{1},a_{2} \in T$. 
Then the remaining element $x \in T\setminus\{a_{1},a_{2}\}$ is not bottom in $\mathcal{D}_{T}$, since either $a_{1}$ or $a_{2}$ must occupy the bottom position.
Next suppose that $a_{1},a_{2} \notin T$. 
Since $\mathcal{D}_{1}$ is a maximal ASPD, there exists $x \in T$ that is not bottom in $(\mathcal{D}_{1})_{T}$. 
By part~(\ref{merging ASPD 1}), the same element $x$ is not bottom in $(\mathcal{D}_{2})_{T}$, and hence not bottom in $\mathcal{D}_{T}$.
Finally, suppose that exactly one of $a_{1},a_{2}$ belongs to $T$. 
Without loss of generality assume $a_{1} \in T$ and $a_{2} \notin T$. 
Since $\mathcal{D}_{2}$ is a maximal ASPD, there exists $x \in T$ that is not bottom in $(\mathcal{D}_{2})_{T}$. 
Because $a_{1}$ is a bottom alternative of $\mathcal{D}_{2}$, we must have $x \neq a_{1}$. 
Hence $x$ is not bottom in $\mathcal{D}_{T}$.

Therefore $\mathcal{D}$ is an ASPD. 
Moreover,
\[
|\mathcal{D}| = |\mathcal{D}_{1}| + |\mathcal{D}_{2}|
= 2^{n-2} + 2^{n-2}
= 2^{n-1}.
\]
Thus $\mathcal{D}$ is maximal by Lemma~\ref{Slinko}(1).
\end{proof}

\begin{example}
The maximal ASPD 
\begin{align*}
\mathcal{D} = 
\begin{array}{|cccccccc|cccccccc|}
\hline
a&b&b& \multicolumn{1}{c|}{c} & b&c&c&d & b&c&c& \multicolumn{1}{c|}{d} & c&d&c&e \\
b&a&c& \multicolumn{1}{c|}{b} & c&b&d&c & c&b&d& \multicolumn{1}{c|}{c} & d&c&e&c \\
c&c&a& \multicolumn{1}{c|}{a} & d&d&b&b & d&d&b& \multicolumn{1}{c|}{b} & e&e&d&d \\ \hline
d&d&d& \multicolumn{1}{c|}{d} & a&a&a&a & e&e&e& \multicolumn{1}{c|}{e} & b&b&b&b \\ \hline
e&e&e&e & e&e&e&e & a&a&a&a & a&a&a&a \\ \hline
\end{array}
\end{align*}
splits into $\mathcal{D}_{1}, \mathcal{D}_{2}$, and $\mathcal{D}^{\prime}$, where 
\begin{align*}
\mathcal{D}_{1} = \begin{array}{|cccc|cccc|}
\hline
a&b&b&c & b&c&c&d \\
b&a&c&b & c&b&d&c \\
c&c&a&a & d&d&b&b \\ \hline
d&d&d&d & a&a&a&a \\ \hline
\end{array} \ , 
\qquad
\mathcal{D}_{2} = \begin{array}{|cccc|cccc|}
\hline
b&c&c&d & c&d&c&e \\
c&b&d&c & d&c&e&c \\
d&d&b&b & e&e&d&d \\ \hline
e&e&e&e & b&b&b&b \\ \hline
\end{array} \ , 
\end{align*}
and 
\begin{align*}
\mathcal{D}^{\prime} = \begin{array}{|cccc|}
\hline
b&c&c&d \\
c&b&d&c \\
d&d&b&b \\ \hline
\end{array} \ .
\end{align*}

Conversely, $\mathcal{D}$ can be recovered by merging $\mathcal{D}_{1}$ and $\mathcal{D}_{2}$ along $\mathcal{D}^{\prime}$. 
\end{example}

   %********************************************************************************************************

  %********************************************************************************************************

%********************************************************************************************************
\section{Axiomatic framework from splitting and merging of species}
\label{subsec:GV}

We briefly recall the notion of a combinatorial species; see \cite{BLL98} for background. 
A \textbf{(combinatorial) species} is a functor from the category of finite sets and bijections to itself. 
More precisely, a species $\mathsf{F}$ assigns to each finite set $A$ (the set of labels) a finite set $\mathsf{F}[A]$ (the $\mathsf{F}$-structures on $A$), and to each bijection $h\colon A\longrightarrow B$ a transport map (relabeling)
\[
\mathsf{F}[h]\colon \mathsf{F}[A]\longrightarrow\mathsf{F}[B].
\]
These maps satisfy the functoriality conditions
\[
\mathsf{F}[\mathrm{id}_A]=\mathrm{id}_{\mathsf{F}[A]}, 
\qquad
\mathsf{F}[g\circ h]=\mathsf{F}[g]\circ\mathsf{F}[h]
\]
for all bijections $h \colon A\longrightarrow B$ and $g\colon B\longrightarrow C$.

A \textbf{natural transformation} $\eta \colon \mathsf{F}\longrightarrow\mathsf{G}$ between species consists of component maps
$\eta_A\colon \mathsf{F}[A]\longrightarrow\mathsf{G}[A]$ for all finite sets $A$ satisfying the naturality condition. 
That is, for every bijection $h\colon A\longrightarrow B$ the diagram
\[
\begin{tikzcd}
\mathsf{F}[A] \ar[r,"{\mathsf{F}[h]}"] \ar[d,"\eta_A"'] &
\mathsf{F}[B] \ar[d,"\eta_B"]\\
\mathsf{G}[A] \ar[r,"{\mathsf{G}[h]}"] &
\mathsf{G}[B]
\end{tikzcd}
\]
commutes.  
A natural transformation whose components are bijections is called a \textbf{natural isomorphism}.

As we saw in Section~\ref{sec:Preliminaries}, MAT-labeled complete graphs, regular vines, and maximal ASPDs admit similar splitting and merging structures. 
In this section, we formulate these operations axiomatically in the language of combinatorial species.
We first fix some notation that will be used throughout this section. 
Let $A$ be a finite set. If $a_1,a_2 \in A$ are distinct elements, define
\[
A_i \coloneqq A\setminus\{a_i\} \quad (i=1,2), 
\qquad
A' \coloneqq A\setminus\{a_1,a_2\}.
\]

\begin{definition}
\label{def:KF} 
Throughout this section, let $\mathsf{F}$ be a species satisfying the following conditions:
\begin{enumerate}
\item $\mathsf{F}[A]=A$ whenever $|A|\le 1$;
\item $\mathsf{F}[A]\cap \mathsf{F}[B]=\varnothing$ whenever $A\neq B$.
\end{enumerate}
We define the species $\SpSplit^{\mathsf{F}}$ as follows. 

The set $\SpSplit^{\mathsf{F}}[A]$ is defined recursively as follows: 
\begin{enumerate}[(i)]
\item If $|A|\le 1$, define $\SpSplit^{\mathsf{F}}[A]\coloneqq A$.
\item If $|A| \ge2$, define
\[
\SpSplit^{\mathsf{F}}[A]
\coloneqq 
\Set{\{\ell_1,\ell_2\} | \ell_i\in \mathsf{F}[A_{i}] \text{ for } i \in \{1,2\} \text{, where }\ a_1, a_2 \in A \text{ and }a_{1} \neq a_{2}}. 
\]
Equivalently, $\SpSplit^{\mathsf{F}}[A]$ is the edge set of the complete multipartite graph whose parts are the sets $\mathsf{F}[A\setminus\{a\}]$ for $a\in A$.
\end{enumerate}

Let $h\colon A\longrightarrow B$ be a bijection. 
Define the transport map 
\[
\SpSplit^{\mathsf{F}}[h]\colon \SpSplit^{\mathsf{F}}[A]\longrightarrow\SpSplit^{\mathsf{F}}[B]
\]
recursively as follows:
\begin{enumerate}[(i)]
\item If $|A|\le 1$, define $\SpSplit^{\mathsf{F}}[h]\coloneqq h$.
\item If $|A| \ge2$, define
\[
\SpSplit^{\mathsf{F}}[h](\{\ell_1,\ell_2\})
\coloneqq
\{ \mathsf{F}[h|_{A_1}](\ell_1),
\mathsf{F}[h|_{A_2}](\ell_2)\}, 
\]
where $A_{i}$ denotes the set such that $\ell_{i} \in \mathsf{F}[A_{i}]$ for $i \in \{1,2\}$. 
\end{enumerate}
\end{definition}

We now introduce two properties of the species $\SpSplit^{\mathsf{F}}$, which will play a central role in our constructions.

\begin{definition}
\label{def:prox-merge} 
A natural transformation from the species $\mathsf{F}$ to $\SpSplit^{\mathsf{F}}$ is called a \textbf{splitting}. 

For a splitting $\sigma^{\mathsf{F}}\colon \mathsf{F}\longrightarrow\SpSplit^{\mathsf{F}}$ and a finite set $A$, let $\sigma^{\mathsf{F}}_A\colon \mathsf{F}[A]\longrightarrow\SpSplit^{\mathsf{F}}[A]$ denote its component at $A$. 
We define the following two properties.

\begin{enumerate}
 
\item (\textbf{Proximity})  
For a finite set $A$ with $|A|\ge2$, suppose
\[
\sigma^{\mathsf{F}}_{A}(\ell)=\{\ell_1,\ell_2\},
\]
where $\ell_i\in\mathsf{F}[A_i]$ for $i\in\{1,2\}$. 
Then
\[
\left| \sigma^{\mathsf{F}}_{A_1}(\ell_{1})
\symdiff 
\sigma^{\mathsf{F}}_{A_2}(\ell_{2}) \right| = 2.
\]
Here, $\symdiff$ denotes the symmetric difference.

\item  (\textbf{Mergeability}) 
For a finite set $A$ with $|A|\ge2$, suppose that $\{\ell_{1}, \ell_{2}\} \in \SpSplit^{\mathsf{F}}[A]$ satisfies 
\[
\left| \sigma^{\mathsf{F}}_{A_1}(\ell_{1})
\symdiff 
\sigma^{\mathsf{F}}_{A_2}(\ell_{2}) \right| = 2,
\]
where $\ell_i\in\mathsf{F}[A_i]$ for $i\in\{1,2\}$. 
Then there exists a unique $\mathsf{F}$-structure $\ell\in\mathsf{F}[A]$ such that
\[
\sigma^{\mathsf{F}}_{A}(\ell)=\{\ell_1,\ell_2\}.
\]
In this case, we call $\ell$ the \textbf{merging} of $\ell_1$ and $\ell_2$.
\end{enumerate}
\end{definition}

\begin{remark}
\label{rem:prime}
For distinct singletons $\{a\}$ and $\{b\}$, we have $\{a\} \symdiff \{b\} = \{a,b\}$. 
Therefore, if $|A| = 2$, then the condition 
$
\left| \sigma^{\mathsf{F}}_{A_1}(\ell_{1})
\symdiff 
\sigma^{\mathsf{F}}_{A_2}(\ell_{2}) \right| = 2
$
is automatically satisfied. 
If $|A| \geq 3$, then
$$
\left| \sigma^{\mathsf{F}}_{A_1}(\ell_{1})
\symdiff 
\sigma^{\mathsf{F}}_{A_2}(\ell_{2}) \right| = 2
$$
if and only if
$$
|\sigma^{\mathsf{F}}_{A_1}(\ell_{1}) \cap \sigma^{\mathsf{F}}_{A_2}(\ell_{2})| = 1. 
$$
In this case, the proximity and mergeability properties are illustrated in Figure~\ref{Fig:prox-merg}. 
For later use, we denote by $\ell^{\prime} \in \mathsf{F}[A']$ the $\mathsf{F}$-structure that satisfies 
\[
\sigma^{\mathsf{F}}_{A_1}(\ell_{1}) \cap \sigma^{\mathsf{F}}_{A_2}(\ell_{2})
=
\{\ell^{\prime}\}.
\]

\begin{figure}[h]
\centering
\begin{tikzpicture}[baseline=45]
\draw (1,1.73) node[v](31){} node[left]{$\ell_1$};
\draw (2,1.73) node[v](32){} node[right]{$\ell_2$};
\draw (1.5,2.595) node[v](41){} node[above]{$\ell$};
\draw[] (31) -- (41) -- (32);
\draw (31)--(0.75,1.3);
\draw (31)--(1.25,1.3);
\draw (32)--(1.75,1.3);
\draw (32)--(2.25,1.3);
\end{tikzpicture}
$\quad\Longrightarrow\quad$
\begin{tikzpicture}[baseline=45]
\draw (0.5, 0.865) node[v](21){};
\draw (1.5, 0.865) node[v](22){} node[below=1pt, xshift=-3.5pt]{$\exists \, \ell'$};
\draw (2.5, 0.865) node[v](23){};
\draw (1,1.73) node[v](31){} node[left]{$\ell_1$};
\draw (2,1.73) node[v](32){} node[right]{$\ell_2$};
\draw (1.5,2.595) node[v](41){} node[above]{$\ell$};
\draw (31) -- (21) ;
\draw (32) -- (23) ;
\draw[dashed] (31) -- (22) -- (32);
\draw[] (31) -- (41) -- (32);
\end{tikzpicture}
\hspace{20mm}
\begin{tikzpicture}[baseline=45]
\draw (0.5, 0.865) node[v](21){};
\draw (1.5, 0.865) node[v](22){} node[below]{$\ell'$};
\draw (2.5, 0.865) node[v](23){};
\draw (1,1.73) node[v](31){} node[left]{$\ell_1$};
\draw (2,1.73) node[v](32){} node[right]{$\ell_2$};
\draw (31) -- (21) ;
\draw (32) -- (23) ;
\draw[] (31) -- (22) -- (32);
\end{tikzpicture}
$\quad\Longrightarrow\quad$
\begin{tikzpicture}[baseline=45]
\draw (0.5, 0.865) node[v](21){};
\draw (1.5, 0.865) node[v](22){} node[below]{$\ell'$};
\draw (2.5, 0.865) node[v](23){};
\draw (1,1.73) node[v](31){} node[left]{$\ell_1$};
\draw (2,1.73) node[v](32){} node[right]{$\ell_2$};
\draw (1.5,2.595) node[v](41){} node[above, xshift=-6pt]{$\exists ! \, \ell$};
\draw (31) -- (21) ;
\draw (32) -- (23) ;
\draw[] (31) -- (22) -- (32);
\draw[dashed] (31) -- (41) -- (32);
\end{tikzpicture}
\caption{Illustration of the proximity property (left) and the mergeability property (right).}
\label{Fig:prox-merg}
\end{figure}
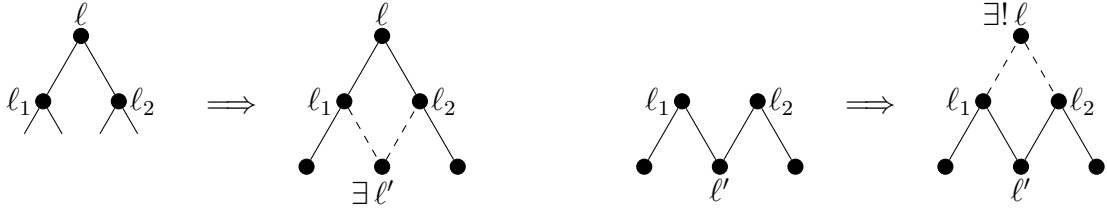
\end{remark}

\begin{example}
\label{ex:Gr}
Let $\SpMAT$ denote the \textbf{species of MAT-labeled complete graphs}. 
Namely, for a finite set $A$, $\SpMAT[A]$ consists of all MAT-labeled complete graphs $(K_{A}, \lambda)$. 
We often abbreviate $(K_{A}, \lambda)$ simply to $\lambda$. 

For a finite set $A$ with $|A| \geq 2$, let $\lambda\in\SpMAT[A]$ and let $a_1,a_2\in A$ be the MAT-simplicial vertices of $(K_A,\lambda)$.
Define the splitting $\sigma^{\SpMAT}$ by
\[
\sigma^{\SpMAT}_{A}(\lambda)
\coloneqq
\{\lambda_1,\lambda_2\}, 
\]
where $\lambda_i\coloneqq\lambda|_{E_{G_{i}}}$ and $G_{i} \coloneqq K_{A_i}$ for $i \in \{1,2\}$. 
Then $\sigma^{\SpMAT}$ satisfies the proximity (Lemma~\ref{lem:splitting-MAT}) and mergeability (Lemma~\ref{lem:merge-MAT}) properties. 
\end{example}

\begin{example}
\label{ex:Vi}
Let $\SpRV$ denote the \textbf{species of regular vines}. 
Namely, for a finite set $A$, $\SpRV[A]$ consists of all regular vines on $A$. 

For a finite set $A$ with $|A| \geq 2$, let $\mathcal{V}\in\SpRV[A]$. 
Assume that $A=\{a_1\}\vee \{a_2\}$ (the join of $\{a_1\}$ and $\{a_2\}$) in $\mathcal{V}$ for $a_1,a_2\in A$. 
Let $\mathcal{V}_i$ be the principal ideal of $\mathcal{V}$ generated by $A_i$ for $i \in \{1,2\}$.  
Define the splitting $\sigma^{\SpRV}$ by 
\[
\sigma^{\SpRV}_{A}(\mathcal{V})
\coloneqq
\{\mathcal{V}_1,\mathcal{V}_2\}.
\]
Then $\sigma^{\SpRV}$ satisfies the proximity (Lemma~\ref{lem:splitting-vine}) and mergeability (Lemma~\ref{lem:merging-vine}) properties.
\end{example}

\begin{example}
\label{ex:mASPD}
Let $\SpASPD$ denote the \textbf{species of maximal ASPDs}. 
Namely, for a finite set $A$, $\SpASPD[A]$ consists of all maximal ASPDs on $A$. 

For a finite set $A$ with $|A| \geq 2$, let $\mathcal{D} \in \SpASPD[A]$. 
Let $a_{1},a_{2} \in A$ be the bottom alternatives of $\mathcal{D}$, and let
$
\mathcal{D}_{i} \coloneqq \mathcal{D}_{A_i}
$
be the restriction of $\mathcal{D}$ to $A_i$ for $i \in \{1,2\}$. 
Define the splitting $\sigma^{\SpASPD}$ by 
\[
\sigma^{\SpASPD}_{A}(\mathcal{D}) \coloneqq \{\mathcal{D}_{1}, \mathcal{D}_{2}\}.
\]
Then $\sigma^{\SpASPD}$ satisfies the proximity (Lemma~\ref{Slinko}) and mergeability (Lemma~\ref{merging ASPD}) properties.
\end{example} 
 
The following theorem, which is the main result of this paper, provides an axiomatic characterization of species whose splittings satisfy the proximity and mergeability properties.

\begin{theorem}
\label{thm:AxCh}
Let $\mathsf{F}$ be a species satisfying the conditions in Definition~\ref{def:KF}, equipped with a splitting $\sigma^{\mathsf{F}}$ satisfying the proximity and mergeability properties in Definition~\ref{def:prox-merge}. 

Then the pair $(\mathsf{F}, \sigma^{\mathsf{F}})$ is unique in the following sense. 
If another pair $(\mathsf{G}, \sigma^{\mathsf{G}})$ satisfies the same conditions, then there exists a unique natural transformation $\eta \colon \mathsf{F} \to \mathsf{G}$ such that \textbf{$\eta$ commutes with the splittings}; that is, for every finite set $A$, the diagram
\begin{center}
\begin{tikzcd}
  \mathsf{F}[A] \ar[r,"\sigma^{\mathsf{F}}_{A}"] \ar[d,"\eta_{A}"'] &
  \SpSplit^{\mathsf{F}}[A] \ar[d,"\SpSplit^{\eta}_{A}"]\\
  \mathsf{G}[A] \ar[r,"\sigma^{\mathsf{G}}_{A}"] &
  \SpSplit^{\mathsf{G}}[A]
\end{tikzcd}
\end{center}
commutes. 

Moreover, $\eta$ is a natural isomorphism. 
In particular, $\mathsf{F}$ and $\mathsf{G}$ are isomorphic. 

Here, $\SpSplit^{\eta}$ is the natural transformation from $\SpSplit^{\mathsf{F}}$ to $\SpSplit^{\mathsf{G}}$ defined recursively as follows:
\begin{enumerate}[(i)]
\item If $|A| \le 1$, define $\SpSplit^{\eta}_{A} \coloneqq \mathrm{id}_{A}$. 
\item If $|A| \ge 2$, define 
\[
\SpSplit^{\eta}_{A}(\{\ell_1,\ell_2\})
\coloneqq
\{\eta_{A_1}(\ell_1), \eta_{A_2}(\ell_2)\},
\]
where $\ell_{i} \in \mathsf{F}[A_{i}]$ for $i \in \{1,2\}$. 
\end{enumerate}
\end{theorem}

\begin{proof}
We construct the desired natural transformation $\eta \colon \mathsf{F} \to \mathsf{G}$ inductively by defining its components
\[
\eta_{A} \colon \mathsf{F}[A] \to \mathsf{G}[A]
\]
according to the cardinality of $A$.

If $|A| \leq 1$, then
\[
\mathsf{F}[A] = \mathsf{G}[A] = \SpSplit^{\mathsf{F}}[A] = \SpSplit^{\mathsf{G}}[A] = A,
\]
and
\[
\sigma^{\mathsf{F}}_{A}
=
\sigma^{\mathsf{G}}_{A}
=
\mathrm{id}_{A}.
\]
Define $\eta_{A} \coloneqq \mathrm{id}_{A}$. 
Then $\eta_{A}$ commutes with the splittings. 
Uniqueness and naturality are immediate.

Now assume that $|A| \geq 2$. 
Let $\ell \in \mathsf{F}[A]$ and write
\[
\sigma^{\mathsf{F}}_{A}(\ell)=\{\ell_{1},\ell_{2}\},
\]
where $\ell_{i}\in\mathsf{F}[A_i]$ for $i\in\{1,2\}$ and $a_{1},a_{2}\in A$ are distinct. 
Define
\[
m_i \coloneqq \eta_{A_i}(\ell_i).
\]
Then
\[
\SpSplit^{\eta}_{A}(\{\ell_1,\ell_2\})
=
\{m_1,m_2\}.
\]

We claim that
\[
\left|
\sigma^{\mathsf{G}}_{A_1}(m_{1})
\symdiff
\sigma^{\mathsf{G}}_{A_2}(m_{2})
\right|
=2.
\]

If $|A|=2$, this is immediate. 
Assume therefore that $|A|\ge3$. 
By the proximity property of $\sigma^{\mathsf{F}}$, we have
\[
\sigma^{\mathsf{F}}_{A_1}(\ell_{1})
\cap
\sigma^{\mathsf{F}}_{A_2}(\ell_{2})
=
\{\ell'\}
\]
for some $\ell'\in\mathsf{F}[{A'}]$. 
By the induction hypothesis applied to $A_i$,
\[
m'
\coloneqq
\eta_{{A'}}(\ell')
\]
belongs to
\[
\sigma^{\mathsf{G}}_{A_1}(m_1)
\cap
\sigma^{\mathsf{G}}_{A_2}(m_2).
\]
Hence,
\[
\sigma^{\mathsf{G}}_{A_1}(m_1)
\cap
\sigma^{\mathsf{G}}_{A_2}(m_2)
\neq \emptyset.
\]

On the other hand,
\[
\sigma^{\mathsf{G}}_{A_1}(m_1)
\neq
\sigma^{\mathsf{G}}_{A_2}(m_2),
\]
since ${A'}$ is the unique common subset of $A_1$ and $A_2$ whose cardinality is one less than $|A_1|=|A_2|$. 
Therefore,
\[
\left|
\sigma^{\mathsf{G}}_{A_1}(m_{1})
\symdiff
\sigma^{\mathsf{G}}_{A_2}(m_{2})
\right|
=2,
\]
proving the claim.

By the mergeability property of $\sigma^{\mathsf{G}}$, there exists a unique element
\[
m\in\mathsf{G}[A]
\]
such that
\[
\sigma^{\mathsf{G}}_{A}(m)
=
\{m_1,m_2\}.
\]
To make $\eta_A$ commute with the splittings, we must define
\[
\eta_A(\ell)\coloneqq m.
\]
This also proves uniqueness.

Since the construction of $\eta_A$ depends only on the set-theoretic property of $A$ and not on any additional structure on $A$, the naturality of $\eta$ follows.
For completeness, we provide the details.

Let
\[
h\colon A\longrightarrow B
\]
be a bijection. 
The naturality condition for $\eta$ is equivalent to the commutativity of the left face of the cube diagram below:
\begin{center}
\begin{tikzcd}[row sep=1.5em, column sep = 1em]
 \mathsf{F}[A] 
	\arrow[rr, "\sigma^{\mathsf{F}}_{A}"] \arrow[dr, "{ \mathsf{F}[h]}"] \arrow[dd,"\eta_{A}"'] 
&&
\SpSplit^{\mathsf{F}}[A] 
	\arrow[dd,swap, "\SpSplit^{\eta}_{A}"  {pos=0.2}]  
	\arrow[dr,"{ \SpSplit^{\mathsf{F}}[h]}"] \\
&
\mathsf{F}[B] 
	\arrow[rr, "\sigma^{\mathsf{F}}_{B}" {pos=0.8}, crossing over]
&&
\SpSplit^{\mathsf{F}}[B] 
	\arrow[dd,"\SpSplit^{\eta}_{B}" {pos=0.3}]\\
\mathsf{G}[A] 
	\arrow[rr, "\sigma^{\mathsf{G}}_{A}" {pos=0.8}]  
	\arrow[dr, swap, "{\mathsf{G}[h]}"]
&&   
\SpSplit^{\mathsf{G}}[A]   
	\arrow[mapsto,dr,"{ \SpSplit^{\mathsf{G}}[h]}"] \\
&  
\mathsf{G}[B]  
	\arrow[rr, "\sigma^{\mathsf{G}}_{B}"]   
&&   
\SpSplit^{\mathsf{G}}[B]
\arrow[from=2-2,to=4-2, "\eta_B", swap, {pos=0.3}, crossing over]
\end{tikzcd}
\end{center}
All other faces commute by the assumptions and the induction hypothesis. 
Hence,
\begin{align*}
\sigma^{\mathsf{G}}_{B} \circ \eta_{B} \circ \mathsf{F}[h]
&= \SpSplit^{\eta}_{B} \circ \sigma^{\mathsf{F}}_{B} \circ \mathsf{F}[h]
&& \text{(front face)}\\
&= \SpSplit^{\eta}_{B} \circ \SpSplit^{\mathsf{F}}[h] \circ \sigma^{\mathsf{F}}_{A}
&& \text{(top face)}\\
&= \SpSplit^{\mathsf{G}}[h] \circ \SpSplit^{\eta}_{A} \circ \sigma^{\mathsf{F}}_{A}
&& \text{(right face)}\\
&= \SpSplit^{\mathsf{G}}[h] \circ \sigma^{\mathsf{G}}_{A} \circ \eta_{A}
&& \text{(back face)}\\
&= \sigma^{\mathsf{G}}_{B} \circ \mathsf{G}[h] \circ \eta_{A}
&& \text{(bottom face)}.
\end{align*}
By the uniqueness in the mergeability property of $\sigma^{\mathsf{G}}$, we conclude that
\[
\eta_{B} \circ \mathsf{F}[h]
=
\mathsf{G}[h] \circ \eta_{A}.
\]
Hence, $\eta$ is natural.

By interchanging the roles of $\mathsf{F}$ and $\mathsf{G}$, we similarly obtain a unique natural transformation
\[
\xi \colon \mathsf{G} \to \mathsf{F}
\]
that commutes with the splittings. 
Since both compositions $\xi\circ\eta$ and $\eta\circ\xi$ also commute with the splittings, uniqueness implies that
\[
\xi \circ \eta = \id_{\mathsf{F}}
\qquad\text{and}\qquad
\eta \circ \xi = \id_{\mathsf{G}}.
\]
Therefore, $\eta$ is a natural isomorphism with inverse $\xi$.
\end{proof}

We conclude this section with two important corollaries of Theorem~\ref{thm:AxCh}.
Since the species of MAT-labeled complete graphs $\SpMAT$ (Example~\ref{ex:Gr}), regular vines $\SpRV$ (Example~\ref{ex:Vi}), and maximal ASPDs $\SpASPD$ (Example~\ref{ex:mASPD}) admit splittings satisfying the proximity and mergeability properties, we obtain the following.

\begin{corollary}
\label{cor:3iso}
Any two of the three species \(\SpMAT, \SpRV, \SpASPD\) are isomorphic.
\end{corollary}

In \cite{kurowicka2010counting-dm} and \cite{kurowicka2010regular-dm}, explicit formulas are given for the numbers of regular vines and their isomorphism classes. 
Hence, we obtain the following.

\begin{corollary}[{\cite[p.~202]{kurowicka2010counting-dm}}, {\cite[p.~227]{kurowicka2010regular-dm}}]
\label{cor:key-seq}
Let $\mathsf{F}$ be a species satisfying the conditions in Definition~\ref{def:KF}, equipped with splittings $\sigma^{\mathsf{F}}$ satisfying the proximity and mergeability properties from Definition~\ref{def:prox-merge}. 
Then
$
|\mathsf{F}[1]| = 1,
$
and for $n \geq 2$,
\[
|\mathsf{F}[n]|
=
2^{(n-2)(n-3)/2-1}\cdot n!.
\]
This sequence appears in the OEIS \cite[A185970]{OEIS}. 

Let \(\tilde{f}_n\) denote the number of unlabeled \(\mathsf{F}\)-structures, that is, the number of isomorphism classes on an $n$-element set. 
Then
$
\tilde{f}_1=\tilde{f}_2=\tilde{f}_3=1,
$
and for \(n \ge 4\),
\[
\tilde{f}_n
=
2^{(n-2)(n-3)/2 - 1}
\sum_{k=0}^{\lfloor n/2 \rfloor - 1}
c_k\,2^{-k(n-k-2)},
\]
where
\[
c_k =
\begin{cases}
1, & 0 \le k < \lfloor n/2 \rfloor - 1, \\
2, & k = \lfloor n/2 \rfloor - 1.
\end{cases}
\]
This sequence appears in the OEIS \cite[A379695]{OEIS}. 
\end{corollary}

We note that related enumeration formulas for maximal ASPDs were also obtained independently by Karpov \cite[Theorem 1]{Karpov25} using a different approach based on binary matrices.
The values of \(\tilde{f}_n\) for \(1 \le n \le 12\) are listed below:
\begin{align*}
1,1,1,
2,
6,
40,
560,
17024,
1066496,
135307264,
34496249856,
17626824704000.
\end{align*}
A recursive formula for \(\tilde{f}_n\) will be given in Corollary~\ref{cor:recursive}.

%********************************************************************************************************

\section{Explicit constructions of the correspondences}
\label{sec:correspondences}

By Theorem~\ref{thm:AxCh}, we can construct, in an inductive manner, isomorphisms between any two of the species \(\SpMAT, \SpRV, \SpASPD\). 
The strength of the theorem lies in the fact that these natural isomorphisms always exist and are unique. 
In this section, we provide explicit descriptions of these isomorphisms. 
These constructions make the correspondences more transparent and allow one to translate structural properties between the different settings.

It is worth noting that the structures \(\SpMAT, \SpRV, \SpASPD\) arise from distinct areas and are defined in quite different ways. 
Thus, constructing correspondences between them directly would be nontrivial. 
However, as will be seen in the proofs of Theorems~\ref{thm:graph-vine}, \ref{thm:graph-ASPD} and~\ref{thm:vine-ASPD}, the recursive structure provided by splitting and merging, together with Theorem~\ref{thm:AxCh}, reduces the problem to verifying that the maps are well defined and that the relevant diagrams commute.
 
We continue to use the notation introduced in the previous section.
Let $\sigma^{\mathsf{F}}\colon \mathsf{F}\longrightarrow\SpSplit^{\mathsf{F}}$ be a splitting and let $A$ be a finite set.
If
\[
\sigma^{\mathsf{F}}_{A}(\ell) = \{\ell_{1}, \ell_{2}\},
\]
denote by $A_{i} = A\setminus\{a_{i}\}$ the subset of $A$ such that $\ell_{i} \in \mathsf{F}[A_{i}]$ for $i \in \{1,2\}$, and let
$
A^{\prime}= A\setminus\{a_{1},a_{2}\}.
$

\subsection{MAT-labelings and regular vines}

It was proved in \cite[Theorem~6.10 and Corollary~6.11]{TTT24} that the categories of MAT-labeled graphs and locally regular vines are equivalent. 
Consequently, the species $\SpMAT$ and $\SpRV$ are isomorphic. 
We recall below an explicit isomorphism between these species.

Let $\lambda \in \SpMAT[A]$ and let $e = \{a,b\} \in \pi_{k}$. 
By Condition~\ref{definition MAT-labeling}(\ref{definition MAT-labeling triangle}), the edge $e$ forms exactly $k-1$ triangles with vertices $c_{1}, \dots, c_{k-1}$ whose labels are less than $k$. 
The set
\[
C_{e} \coloneqq \{a,b,c_{1}, \dots, c_{k-1}\}
\]
is called the \textbf{principal clique} generated by $e$. 
Define
$
\xi \colon \SpMAT \longrightarrow \SpRV
$
by
\begin{align*}
\xi_{A}(\lambda)
\coloneqq
\Set{\{a\} \mid a \in A}
\cup
\Set{C_{e} \mid e \in E_{K_{A}}}
\subseteq 2^{A}.
\end{align*}
See \cite[Definition~4.11 and Theorem~4.13]{TTT24} for a proof that $\xi_{A}(\lambda)$ is a regular vine.

Define
$
\eta \colon \SpRV \longrightarrow \SpMAT
$
by
\begin{align*}
\eta_{A}(\mathcal{V})(a,b)
\coloneqq
\rk(\{a\}\vee \{b\}) -1.
\end{align*}
See \cite[Definition~5.16 and Theorem~5.17]{TTT24} for a proof that $\eta_{A}(\mathcal{V})$ is an MAT-labeling of $K_{A}$.

It was shown in \cite{TTT24} by a direct proof that $\xi$ and $\eta$ are natural isomorphisms and inverses of each other. 
Here, we give an alternative proof using Theorem~\ref{thm:AxCh}.

\begin{theorem}
\label{thm:graph-vine}
The natural transformations
$
\xi \colon \SpMAT \longrightarrow \SpRV
$
and
$
\eta \colon \SpRV \longrightarrow \SpMAT
$
commute with the splittings. 
In particular, by Theorem~\ref{thm:AxCh}, they are natural isomorphisms and inverses of each other.
\end{theorem}

\begin{proof}
We proceed by induction on $|A|$. 
If $|A| \leq 2$, then $\xi_{A}$ and $\eta_{A}$ trivially commute with the splittings. 
Assume that $|A| \geq 3$.

Let $\lambda \in \SpMAT[A]$. 
Suppose that
$
\sigma^{\SpMAT}_{A}(\lambda)=\{\lambda_{1}, \lambda_{2}\}
$
and 
$
\sigma^{\SpMAT}_{A_{1}}(\lambda_{1})
\cap
\sigma^{\SpMAT}_{A_{2}}(\lambda_{2})
=
\{\lambda^{\prime}\}
$
(see Example~\ref{ex:Gr} and Remark \ref{rem:prime}).
Let
$
\mathcal{V}_{i} \coloneqq \xi_{A_{i}}(\lambda_{i})
$
for $i \in \{1,2\}$ and
$
\mathcal{V}^{\prime}
\coloneqq
\xi_{A^{\prime}}(\lambda^{\prime}).
$
By the induction hypothesis,
\[
\sigma^{\SpRV}_{A_{1}}(\mathcal{V}_{1})
\cap
\sigma^{\SpRV}_{A_{2}}(\mathcal{V}_{2})
=
\{\mathcal{V}^{\prime}\}.
\]
From the definition of $\xi$, we have
$
\mathcal{V}^{\prime}
\subseteq
\mathcal{V}_{1} \cap \mathcal{V}_{2}.
$
By Lemmas~\ref{lem:v'=v1capv2} and~\ref{lem:merging-vine}, the merging
$
\mathcal{V}_{1} \cup \mathcal{V}_{2} \cup \{A\}
$
is a regular vine. 
One can verify that
\[
\xi_{A}(\lambda)
=
\mathcal{V}_{1} \cup \mathcal{V}_{2} \cup \{A\}.
\]
Hence, $\xi$ commutes with the splittings.

Now let $\mathcal{V} \in \SpRV[A]$. 
Suppose that
$
\sigma^{\SpRV}_{A}(\mathcal{V})
=
\{\mathcal{V}_{1}, \mathcal{V}_{2}\}
$
and 
$
\sigma^{\SpRV}_{A_{1}}(\mathcal{V}_{1})
\cap
\sigma^{\SpRV}_{A_{2}}(\mathcal{V}_{2})
=
\{\mathcal{V}^{\prime}\}
$
(see Example~\ref{ex:Vi} and Remark \ref{rem:prime}).
Let
$
\lambda_{i}
\coloneqq
\eta_{A_{i}}(\mathcal{V}_{i})
$
for $i \in \{1,2\}$ and
$
\lambda^{\prime}
\coloneqq
\eta_{A^{\prime}}(\mathcal{V}^{\prime}).
$
Since
$
\mathcal{V}_{i} \supseteq \mathcal{V}^{\prime},
$
the restriction of $\lambda_{i}$ to the edge set of $K_{A^{\prime}}$ equals $\lambda^{\prime}$ for each $i \in \{1,2\}$. 
Therefore, by Lemma~\ref{lem:merge-MAT}, the labelings $\lambda_{1}$ and $\lambda_{2}$ admit a merging, and this merging equals $\eta_{A}(\mathcal{V})$. 
Hence, $\eta$ commutes with the splittings.
\end{proof}

\begin{example}
See Figure~\ref{fig:graph-vine} for an example of the correspondence in Theorem~\ref{thm:graph-vine}. 
For example, the node $abcd$ of rank $4$ in the vine is the join of $a$ and $d$; hence the edge $\{a,d\}$ in the graph receives the label $3$. 

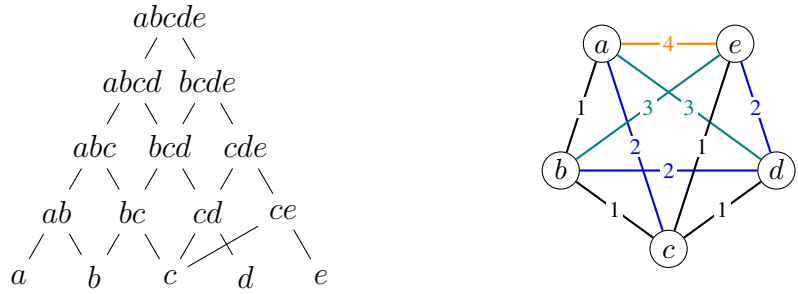
\begin{figure}[htbp!]
\begin{subfigure}{.4\textwidth}
   \centering
   \begin{tikzpicture}[baseline=(current bounding box.center), scale=1, transform shape]
\draw (0, 0) node[](0){$a$};
\draw (1, 0) node[](1){$b$};
\draw (2, 0) node[](2){$c$};
\draw (3, 0) node[](3){$d$};
\draw (4, 0) node[](4){$e$};
\draw (0.5, 0.865) node[](01){$ab$};
\draw (0)--(01);
\draw (1)--(01);
\draw (1.5, 0.865) node[](12){$bc$};
\draw (1)--(12);
\draw (2)--(12);
\draw (2.5, 0.865) node[](23){$cd$};
\draw (2)--(23);
\draw (3)--(23);
\draw (3.5, 0.865) node[](24){$ce$};
\draw (2)--(24);
\draw (4)--(24);
\draw (1.0, 1.730) node[](012){$abc$};
\draw (01)--(012);
\draw (12)--(012);
\draw (2.0, 1.730) node[](123){$bcd$};
\draw (12)--(123);
\draw (23)--(123);
\draw (3.0, 1.730) node[](234){$cde$};
\draw (23)--(234);
\draw (24)--(234);
\draw (1.5, 2.595) node[](0123){$abcd$};
\draw (012)--(0123);
\draw (123)--(0123);
\draw (2.5, 2.595) node[](1234){$bcde$};
\draw (123)--(1234);
\draw (234)--(1234);
\draw (2.0, 3.460) node[](01234){$abcde$};
\draw (0123)--(01234);
\draw (1234)--(01234);
\end{tikzpicture}
\end{subfigure}%
\begin{subfigure}{.4\textwidth}
  \centering
  \begin{tikzpicture}[baseline=(current bounding box.center), scale=1, transform shape]
\draw (126: 15mm) node[c](0){$a$};
\draw (198: 15mm) node[c](1){$b$};
\draw (270: 15mm) node[c](2){$c$};
\draw (342: 15mm) node[c](3){$d$};
\draw (414: 15mm) node[c](4){$e$};
\draw[label1] (0)--node[midway, fill=white, inner sep=1pt] {\scalebox{.7}{1}} (1) ;
\draw[label2] (0)--node[midway, fill=white, inner sep=1pt] {\scalebox{.7}{\textcolor{blue}{2}}} (2);
\draw[label3] (0)--node[midway, fill=white, inner sep=1pt] {\scalebox{.7}{\textcolor{teal}{3}}}(3);
\draw[label4] (0)--node[midway, fill=white, inner sep=1pt] {\scalebox{.7}{\textcolor{orange}{4}}}(4);
\draw[label1] (1)--node[midway, fill=white, inner sep=1pt] {\scalebox{.7}{1}} (2);
\draw[label2] (1)--node[midway, fill=white, inner sep=1pt] {\scalebox{.7}{\textcolor{blue}{2}}} (3);
\draw[label3] (1)--node[midway, fill=white, inner sep=1pt] {\scalebox{.7}{\textcolor{teal}{3}}}(4);
\draw[label1] (2)--node[midway, fill=white, inner sep=1pt] {\scalebox{.7}{1}} (3);
\draw[label1] (2)--node[midway, fill=white, inner sep=1pt] {\scalebox{.7}{1}} (4);
\draw[label2] (3)--node[midway, fill=white, inner sep=1pt] {\scalebox{.7}{\textcolor{blue}{2}}}(4);
\end{tikzpicture}

\end{subfigure}
\caption{A regular vine (left) and the corresponding MAT-labeled complete graph (right) under the correspondence in Theorem~\ref{thm:graph-vine}.}
\label{fig:graph-vine}
\end{figure}
\end{example}

\subsection{Maximal ASPDs and MAT-labelings}
Let $\mathcal{D}$ be a domain of preferences on a finite set $A$. 
Two distinct elements $a,b \in A$ are called \textbf{contiguous} in $\mathcal{D}$ if there exists $\omega \in \mathcal{D}$ such that $\{a,b\} = \{\omega(i), \omega(i+1)\}$ for some $1 \le i \le n-1$.

\begin{lemma}\label{contiguousness}
Let \(\mathcal{D}\) be a maximal ASPD on a finite set \(A\). 
We use the diagrammatic representation of \(\mathcal{D}\) from Lemma~\ref{Slinko}:
\begin{align*}
\mathcal{D} = 
\begin{tikzpicture}[baseline=25]
\draw (0,0) rectangle (6,2);
\draw (0,.5) -- (6,.5);
\draw (3,0) -- (3,2);
\draw (1.5,0.2) node{$a_{1}$};
\draw (4.5,0.2) node{$a_{2}$};
\draw (1.5, 1.2) node{$\mathcal{D}_{1}$};
\draw (4.5, 1.2) node{$\mathcal{D}_{2}$};
\end{tikzpicture} 
= \begin{tikzpicture}[baseline=25]
\draw (0,0) rectangle (6,2);
\draw (0,.5) -- (6,.5);
\draw (3,0) -- (3,2);
\draw (1.5,0.2) node{$a_{1}$};
\draw (4.5,0.2) node{$a_{2}$};
\draw (1.5,.5) -- (1.5,2);
\draw (4.5,.5) -- (4.5,2);
\draw (0,1) -- (6,1);
\draw (0.75,.74) node{$b_{1}$};
\draw (2.25,.7) node{$a_{2}$};
\draw (3.75,.7) node{$a_{1}$};
\draw (5.25,.74) node{$b_{2}$};
\draw (2.25, 1.5) node{$\mathcal{D}^{\prime}$};
\draw (3.75, 1.5) node{$\mathcal{D}^{\prime}$};
\end{tikzpicture} \, ,
\end{align*}
where \(a_{1}, a_{2}\) are the bottom alternatives of \(\mathcal{D}\), \(b_1,b_2 \in A'=A\setminus\{a_{1},a_{2}\}\), and \(\mathcal{D}_{i}, \mathcal{D}^{\prime}\) are maximal ASPDs on \(A_i=A\setminus\{a_{i}\}\) and \(A'\), respectively.

Let \(x,y \in A\) be distinct alternatives. Then:
\begin{enumerate}[(1)]
\item\label{contiguousness 1} 
The alternatives \(x\) and \(y\) are contiguous in \(\mathcal{D}\).
\item\label{contiguousness 2} 
If \(x,y \in A'\), then a topmost contiguous occurrence of \(x\) and \(y\) appears in \(\mathcal{D}^{\prime}\).
\end{enumerate}
\end{lemma}

\begin{proof}
We first prove part (\ref{contiguousness 1}) by induction on $|A|$. 
The case $|A| = 2$ is trivial. 
Suppose $|A| \geq 3$. 
If $\{x,y\} = \{a_{1}, a_{2}\}$, then $x$ and $y$ are  clearly contiguous in $\mathcal{D}$. 
Otherwise, by symmetry, we may assume that $a_{1} \not\in \{x,y\}$. 
By the induction hypothesis, $x$ and $y$ are contiguous in $\mathcal{D}_{1}$ hence in  $\mathcal{D}$. 

Now we show part (\ref{contiguousness 2}) also by induction on $|A|$. 
Suppose that $|A|=4$. 
Then there are exactly two maximal ASPDs $\mathcal{D}_{4,1}$ and $\mathcal{D}_{4,2}$ illustrated in Figure \ref{fig:ASPD-rk4} and the assertion is true for these domains. 
Suppose $|A| \geq 5$ and let $x,y \in A'$. 
It is suffices to show that a topmost contiguous occurrence of $x$ and $y$ in $\mathcal{D}_{1}$ appears in $\mathcal{D}'$. 

First, consider the case $b_{1} \in \{x,y\}$. 
Then, in $\mathcal{D}_{1}$, any contiguous occurrence of $x$ and $y$ in a preference whose bottom alternative is $b_{1}$ occurs at the bottom. 
From part (\ref{contiguousness 1}), at least one contiguous occurrence of $x$ and $y$ appears in $\mathcal{D}'$. 
Therefore, a topmost contiguous occurrence of $x$ and $y$ in $\mathcal{D}_{1}$ appears in $\mathcal{D}'$. 

Next, suppose that $b_{1} \not\in \{x,y\}$. 
Since $x$ and $y$ are not bottom alternatives of $\mathcal{D}_{1}$, a topmost contiguous occurrence of $x$ and $y$ in $\mathcal{D}_{1}$ appears in $\mathcal{D}_{1}'$ by the induction hypothesis, where 
\begin{align*}
\mathcal{D}_{1} = 
\begin{tikzpicture}[baseline=25]
\draw (0,0) rectangle (6,2);
\draw (0,.5) -- (6,.5);
\draw (3,0) -- (3,2);
\draw (1.5,0.24) node{$b_{1}$};
\draw (4.5,0.2) node{$a_{2}$};
\draw (4.5, 1.2) node{$\mathcal{D}'$};
\end{tikzpicture} 
= \begin{tikzpicture}[baseline=25]
\draw (0,0) rectangle (6,2);
\draw (0,.5) -- (6,.5);
\draw (3,0) -- (3,2);
\draw (1.5,0.24) node{$b_{1}$};
\draw (4.5,0.2) node{$a_{2}$};
\draw (1.5,.5) -- (1.5,2);
\draw (4.5,.5) -- (4.5,2);
\draw (0,1) -- (6,1);
\draw (2.25,.7) node{$a_{2}$};
\draw (3.75,.74) node{$b_{1}$};
\draw (2.25, 1.46) node{$\mathcal{D}_{1}^{\prime}$};
\draw (3.75, 1.46) node{$\mathcal{D}_{1}^{\prime}$};
\end{tikzpicture} \, . 
\end{align*}
Therefore, a topmost contiguous occurrence of $x$ and $y$ in $\mathcal{D}_{1}$ appears in $\mathcal{D}'$.
\end{proof}

Recall the notion of an MAT-perfect elimination ordering (MAT-PEO) of an MAT-labeled graph from Definition \ref{def:MAT-PEO}.

\begin{theorem}
\label{thm:graph-ASPD}
Define natural transformations $\xi \colon \SpMAT \longrightarrow \SpASPD$ and $\eta \colon \SpASPD \longrightarrow \SpMAT$ by 
\begin{align*}
\xi_{A}(\lambda) &\coloneqq \Set{\omega \in \mathcal{L}(A) | \text{$\omega$ is an MAT-PEO of $(K_{A},\lambda)$}}, \\
\eta_{A}(\mathcal{D})(a,b) &\coloneqq \min\Set{i \in [n - 1] | \text{there exists $\omega \in \mathcal{D}$ such that $\{a,b\} = \{\omega(i), \omega(i+1)\}$} }, 
\end{align*}
where $n\coloneqq |A|$. 
Then $\xi$ and $\eta$ commute with the splittings. 
In particular, by Theorem \ref{thm:AxCh}, they are natural isomorphisms and inverses of each other. 
\end{theorem}
\begin{proof}
Let $\lambda \in \SpMAT[A]$ be an MAT-labeling of the complete graph $K_{A}$. 
We first show  by induction on \(n= |A|\) that the map $\xi_A$ is well defined, i.e., 
\begin{align*}
\mathcal{D}\coloneqq\xi_A(\lambda) \in \SpASPD[A].
\end{align*}

The case $n \le 2$ is trivial. 
Suppose $n \geq 3$. 
Suppose that
$
\sigma^{\SpMAT}_{A}(\lambda)=\{\lambda_{1}, \lambda_{2}\}
$
and 
$
\sigma^{\SpMAT}_{A_{1}}(\lambda_{1})
\cap
\sigma^{\SpMAT}_{A_{2}}(\lambda_{2})
=
\{\lambda^{\prime}\}
$
(see Example~\ref{ex:Gr} and Remark \ref{rem:prime}).
Note that $A_{i} = A\setminus\{a_{i}\}$, where $a_{1},a_{2}$ are the MAT-simplicial vertices of $(K_{A}, \lambda)$.
Since every preference in $\mathcal{D}$ has either $a_{1}$ or $a_{2}$ as the bottom alternative, $\mathcal{D}$ can be partitioned as 
\begin{align*}
\mathcal{D} = 
\Set{\omega \in \mathcal{D} | \omega(n) = a_{1}} \cup \Set{\omega \in \mathcal{D} | \omega(n) = a_{2}}. 
\end{align*}

Since $a_{i}$ is MAT-simplicial in $(K_{A_{3-i}}, \lambda_{3-i})$ 
for $i \in \{1,2\}$, the map 
\begin{align*}
f \colon \Set{\omega \in \mathcal{D} | \omega(n-1) = a_{2},\, \omega(n) = a_{1}} \longrightarrow
\Set{\omega \in \mathcal{D} | \omega(n-1) = a_{1},\, \omega(n) = a_{2}}
\end{align*}
defined by 
\begin{align*}
f(\omega)(i) \coloneqq \begin{cases}
\omega(i) & \text{ if } 1 \leq i \leq n-2; \\
a_{1} & \text{ if } i = n-1; \\
a_{2} & \text{ if } i = n,
\end{cases}
\end{align*}
is a bijection. Then $\mathcal{D} $ can be written as 
\begin{align*}
\mathcal{D} = 
\begin{tikzpicture}[baseline=25]
\draw (0,0) rectangle (6,2);
\draw (0,.5) -- (6,.5);
\draw (3,0) -- (3,2);
\draw (1.5,0.2) node{$a_{1}$};
\draw (4.5,0.2) node{$a_{2}$};
\draw (1.5, 1.2) node{$\mathcal{D}_{1}$};
\draw (4.5, 1.2) node{$\mathcal{D}_{2}$};
\end{tikzpicture}
= \begin{tikzpicture}[baseline=25]
\draw (0,0) rectangle (6,2);
\draw (0,.5) -- (6,.5);
\draw (3,0) -- (3,2);
\draw (1.5,0.2) node{$a_{1}$};
\draw (4.5,0.2) node{$a_{2}$};
\draw (1.5,.5) -- (1.5,2);
\draw (4.5,.5) -- (4.5,2);
\draw (0,1) -- (6,1);
\draw (2.25,.7) node{$a_{2}$};
\draw (3.75,.7) node{$a_{1}$};
\draw (2.25, 1.5) node{$\mathcal{D}^{\prime}$};
\draw (3.75, 1.5) node{$\mathcal{D}^{\prime}$};
\end{tikzpicture} \, , 
\end{align*}
where $\mathcal{D}_{i} = \xi_{A_i}(\lambda_i)$ for $i \in \{1,2\}$ and $\mathcal{D}^{\prime}=\xi_{A'}(\lambda')$.
By the induction hypothesis, $\mathcal{D}_{i}$ and $\mathcal{D}^{\prime}$ are maximal ASPDs. 
Therefore, by Lemma \ref{merging ASPD}, $\mathcal{D}$ is the merging of $\mathcal{D}_{1}$ and $\mathcal{D}_{2}$. 
Hence, $\mathcal{D}  \in \SpASPD[A]$. 
This also shows that $\xi$ commutes with the splittings.

Now, let $\mathcal{D} \in \SpASPD[A]$.
We show by induction on \(n\) that the map $\eta_A$ is well defined, i.e., 
\begin{align*}
\lambda\coloneqq\eta_{A}(\mathcal{D}) \in \SpMAT[A].
\end{align*}
Note that the label $\lambda(a,b)$ of each edge $\{a,b\}$ of $K_A$ is well defined by Lemma~\ref{contiguousness}\eqref{contiguousness 1}.

The case $n \le 2$ is trivial. 
Suppose $n \geq 3$. 
Suppose that
$\sigma^{\SpASPD}_{A}(\mathcal{D}) = \{\mathcal{D}_{1}, \mathcal{D}_{2}\}$ 
and 
$\sigma^{\SpASPD}_{A_{1}}(\mathcal{D}_{1}) \cap \sigma^{\SpASPD}_{A_{2}}(\mathcal{D}_{2}) = \{\mathcal{D}^{\prime}\}$
(see Example~\ref{ex:mASPD} and Remark \ref{rem:prime}).
Then, by Lemma \ref{Slinko}, we may write
\begin{align*}
\mathcal{D} = 
\begin{tikzpicture}[baseline=25]
\draw (0,0) rectangle (6,2);
\draw (0,.5) -- (6,.5);
\draw (3,0) -- (3,2);
\draw (1.5,0.2) node{$a_{1}$};
\draw (4.5,0.2) node{$a_{2}$};
\draw (1.5, 1.2) node{$\mathcal{D}_{1}$};
\draw (4.5, 1.2) node{$\mathcal{D}_{2}$};
\end{tikzpicture}
= \begin{tikzpicture}[baseline=25]
\draw (0,0) rectangle (6,2);
\draw (0,.5) -- (6,.5);
\draw (3,0) -- (3,2);
\draw (1.5,0.2) node{$a_{1}$};
\draw (4.5,0.2) node{$a_{2}$};
\draw (1.5,.5) -- (1.5,2);
\draw (4.5,.5) -- (4.5,2);
\draw (0,1) -- (6,1);
\draw (2.25,.7) node{$a_{2}$};
\draw (3.75,.7) node{$a_{1}$};
\draw (2.25, 1.5) node{$\mathcal{D}^{\prime}$};
\draw (3.75, 1.5) node{$\mathcal{D}^{\prime}$};
\end{tikzpicture} \, .
\end{align*}
By the induction hypothesis, $\lambda_i\coloneqq\eta_{A_i}(\mathcal{D}_{i})$ for each $i \in \{1,2\}$ and $\lambda'\coloneqq\eta_{A'}(\mathcal{D}^{\prime})$ are MAT-labeling of the complete graphs $K_{A_i}$ and $K_{A'}$, respectively. 

Clearly, $\lambda(a_{1},a_{2}) = n-1$. 
If $x \in A'$, then 
$$\lambda(a_{3-i},x) = \lambda_i(a_{3-i},x).$$ 
If $x,y \in A'$ and $x \ne y$, then by Lemma \ref{contiguousness}(\ref{contiguousness 2}), 
$$\lambda(x,y) = \lambda'(x,y).
$$
Therefore, by Lemma \ref{lem:merge-MAT}, $\lambda $ is the merging of $\lambda_1$  and $\lambda_2$. Hence, $\lambda\in \SpMAT[A]$. 
This also proves that $\eta$ commutes with the splittings.

\end{proof}

\begin{example}
See Figure~\ref{fig:graph-ASPD} for an example of the correspondence in Theorem \ref{thm:graph-ASPD}. 
Every preference in the domain is an MAT-PEO of the graph. 
For the converse, the label of an edge is determined by the position of topmost occurrence of the endpoints. 
For instance, the edge \(\{a,c\}\) labeled by \(2\) in the graph indicates that the topmost occurrence of \(a\) and \(c\) as contiguous alternatives in the domain appears at position \(2\).
\end{example}

 \begin{figure}[htbp!]

\begin{tikzpicture}[baseline=(current bounding box.center), scale=1, transform shape]
\draw (126: 15mm) node[c](0){$a$};
\draw (198: 15mm) node[c](1){$b$};
\draw (270: 15mm) node[c](2){$c$};
\draw (342: 15mm) node[c](3){$d$};
\draw (414: 15mm) node[c](4){$e$};
\draw[label1] (0)--node[midway, fill=white, inner sep=1pt] {\scalebox{.7}{1}} (1) ;
\draw[label2] (0)--node[midway, fill=white, inner sep=1pt] {\scalebox{.7}{\textcolor{blue}{2}}} (2);
\draw[label3] (0)--node[midway, fill=white, inner sep=1pt] {\scalebox{.7}{\textcolor{teal}{3}}}(3);
\draw[label4] (0)--node[midway, fill=white, inner sep=1pt] {\scalebox{.7}{\textcolor{orange}{4}}}(4);
\draw[label1] (1)--node[midway, fill=white, inner sep=1pt] {\scalebox{.7}{1}} (2);
\draw[label2] (1)--node[midway, fill=white, inner sep=1pt] {\scalebox{.7}{\textcolor{blue}{2}}} (3);
\draw[label3] (1)--node[midway, fill=white, inner sep=1pt] {\scalebox{.7}{\textcolor{teal}{3}}}(4);
\draw[label1] (2)--node[midway, fill=white, inner sep=1pt] {\scalebox{.7}{1}} (3);
\draw[label1] (2)--node[midway, fill=white, inner sep=1pt] {\scalebox{.7}{1}} (4);
\draw[label2] (3)--node[midway, fill=white, inner sep=1pt] {\scalebox{.7}{\textcolor{blue}{2}}}(4);
\end{tikzpicture}
\qquad
\begin{tikzpicture}[baseline=0]
\draw[label1, fill=black!50!white, fill opacity=0.2] (-4.5,0.3) rectangle (-4.06,1.2);
\draw[label1, fill=black!50!white, fill opacity=0.2] (-3.35,0.3) rectangle (-2.95,1.2);
\draw[label1, fill=black!50!white, fill opacity=0.2] (-1.05,0.3) rectangle (-0.66,1.2);
\draw[label1, fill=black!50!white, fill opacity=0.2] (3.52,0.3) rectangle (3.9,1.2);
\draw[label2, fill=blue!50!white, fill opacity=0.2] (-3.9,-0.2) rectangle (-3.53,0.75);
\draw[label2, fill=blue!50!white, fill opacity=0.2] (-1.62,-0.2) rectangle (-1.23,0.75);
\draw[label2, fill=blue!50!white, fill opacity=0.2] (2.4,-0.2) rectangle (2.8,0.75);
\draw[label3, fill=teal!50!white, fill opacity=0.2] (-2.8,-0.7) rectangle (-2.35,0.2);
\draw[label3, fill=teal!50!white, fill opacity=0.2] (2.95,-0.7) rectangle (3.35,0.2);
\draw[label4, fill=orange!50!white, fill opacity=0.2] (-2.2,-1.2) rectangle (-1.8,-0.25);
\draw (0,0) node{$\begin{array}{|cccc|cccc|cccc|cccc|}
\hline
a&b&b& c & b&c&c&d & b&c&c& d & c&d&c&e \\
b&a&c& b & c&b&d&c & c&b&d& c & d&c&e&c \\
c&c&a& a & d&d&b&b & d&d&b& b & e&e&d&d \\ 
d&d&d& d & a&a&a&a & e&e&e& e & b&b&b&b \\ 
e&e&e&e & e&e&e&e & a&a&a&a & a&a&a&a \\ \hline
\end{array}$};
\end{tikzpicture}

\caption{An MAT-labeled graph (left) and its corresponding maximal ASPD (right) under the correspondence in Theorem~\ref{thm:graph-ASPD}. 
}
\label{fig:graph-ASPD}
\end{figure}
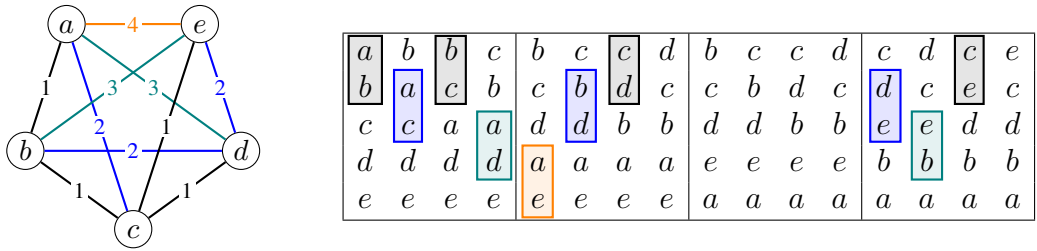

   %********************************************************************************************************

\subsection{Maximal ASPDs and regular vines}

\begin{theorem}\label{thm:vine-ASPD}
Define natural transformations $\xi \colon \SpRV \longrightarrow \SpASPD$ and $\eta \colon \SpASPD \longrightarrow \SpRV$ by 
\begin{align*}
\xi_{A}(\mathcal{V}) &\coloneqq \Set{ \omega \in \mathcal{L}(A) | \{\omega(1)\} \subseteq \{\omega(1), \omega(2)\} \subseteq \dots \subseteq \{\omega(1), \dots, \omega(n)\} \text{ is a maximal chain in $\mathcal{V}$} }, \\
\eta_{A}(\mathcal{D}) &\coloneqq \Set{\{\omega(1), \dots, \omega(k)\} \in 2^{A} | \omega \in \mathcal{D}, \,1 \leq k \leq n}, 
\end{align*}
where $n \coloneqq |A|$. 
Then $\xi$ and $\eta$ commute with the splittings. 
In particular, by Theorem \ref{thm:AxCh}, they are natural isomorphisms and inverses of each other. 
\end{theorem}

\begin{proof}
Let $\mathcal{V}\in \SpRV[A]$ be a regular vine on $A$. 
We first show  by induction on \(n\) that the map $\xi_A$ is well defined, i.e., 
\begin{align*}
\mathcal{D}\coloneqq\xi_{A}(\mathcal{V}) \in \SpASPD[A].
\end{align*}

The case $n \le 2$ is trivial. 
Suppose $n \geq 3$. 
Assume that $\sigma^{\SpRV}_{A}(\mathcal{V}) = \{\mathcal{V}_{1}, \mathcal{V}_{2}\}$ and $\sigma^{\SpRV}_{A_{1}}(\mathcal{V}_{1}) \cap \sigma^{\SpRV}_{A_{2}}(\mathcal{V}_{2}) = \{\mathcal{V}^{\prime}\}$ 
(see Example~\ref{ex:Vi} and Remark \ref{rem:prime}).
By the induction hypothesis, $\mathcal{D}_i\coloneqq\xi_{A_i}(\mathcal{V}_i) \in \SpASPD[A_i]$ and $\mathcal{D}'\coloneqq\xi_{A'}(\mathcal{V}') \in \SpASPD[A']$. 
Since every maximal chain of $\mathcal{V}$ passes through either $A_{1}$ or $A_{2}$, and since the maximal chains of $\mathcal{V}_{1}$ and $\mathcal{V}_{2}$ passing through $A'$ coincide except for their top elements, \(\mathcal{D}\) can be written as 
\begin{align*}
\mathcal{D}= 
\begin{tikzpicture}[baseline=25]
\draw (0,0) rectangle (6,2);
\draw (0,.5) -- (6,.5);
\draw (3,0) -- (3,2);
\draw (1.5,0.2) node{$a_{1}$};
\draw (4.5,0.2) node{$a_{2}$};
\draw (1.5, 1.2) node{$\mathcal{D}_1$};
\draw (4.5, 1.2) node{$\mathcal{D}_2$};
\end{tikzpicture}
= \begin{tikzpicture}[baseline=25]
\draw (0,0) rectangle (6,2);
\draw (0,.5) -- (6,.5);
\draw (3,0) -- (3,2);
\draw (1.5,0.2) node{$a_{1}$};
\draw (4.5,0.2) node{$a_{2}$};
\draw (1.5,.5) -- (1.5,2);
\draw (4.5,.5) -- (4.5,2);
\draw (0,1) -- (6,1);
\draw (2.25,.7) node{$a_{2}$};
\draw (3.75,.7) node{$a_{1}$};
\draw (2.25, 1.5) node{$\mathcal{D}'$};
\draw (3.75, 1.5) node{$\mathcal{D}'$};
\end{tikzpicture}, 
\end{align*}
Therefore, by Lemma \ref{merging ASPD}, $\mathcal{D}$ is the merging of $\mathcal{D}_1$ and $\mathcal{D}_{2}$. 
Hence,  $\mathcal{D} \in \SpASPD[A]$. 
This also proves that $\xi$ commutes with the splittings.

Now, let $\mathcal{D} \in \SpASPD[A]$.
 We show  by induction on \(n\) that the map $\eta_A$ is well defined, i.e., 
\begin{align*}
\mathcal{V}\coloneqq\eta_{A}(\mathcal{D}) \in \SpASPD[A].
\end{align*}

The case $n \le 2$ is trivial. 
Suppose $n \geq 3$. 
Let $\sigma^{\SpASPD}_{A}(\mathcal{D}) = \{\mathcal{D}_{1}, \mathcal{D}_{2}\}$ and $\sigma^{\SpASPD}_{A_{1}}(\mathcal{D}_{1}) \cap \sigma^{\SpASPD}_{A_{2}}(\mathcal{D}_{2}) = \{\mathcal{D}^{\prime}\}$
(see Example~\ref{ex:mASPD} and Remark \ref{rem:prime}).
Then, by Lemma \ref{Slinko}, we may write
\begin{align*}
\mathcal{D} = 
\begin{tikzpicture}[baseline=25]
\draw (0,0) rectangle (6,2);
\draw (0,.5) -- (6,.5);
\draw (3,0) -- (3,2);
\draw (1.5,0.2) node{$a_{1}$};
\draw (4.5,0.2) node{$a_{2}$};
\draw (1.5, 1.2) node{$\mathcal{D}_{1}$};
\draw (4.5, 1.2) node{$\mathcal{D}_{2}$};
\end{tikzpicture}
= \begin{tikzpicture}[baseline=25]
\draw (0,0) rectangle (6,2);
\draw (0,.5) -- (6,.5);
\draw (3,0) -- (3,2);
\draw (1.5,0.2) node{$a_{1}$};
\draw (4.5,0.2) node{$a_{2}$};
\draw (1.5,.5) -- (1.5,2);
\draw (4.5,.5) -- (4.5,2);
\draw (0,1) -- (6,1);
\draw (2.25,.7) node{$a_{2}$};
\draw (3.75,.7) node{$a_{1}$};
\draw (2.25, 1.5) node{$\mathcal{D}^{\prime}$};
\draw (3.75, 1.5) node{$\mathcal{D}^{\prime}$};
\end{tikzpicture} \, .
\end{align*}

By the induction hypothesis, $\mathcal{V}_{i}\coloneqq\eta_{A_i}(\mathcal{D}_{i})$ for $i \in \{1,2\}$ and $\mathcal{V}'\coloneqq\eta_{A'}(\mathcal{D}^{\prime})$ are regular vines on $A_i$ and $A'$, respectively. 
It is easy to see that 
\begin{align*}
\mathcal{V}' \subseteq \mathcal{V}_1 \cap \mathcal{V}_2 \quad \text{and} \quad \mathcal{V} = \mathcal{V}_1 \cup \mathcal{V}_2 \cup \{A\}.
\end{align*}
Therefore, by Lemmas \ref{lem:v'=v1capv2} and \ref{lem:merging-vine}, $\mathcal{V}$ is the merging of $ \mathcal{V}_1$  and $ \mathcal{V}_2$,  hence a regular vine. 
This also shows that $\eta$ commutes with the splittings. 
\end{proof}

The following is an immediate consequence of Theorem \ref{thm:vine-ASPD}.

\begin{corollary}\label{orders in maxASPD <-> max chains of R-vine}
Let $\mathcal{D}$ be a maximal ASPD and let $\mathcal{V}$ be the regular vine corresponding to $\mathcal{D}$ under the isomorphism in  Theorem \ref{thm:vine-ASPD}. 
Then for every $\omega \in \mathcal{D}$, 
\begin{align*}
\{\omega(1)\} \subseteq \{\omega(1), \omega(2)\} \subseteq \dots \subseteq \{\omega(1), \omega(2), \dots, \omega(n)\}
\end{align*}
is a maximal chain in $\mathcal{V}$. 
Conversely, for every maximal chain $C_{1} \subseteq \dots \subseteq C_{n}$ in $\mathcal{V}$, 
there exists (a unique) $\omega \in \mathcal{D}$ such that for all $1 \le k \le n$,
$$C_k= \{\omega(1), \dots, \omega(k)\}.$$
As a consequence, the isomorphism gives a one-to-one correspondence between $\mathcal{D}$ and the set of maximal chains of $\mathcal{V}$. 
\end{corollary}

\begin{example}
See Figure~\ref{fig:vine-ASPD} for an example of the correspondence in Theorem \ref{thm:vine-ASPD}.
For instance, the maximal chain $c \subseteq cd \subseteq bcd \subseteq bcde \subseteq abcde$ in the vine corresponds to the preference $cdbea$ in the domain.
\end{example}

 \begin{figure}[htbp!]
   \centering
\begin{tikzpicture}[baseline=(current bounding box.center), scale=1, transform shape]
\draw[rounded corners=2mm, blue, thick, fill=blue!50!white, fill opacity=0.2] (1.5,-0.2)--(2.3,-0.2)--(2.9,0.85)--(2.4,1.7)--(3.05,2.5)--(2.5,3.7)--(1.35,3.7)--(1.35,3.3)--(2.1,2.65)--(1.5,1.7)--(2.1,0.85)--cycle;
\draw (0, 0) node[](0){$a$};
\draw (1, 0) node[](1){$b$};
\draw (2, 0) node[](2){$c$};
\draw (3, 0) node[](3){$d$};
\draw (4, 0) node[](4){$e$};
\draw (0.5, 0.865) node[](01){$ab$};
\draw (0)--(01);
\draw (1)--(01);
\draw (1.5, 0.865) node[](12){$bc$};
\draw (1)--(12);
\draw (2)--(12);
\draw (2.5, 0.865) node[](23){$cd$};
\draw (2)--(23);
\draw (3)--(23);
\draw (3.5, 0.865) node[](24){$ce$};
\draw (2)--(24);
\draw (4)--(24);
\draw (1.0, 1.730) node[](012){$abc$};
\draw (01)--(012);
\draw (12)--(012);
\draw (2.0, 1.730) node[](123){$bcd$};
\draw (12)--(123);
\draw (23)--(123);
\draw (3.0, 1.730) node[](234){$cde$};
\draw (23)--(234);
\draw (24)--(234);
\draw (1.5, 2.595) node[](0123){$abcd$};
\draw (012)--(0123);
\draw (123)--(0123);
\draw (2.5, 2.595) node[](1234){$bcde$};
\draw (123)--(1234);
\draw (234)--(1234);
\draw (2.0, 3.460) node[](01234){$abcde$};
\draw (0123)--(01234);
\draw (1234)--(01234);
\end{tikzpicture}
\qquad
\begin{tikzpicture}[baseline=0]
\draw[blue, thick, fill=blue!50!white, fill opacity=0.2] (1.15,-1.18) rectangle (1.7,1.18);
\draw (0,0) node{$\begin{array}{|cccc|cccc|cccc|cccc|}
\hline
a&b&b& c & b&c&c&d & b&c&c& d & c&d&c&e \\
b&a&c& b & c&b&d&c & c&b&d& c & d&c&e&c \\
c&c&a& a & d&d&b&b & d&d&b& b & e&e&d&d \\ 
d&d&d& d & a&a&a&a & e&e&e& e & b&b&b&b \\ 
e&e&e&e & e&e&e&e & a&a&a&a & a&a&a&a \\ \hline
\end{array}$};
\end{tikzpicture}
\caption{A regular vine (left) and the corresponding maximal ASPD (right) under the correspondence in Theorem~\ref{thm:vine-ASPD}. }
\label{fig:vine-ASPD}
\end{figure}
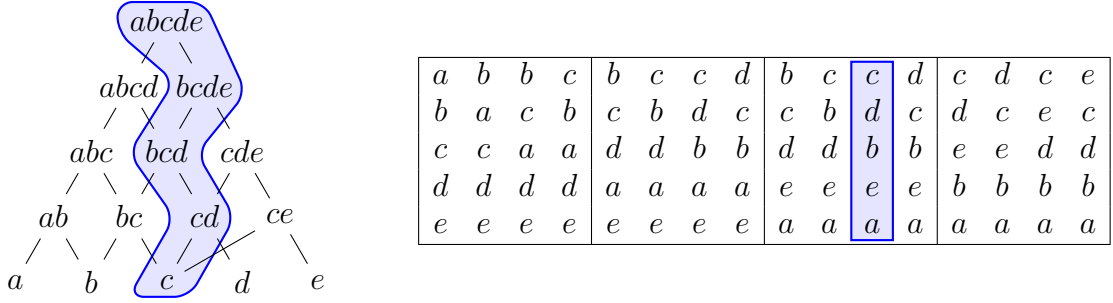

%********************************************************************************************************
\section{Applications to social choice theory}
\label{sec:appl-sct}

Theorem~\ref{thm:AxCh} characterizes maximal ASPDs through two structural axioms. 
Combined with the correspondences established in Theorems~\ref{thm:graph-ASPD} and~\ref{thm:vine-ASPD}, this characterization yields concrete combinatorial representations of maximal ASPDs via MAT-labeled complete graphs and regular vines. 
In particular, the representation by regular vines naturally leads to an explicit formula for the number of non-isomorphic maximal ASPDs (Corollary~\ref{cor:key-seq}).

These results provide our main applications to social choice. 
In this section we present three further consequences obtained from the correspondence with regular vines in Theorem~\ref{thm:vine-ASPD}.

Throughout this section, let $\mathcal{D}$ be a maximal ASPD on a finite set $A$, and let $\mathcal{V}$ denote the regular vine corresponding to $\mathcal{D}$ under the isomorphism $\eta$ in Theorem~\ref{thm:vine-ASPD}.

%********************************************************************************************************

\subsection{Maximal BSPDs and D-vines}

We show that maximal BSPDs correspond to D-vines under the isomorphism in Theorem~\ref{thm:vine-ASPD}, thereby giving a poset characterization of these domains. This correspondence further illustrates the close alignment between the theory of single-peaked domains and that of regular vines. 

\begin{proposition}
\label{prop:BSPD-Dvine}
Let $\mathcal{D}$ be a maximal ASPD and $\mathcal{V}$ the corresponding regular vine. 
Then $\mathcal{D}$ is a maximal BSPD if and only if $\mathcal{V}$ is a D-vine. 
\end{proposition}

\begin{proof}
The correspondence in Theorem \ref{thm:vine-ASPD} maps maximal chains in $\mathcal{V}$ to preferences in $\mathcal{D}$
(see also Corollary \ref{orders in maxASPD <-> max chains of R-vine}).
The assertion follows from the characterizations of D-vines  and maximal BSPDs in Propositions \ref{prop:Dvine-chain} and \ref{prop:BSPD}. 
\end{proof}

%********************************************************************************************************

\subsection{First-rank distribution}

It was shown in \cite[Theorem 1]{Slinko19} that any maximal ASPD is \textbf{minimally rich}, meaning that every alternative appears as the first-ranked alternative in at least one preference in the domain. 
We refine this result by giving a formula for the number of times each alternative is first-ranked.

For a domain $\mathcal{D}$ on a set $A$, define the \textbf{first-rank distribution} $\operatorname{first}(\mathcal{D})$ to be the multiset 
\[
\operatorname{first}(\mathcal{D}) = \Set{t_a | a \in A}, 
\quad \text{where } 
t_a \coloneqq \lvert\Set{\omega \in \mathcal{D} | \omega(1) = a }\rvert.
\]
As a consequence of Corollary \ref{orders in maxASPD <-> max chains of R-vine}, we obtain the following.

\begin{proposition}
\label{prop:TRD}
Let $\mathcal{D}$ be a maximal ASPD on a finite set $A$.
Then for each $a \in A$, the first-rank number $t_a$ equals the number of maximal chains in $\mathcal{V}$ starting from the minimal element $\{a\}$. 
\end{proposition}

The number of maximal chains in a regular vine can be computed via a rule analogous to Pascal's triangle. 
In particular, the first-rank distributions corresponding to D-vines and C-vines are given by binomial coefficients and powers of $2$, respectively, which can be verified by induction on the number of alternatives. 

\begin{example}
The first-rank distributions of the domains $\mathcal{D}_{4,1}$ and $\mathcal{D}_{4,2}$ in Figure~\ref{fig:ASPD-rk4} can be computed by counting maximal chains in the corresponding regular vines $\mathcal{V}_{4,1}$ and $\mathcal{V}_{4,2}$ in Figure~\ref{fig:TRD}. 
\end{example}

 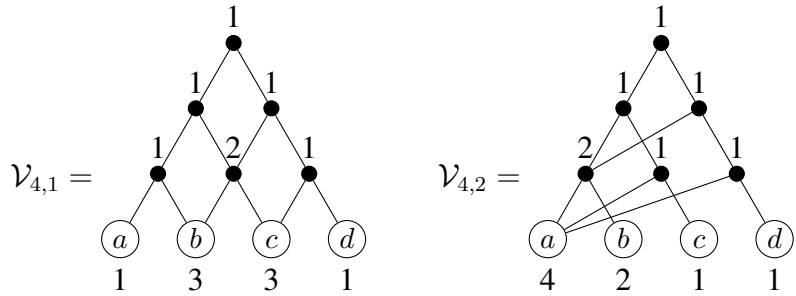
\begin{figure}[htbp!]
 \centering
$\mathcal{V}_{4,1} =$ \begin{tikzpicture}[baseline=20]
\draw (0,0) node[c](11){$a$} node[yshift=-15]{1};
\draw (1,0) node[c](12){$b$} node[yshift=-15]{3};
\draw (2,0) node[c](13){$c$} node[yshift=-15]{3};
\draw (3,0) node[c](14){$d$} node[yshift=-15]{1};
\draw (0.5, 0.865) node[v](21){} node[yshift=10]{1};
\draw (1.5, 0.865) node[v](22){} node[yshift=10]{2};
\draw (2.5, 0.865) node[v](23){} node[yshift=10]{1};
\draw (1,1.73) node[v](31){} node[yshift=10]{1};
\draw (2,1.73) node[v](32){} node[yshift=10]{1};
\draw (1.5,2.595) node[v](41){} node[yshift=10]{1};
\draw (41) -- (31) -- (21) -- (11);
\draw (41) -- (32) -- (23) -- (14);
\draw (31) -- (22) -- (32);
\draw (21) -- (12) -- (22) -- (13) -- (23);
\end{tikzpicture}
\qquad
$\mathcal{V}_{4,2} =$ \begin{tikzpicture}[baseline=20]
\draw (0,0) node[c](11){$a$} node[yshift=-15]{4};
\draw (1,0) node[c](12){$b$} node[yshift=-15]{2};
\draw (2,0) node[c](13){$c$} node[yshift=-15]{1};
\draw (3,0) node[c](14){$d$} node[yshift=-15]{1};
\draw (0.5, 0.865) node[v](21){} node[yshift=10]{2};
\draw (1.5, 0.865) node[v](22){} node[yshift=10]{1};
\draw (2.5, 0.865) node[v](23){} node[yshift=10]{1};
\draw (1,1.73) node[v](31){} node[yshift=10]{1};
\draw (2,1.73) node[v](32){} node[yshift=10]{1};
\draw (1.5,2.595) node[v](41){} node[yshift=10]{1};
\draw (41) -- (31) -- (21) -- (11);
\draw (41) -- (32) -- (23) -- (14);
\draw (31) -- (22) -- (13);
\draw (21) -- (12);
\draw (32) -- (21);
\draw (22) -- (11);
\draw (23) -- (11);
\end{tikzpicture}

\caption{The regular vines corresponding to the domains $\mathcal{D}_{4,1}$ and $\mathcal{D}_{4,2}$ in Figure~\ref{fig:ASPD-rk4}. 
The first-rank distributions are 
$\operatorname{first}(\mathcal{D}_{4,1}) = \{1,3,3,1\}$ and $\operatorname{first}(\mathcal{D}_{4,2}) = \{4,2,1,1\}$.}
\label{fig:TRD}
\end{figure}

%********************************************************************************************************

\subsection{Richness}

In this subsection, we discuss another application of Theorem~\ref{thm:vine-ASPD}, concerning a generalization of minimal richness. 

A domain $\mathcal{D}$ on a set $A$ is \textbf{$k$-rich} if for every $j \le k$ and  $a \in A$, there exists $\omega \in \mathcal{D}$ such that $\omega(j) = a$. 
We say that $\mathcal{D}$ has \textbf{richness} $k$ if it is $k$-rich but not $(k+1)$-rich, and denote the richness by $\rich(\D)$. 

\begin{theorem}[{\cite[Lemma 1 and Theorem 1]{MRZ24}}]
\label{MRZ possible richness}
If $\mathcal{D}$ is a maximal ASPD on $n$ alternatives, then
\[
2 \le \rich(\D) \le \left\lfloor\frac{n}{2} \right\rfloor + 1. 
\]
\end{theorem}

The notion of richness has several interpretations in social choice theory; see \cite[\S4]{MRZ24} for details. 
Maximal ASPDs with minimal richness were classified in \cite[Theorem 2]{MRZ24}. 
It was posed in \cite[Problem 1]{MRZ24} to find a structural characterization of maximal richness. 
We provide such a characterization for all possible richness values in Theorem~\ref{richness in terms of R-vine}, thereby resolving the problem for maximal richness. 

\begin{lemma}\label{higher ranking}
Let $\mathcal{D}$ be a maximal ASPD on an $n$-element set $A$. 
Let $a \in A$ and $\omega \in \mathcal{D}$ with $\omega(k) = a$. 
Then for every $j \le k$, there exists $\omega' \in \mathcal{D}$ such that $\omega'(j) = a$. 
Consequently,
\[
\rich(\D) = \max\Set{k \in [n] | \text{ for every $a \in A$, there exists $\omega \in \mathcal{D}$ with $\omega(k) = a$}}.
\]
\end{lemma}

\begin{proof}
We argue by induction on $|A|$. 
The case $|A| \le 2$ is immediate. 
Suppose $|A| \ge 3$. 
Let $\mathcal{D}_1$ and $\mathcal{D}_2$ be the maximal ASPDs obtained from $\mathcal{D}$ by splitting (see Lemma~\ref{Slinko}). 
If $a$ is a bottom alternative in $\mathcal{D}$, then $a$ is a bottom alternative of $\mathcal{D}_{1}$ or $\mathcal{D}_{2}$. 
Therefore, by the induction hypothesis, the claim holds. 
Otherwise, apply the induction hypothesis to the domains $\mathcal{D}_1$ and $\mathcal{D}_2$. 
\end{proof}

\begin{theorem}\label{richness in terms of R-vine}
Let $\mathcal{D}$ be a maximal ASPD on $n$-element set $A$ and $\mathcal{V}$ the corresponding regular vine under the correspondence in Theorem~\ref{thm:vine-ASPD}. 
Then 
\begin{align*} 
\rich(\D) = \min\Set{k \in [n] | \bigcap_{\substack{S \in \mathcal{V} \\ |S| = k}}S \neq \varnothing}. \end{align*}
\end{theorem}

\begin{proof}
Let $r \coloneqq \rich(\D)$ and let $r'$ denote the right-hand side. 

First, suppose that $r > r'$. 
Let $a \in \bigcap_{S \in \mathcal{V}, \,|S| = r'} S$. 
By Lemma~\ref{higher ranking}, there exists $\omega \in \mathcal{D}$ such that $\omega(r'+1) = a$. 
Let 
\[
T \coloneqq \{\omega(1), \dots, \omega(r')\}.
\]
Then $T \in \mathcal{V}$ by Theorem~\ref{thm:vine-ASPD}. 
Since $|T| = r'$, we must have $a \in \bigcap_{S \in \mathcal{V}, \,|S| = r'} S \subseteq T$, which contradicts $a = \omega(r'+1) \notin T$. 
Thus, $r \le r'$. 

Next, we show that 
\[
\bigcap_{\substack{S \in \mathcal{V} \\ |S| = r}} S \neq \varnothing.
\]
By the maximality of $r$ (see Lemma~\ref{higher ranking}), there exists $a' \in A$ such that $\omega(r+1) \neq a'$ for all $\omega \in \mathcal{D}$. 
Applying Lemma~\ref{higher ranking} again, we obtain that $\omega(k) \neq a'$ for all $\omega \in \mathcal{D}$ and $k \ge r+1$.

Let $S \in \mathcal{V}$ with $|S| = r$. 
Then there exists $\omega \in \mathcal{D}$ such that 
\[
S = \{\omega(1), \dots, \omega(r)\}
\]
by Theorem~\ref{thm:vine-ASPD}. 
Hence $a' \in S$, and therefore $a'$ belongs to every such $S$, proving the claim.

By the minimality of $r'$, this implies $r \ge r'$. 
Consequently, $r = r'$, completing the proof.
\end{proof}

The following is an immediate consequence of Theorems~\ref{MRZ possible richness} and \ref{richness in terms of R-vine}. 

\begin{corollary}
\label{cor:min-max-rich}
Let $\mathcal{D}$ be a maximal ASPD on $n$-element set $A$. Then 
\begin{enumerate}
\item $\rich(\D) = 2$ if and only if the first associated tree of $\mathcal{V}$ is a star graph; 
\item $\rich(\D) = \lfloor n/2 \rfloor + 1$ if and only if 
\[
\bigcap_{\substack{S \in \mathcal{V} \\ |S| = \lfloor n/2 \rfloor}} S = \varnothing. 
\]
\end{enumerate}
\end{corollary}

By definition, all associated trees of a C-vine are star graphs; hence C-vines have minimal richness $2$. 
On the other hand, by the description of D-vines in Remark~\ref{rem:Dvine-typeA}, one verifies that the Condition ~\ref{cor:min-max-rich}(2) holds, so D-vines attain the maximal richness $\lfloor n/2 \rfloor + 1$. 
This further illustrates the strong correspondence between vine structures and single-peaked domains, with C-vines and D-vines representing the two extreme cases. 
%********************************************************************************************************
\section{Extremal lattices}
\label{sec:EL}

\subsection{Extremal lattices and regular vines}
\label{subsec:EL-vine}

In this subsection, we discuss a class of lattices arising in \emph{formal concept analysis} (FCA). 
We prove directly that these lattices are essentially the same as regular vines.

A \textbf{lattice} is a poset in which every pair of elements has a join and a meet. An element in a poset is called \textbf{join-irreducible} if $a = b \vee c$ implies $a=b$ or $a=c$ for any non-minimal element $a$.  
An \textbf{atom} in a lattice is an element that covers the minimal element. 
Every atom is join-irreducible. 
Let $B(k)=(2^{[k]}, \subseteq)$ denote the Boolean lattice on $[k]=\{1,\ldots,k\}$.
We say that a lattice $\mathcal{P}$ is \textbf{$B(k)$-free} if it does not contain an induced subposet isomorphic to $B(k)$. 

\begin{definition}
For integers $1 \le k \le n$, a lattice $\mathcal{P}$ is called an \textbf{$(n,k)$-extremal lattice} if the following conditions are satisfied:
\begin{enumerate}[(1)]
\item $\mathcal{P}$ has at most $n$ join-irreducible elements;
\item $\mathcal{P}$ is $B(k)$-free;
\item $\mathcal{P}$ has exactly 
$
\sum_{i=0}^{k-1} \binom{n}{i}
$
elements.
\end{enumerate}
\end{definition}

By \cite[Corollary~10]{AC17}, any lattice satisfying conditions (1) and (2) has at most $\sum_{i=0}^{k-1} \binom{n}{i}$ elements. Thus, condition (3) means that $\mathcal{P}$ attains the maximum possible size under these constraints, which justifies the term \emph{extremal}.

Extremal lattices arise naturally in FCA, which studies the structure of data through object--attribute relationships. We briefly recall the necessary definitions.

A \textbf{formal context} is a triple $\mathfrak{C} = (G, M, I)$ consisting of a set of objects $G$, a set of attributes $M$, and an incidence relation $I \subseteq G \times M$, where $(g,m) \in I$ indicates that the object $g$ has the attribute $m$.
For a subset $A \subseteq G$, define
\[
A^I \coloneqq \Set{ m \in M | (g,m) \in I \text{ for all } g \in A },
\]
and for $B \subseteq M$, define
\[
B^I \coloneqq \Set{ g \in G | (g,m) \in I \text{ for all } m \in B }.
\]

A pair $(A,B)$ is called a \textbf{formal concept} if $A^I = B$ and $B^I = A$. 
Here, $A$ and $B$ are called the \emph{extent} and \emph{intent} of the concept, respectively.
The set of all formal concepts is partially ordered by
\[
(A_1,B_1) \le (A_2,B_2) \iff A_1 \subseteq A_2 
\quad (\text{equivalently, } B_1 \supseteq B_2),
\]
and this poset forms a lattice $\mathcal{P}(\mathfrak{C})$, called the \textbf{concept lattice}.

An \textbf{implication} is an expression $X \Rightarrow Y$ for $X,Y \subseteq M$, 
which is said to hold if $X^I \subseteq Y^I$, that is, every object having all attributes in $X$ also has all attributes in $Y$. 
It is \emph{trivial} if $Y \subseteq X$, and \emph{nontrivial} otherwise.

An important example is the \textbf{contranominal scale}. For $k \in \mathbb{Z}_{>0}$, define
\[
\mathfrak{C}_k = ([k], [k], \neq),
\]
where $(i,j) \in I$ if and only if $i \ne j$. Thus each object has all attributes except one.

The concept lattice $\mathcal{P}(\mathfrak{C}_k)$ is isomorphic to the Boolean lattice $B(k)$. 
Equivalently, the context admits no nontrivial implications. Intuitively, the attributes are maximally independent: no attribute can be inferred from any combination of others.

\begin{example}
\label{ex:CS}
Consider the context $\mathfrak{C} = (G,M,I)$ with
\[
G = \{1,2,3\}, \quad M = \{a,b,c\},
\]
and incidence table
\[
\begin{array}{c|c|c|c}
 & a & b & c \\ \hline
1 & &  \checkmark &  \checkmark  \\ \hline
2 & \checkmark  &  & \checkmark \\ \hline
3 & \checkmark  & \checkmark &
\end{array}
\]
This is the contranominal scale on three elements. Its concept lattice is the Boolean lattice $B(3)$.
\end{example}

This example is maximal in the sense that all subsets of $M$ occur as intents. 
Any modification yields dependencies. For instance, adding $(3,c)$ yields the nontrivial implication $\{a\} \Rightarrow \{c\}$, which collapses part of the Boolean structure and reduces the number of concepts.

Thus, forbidding large Boolean sublattices can be viewed as restricting the independence of attributes. Extremal lattices correspond to contexts whose concept lattices are as large as possible under such constraints.

For $k=1$ and $k=2$, the structure is simple: every $(n,1)$-extremal lattice is a singleton, and every $(n,2)$-extremal lattice is a chain of length $n$. 
The first nontrivial case is therefore $k=3$, which we now study.
From the general theory developed in \cite{AC17}, we extract several results that will be used in this paper.

\begin{lemma}[{\cite[Corollary 16, Lemma 18]{AC17}}]
Let $\mathcal{P}$ be an $(n,3)$-extremal lattice. Then the following properties hold:
\begin{enumerate}[(1)]
\item $\mathcal{P}$ has exactly $n$ join-irreducible elements, and these are precisely the atoms of $\mathcal{P}$;
\item all maximal chains have length $n$, so $\mathcal{P}$ is graded. 
\end{enumerate}
\end{lemma}

Let $\mathcal{P}$ be an $(n,3)$-extremal lattice and let $\mathcal{C}$ be a maximal chain of $\mathcal{P}$. 
Let $\dot{\mathcal{C}} \coloneqq \Set{\dot{x} \mid x \in \mathcal{C}}$ be a disjoint copy of $\mathcal{C}$. 
The \textbf{doubling} $\mathcal{P}[\mathcal{C}]$ of $\mathcal{C}$ in $\mathcal{P}$ is the poset on
\[
\mathcal{P} \sqcup \dot{\mathcal{C}}
\]
whose order relation is defined by
\begin{align*}
x \le y 
&\text{ whenever } x \le y \text{ in } \mathcal{P}
\quad (x,y \in \mathcal{P}),\\
x \le \dot{y} 
&\text{ whenever } x \le y \text{ in } \mathcal{P}
\quad(x \in \mathcal{P},\, y \in \mathcal{C}),\\
\dot{x} \le \dot{y} 
&\text{ whenever } x \le y \text{ in } \mathcal{P}
\quad (x,y \in \mathcal{C}).
\end{align*}

\begin{example}
Figure~\ref{fig:doubling} illustrates the doubling construction for an $(4,3)$-extremal lattice along a maximal chain. 
The highlighted chain is duplicated to form the dotted copy. 
In the doubled poset, ordinary elements may lie below dotted elements, but dotted elements are never below ordinary elements.
\end{example}

\begin{figure}[htbp!]
 \centering
\begin{tikzpicture}[baseline=20]
\draw (1.5,-0.865) node[c](bot){$x_{0}$};
\draw (0, 0) node[c](0){$x_{1}$};
\draw (1, 0) node[v](1){};
\draw (2, 0) node[v](2){};
\draw (3, 0) node[v](3){};

\draw[double] (bot)--(0);
\draw (bot)--(1);
\draw (bot)--(2);
\draw (bot)--(3);

\draw (0.5, 0.865) node[c](01){$x_{2}$};
\draw[double] (0)--(01);
\draw (1)--(01);

\draw (1.5, 0.865) node[v](12){};
\draw (1)--(12);
\draw (2)--(12);

\draw (2.5, 0.865) node[v](13){};
\draw (1)--(13);
\draw (3)--(13);

\draw (1.0, 1.730) node[c](012){$x_{3}$};
\draw[double] (01)--(012);
\draw (12)--(012);

\draw (2.0, 1.730) node[v](123){};
\draw (12)--(123);
\draw (13)--(123);

\draw (1.5, 2.595) node[c](0123){$x_{4}$};

\draw[double] (012)--(0123);
\draw (123)--(0123);
\end{tikzpicture}
\qquad \qquad \qquad
 \begin{tikzpicture}[baseline=20]
\draw (1,-0.865) node[c](bot){$x_{0}$};
\draw (0, 0) node[c](0){$x_{1}$};
\draw (1, 0) node[v](1){};
\draw (2, 0) node[v](2){};
\draw (3, 0) node[v](3){};

\draw[double] (bot)--(0);
\draw (bot)--(1);
\draw (bot)--(2);
\draw (bot)--(3);

\draw (0.5, 0.865) node[c](01){$x_{2}$};
\draw[double] (0)--(01);
\draw (1)--(01);

\draw (1.5, 0.865) node[v](12){};
\draw (1)--(12);
\draw (2)--(12);

\draw (2.5, 0.865) node[v](13){};
\draw (1)--(13);
\draw (3)--(13);

\draw (1.0, 1.730) node[c](012){$x_{3}$};
\draw[double] (01)--(012);
\draw (12)--(012);

\draw (2.0, 1.730) node[v](123){};
\draw (12)--(123);
\draw (13)--(123);

\draw (1.5, 2.595) node[c](0123){$x_{4}$};

\draw[double] (012)--(0123);
\draw (123)--(0123);

\draw (-1, 0) node[c](dbot){$\dot{x_{0}}$};
\draw (-0.5, 0.865) node[c](d0){$\dot{x_{1}}$};
\draw (0, 1.730) node[c](d01){$\dot{x_{2}}$};
\draw (0.5, 2.595) node[c](d012){$\dot{x_{3}}$};
\draw (1, 3.46) node[c](d0123){$\dot{x_{4}}$};
\draw[double] (dbot)--(d0)--(d01)--(d012)--(d0123);

\draw (bot)--(dbot);
\draw (0)--(d0);
\draw (01)--(d01);
\draw (012)--(d012);
\draw (0123)--(d0123);
\end{tikzpicture}

\caption{A $(4,3)$-extremal lattice (left) with a maximal chain highlighted in double lines, and its doubling (right) along that chain.}
\label{fig:doubling}
\end{figure}
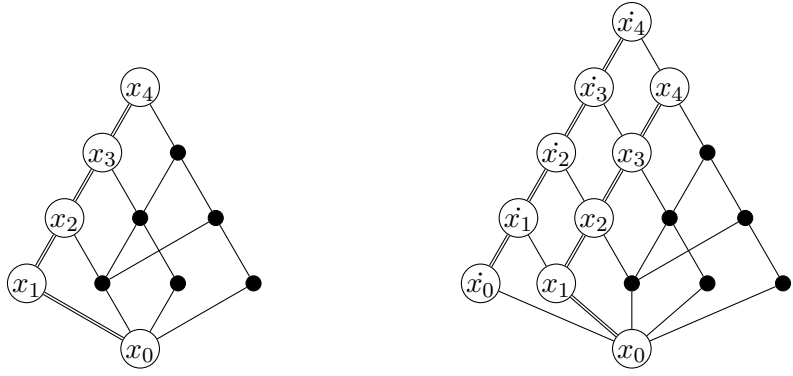

\begin{theorem}[{\cite[Theorem 23]{AC17}}]\label{AC17 Theorem 23}
Let $\mathcal{P}$ be an $(n-1,3)$-extremal lattice with $n \geq 2$ and let $\mathcal{C}$ be a maximal chain of $\mathcal{P}$. 
Then the doubling $\mathcal{P}[\mathcal{C}]$ is an $(n,3)$-extremal lattice. 
\end{theorem}

\begin{theorem}[{\cite[Theorem 27]{AC17}}]\label{AC17 Theorem 27}
Let $\mathcal{P}$ be an $(n,3)$-extremal lattice with $n \geq 2$. 
Then there exist an induced subposet $\mathcal{P}_{1} \subseteq \mathcal{P}$ and a maximal chain $\mathcal{C} \subseteq \mathcal{P}_{1}$ such that $\mathcal{P} = \mathcal{P}_{1}[\mathcal{C}]$. 
\end{theorem}

We will show that $(n,3)$-extremal lattices and regular vines are essentially equivalent notions. 
In this context, the realization of a regular vine as an induced subposet of a Boolean lattice is inessential; only its poset structure matters. 
Therefore, throughout this section, a regular vine will mean a poset satisfying the conditions in Definition~\ref{def:regular vine; poset}.

\begin{theorem} 
\label{thm:EL-vine}
A lattice $\mathcal{P}$ with minimum element $\hat0$ is an $(n,3)$-extremal lattice if and only if $\mathcal{P}\setminus \{\hat0\}$ is a regular vine.
\end{theorem}

\begin{proof}
Let $\mathcal{V} \coloneqq \mathcal{P}\setminus\{\hat{0}\}$. 
We proceed by induction on $n$, the number of atoms of $\mathcal{P}$. 
The case $n=1$ is trivial. 
Assume that $n \geq 2$. 

First, suppose that $\mathcal{P}$ is an $(n,3)$-extremal lattice, and we show that $\mathcal{V}$ is a regular vine. 
By Theorem~\ref{AC17 Theorem 27}, we have $\mathcal{P} = \mathcal{P}_{1}[\mathcal{C}]$, where $\mathcal{P}_{1}$ is an induced subposet of $\mathcal{P}$ that is an $(n-1,3)$-extremal lattice and $\mathcal{C}$ is a maximal chain of $\mathcal{P}_{1}$. 
By the induction hypothesis, $\mathcal{V}_{1} \coloneqq \mathcal{P}_{1}\setminus\{\hat{0}\}$ is a regular vine. 

Let $\dot{x} \in \dot{\mathcal{C}}$ be a non-minimal element of $\mathcal{V}$. 
By the definition of doubling, $\dot{x}$ covers exactly two elements, namely the original element $x \in \mathcal{C}$ and the element $\dot{y} \in \dot{\mathcal{C}}$, where $y$ is the element of $\mathcal{C}$ covered by $x$. 

Next, we show that, for each $i \in \{1, \dots, n-1\}$, the graph $T_{i}$ on $\mathcal{V}(i)$ with edge set $\mathcal{V}(i+1)$ is a tree. 
Suppose that $\dot{x} \in \dot{\mathcal{C}} \cap \mathcal{V}(i)$. 
Then, by the definition of doubling,
\[
\mathcal{V}(i) = \mathcal{V}_{1}(i) \cup \{\dot{x}\},
\]
and $\dot{x}$ is a leaf of $T_{i}$. 
Since $T_{i}\setminus \dot{x}$ is a tree by the induction hypothesis, it follows that $T_{i}$ is also a tree. 

Finally, we verify the proximity condition for $\mathcal{V}$. 
Suppose that $y \in \mathcal{P}_{1}$ and $\dot{x} \in \dot{\mathcal{C}}$ are covered by a common element. 
By the definition of doubling, this common upper cover must be $\dot{y} \in \dot{\mathcal{C}}$. 
Hence, $y$ and $\dot{x}$ both cover the element $x$. 
Therefore, $\mathcal{V}$ satisfies the proximity condition, and thus $\mathcal{V}$ is a regular vine. 

Conversely, suppose that $\mathcal{V}$ is a regular vine, and we show that $\mathcal{P}$ is an $(n,3)$-extremal lattice. 
Let $\mathcal{V}_{1}$ be the principal ideal generated by an element covered by the maximal element of $\mathcal{V}$. 
Then $\mathcal{V}_{1}$ is a regular vine with $n-1$ minimal elements. 
Let
\[
\dot{\mathcal{C}} \coloneqq \mathcal{V}\setminus\mathcal{V}_{1}.
\]
By Proposition~\ref{prop:card-R-vine}, the set $\dot{\mathcal{C}}$ contains exactly one element at each rank, and hence forms a maximal chain of $\mathcal{V}$. 

Since every non-minimal element of $\dot{\mathcal{C}}$ covers exactly two elements in $\mathcal{V}$, the subset $\mathcal{C}$ of $\mathcal{P}_{1} \coloneqq \mathcal{V}_{1} \cup \{\hat{0}\}$ defined by
\begin{align*}
\mathcal{C} \coloneqq \Set{x \in \mathcal{V}_{1} \mid \text{$x$ is covered by an element of $\dot{\mathcal{C}}$}} \cup \{\hat{0}\}
\end{align*}
contains exactly one element at each rank. 
Moreover, the proximity condition of $\mathcal{V}$ implies that $\mathcal{C}$ is a maximal chain of $\mathcal{P}_{1}$. 

By the induction hypothesis, $\mathcal{P}_{1}$ is an $(n-1,3)$-extremal lattice. 
Therefore,
$
\mathcal{P} = \mathcal{P}_{1}[\mathcal{C}]
$
is an $(n,3)$-extremal lattice by Theorem~\ref{AC17 Theorem 23}. 
\end{proof}

Chornomaz \cite{Chor16} showed that the automorphism group of an $(n,3)$-extremal lattice is a subgroup of the symmetric group $S_{2}$ of degree $2$. 
Moreover, Chornomaz obtained recurrence relations for the numbers $p_{n}$ and $q_{n}$ of isomorphism classes of $(n,3)$-extremal lattices whose automorphism groups are $S_{2}$ and $\{\mathrm{id}\}$, respectively. 
In fact, the observation that the numbers of isomorphism classes of $(n,3)$-extremal lattices and regular vines coincide for the first several values played an important role in discovering the equivalence between these two objects established in Theorem~\ref{thm:EL-vine}.
We therefore obtain the following.

\begin{corollary}[{\cite[Theorem 8.1]{Chor16}}]
\label{cor:recursive}
Let $\mathsf{F}$ be a species and \(\tilde{f}_n\) the number of unlabeled \(\mathsf{F}\)-structures described in Corollary \ref{cor:key-seq}.
Then 
\[
\tilde{f}_n = p_n + q_n,
\]
where 
\[
\begin{pmatrix}
p_1\\q_1
\end{pmatrix}
=
\begin{pmatrix}
0\\1
\end{pmatrix}, \quad
\begin{pmatrix}
p_2\\q_2
\end{pmatrix}
=
\begin{pmatrix}
1\\0
\end{pmatrix},
\]
and for $n \ge 3$,
\[
\begin{pmatrix}
p_n\\q_n
\end{pmatrix}
=
\begin{pmatrix}
2^{n-3} & 2^{n-3} \\ 
2^{n-4}(2^{n-4}-1) & 2^{n-4}(2^{n-3}-1)
\end{pmatrix}
\begin{pmatrix}
p_{n-2}\\q_{n-2}
\end{pmatrix}.
\]
\end{corollary}
 
Similar recursive formulas for maximal ASPDs also appear in \cite[Lemmas 1 and 2]{Karpov25}.

\subsection{Extremal binary matrices without triangles}
\label{subsec:EM}
In this subsection, we discuss extremal binary matrices with no triangles arising in combinatorial matrix theory. 
We show directly that these matrices are equivalent to $(n,3)$-extremal lattices and therefore also fit naturally into the splitting-and-merging framework.

There is a natural correspondence between $2^{[n]}$ and the set of binary column vectors $\{0,1\}^{n}$. 
Thus, every non-empty subset of $2^{[n]}$ can be represented as a binary matrix. 
A \textbf{triangle} is the binary matrix
\begin{align*}
\begin{pmatrix}
1 & 1 & 0 \\
1 & 0 & 1 \\
0 & 1 & 1
\end{pmatrix}. 
\end{align*}
A binary matrix $M$ is said to have \textbf{no triangles} if no row and column permutation of $M$ contains a triangle as a submatrix. 

Anstee \cite[Corollary 2.2]{anstee1980properties-joctsa} showed that a binary matrix with $n$ rows, distinct columns, and no triangles has at most $1+n+\binom{n}{2}$ columns. 
A binary matrix achieving this upper bound is called an \textbf{extremal binary matrix with no triangles}. 

Let $\mathcal{D} \subseteq \mathcal{L}([n])$ be a maximal ASPD and define
\begin{align*}
\mathcal{P} \coloneqq \Set{\{\omega(1), \dots, \omega(k)\} \mid \omega \in \mathcal{D},\ k \in [n]} \cup \{\varnothing\}  \subseteq 2^{[n]}. 
\end{align*}
Let $M$ be the binary matrix corresponding to $\mathcal{P}$. 

\begin{theorem}[{\cite[Theorem 2]{Karpov25}}]
\label{thm:Karpov-characterize}
The matrix $M$ is an extremal binary matrix with no triangles. 
Moreover, this construction gives a bijection between maximal ASPDs and extremal binary matrices with no triangles, up to row and column permutations.
\end{theorem}

Note that $\mathcal{P}$ is an $(n,3)$-extremal lattice by Theorems~\ref{thm:vine-ASPD} and~\ref{thm:EL-vine}. 
We now give a direct proof that $(n,3)$-extremal lattices are equivalent to extremal binary matrices with no triangles.

\begin{proposition}
\label{prop:EL-EM}
Let $\mathcal{P} \subseteq 2^{[n]}$ be a lattice such that $\{a\} \in \mathcal{P}$ for all $a \in [n]$, and let $M$ be the corresponding binary matrix. 
Then $\mathcal{P}$ is $B(3)$-free if and only if $M$ has no triangles. 
\end{proposition}

\begin{proof}
Suppose that $M$ contains a triangle. 
Then there exist $S_{1}, S_{2}, S_{3} \in \mathcal{P}$ and distinct elements $a_{1}, a_{2}, a_{3} \in [n]$ such that
\begin{align*}
\{a_{1},a_{2}\} &\subseteq S_{1}, \quad a_{3} \notin S_{1}, \\
\{a_{1},a_{3}\} &\subseteq S_{2}, \quad a_{2} \notin S_{2}, \\
\{a_{2},a_{3}\} &\subseteq S_{3}, \quad a_{1} \notin S_{3}.  
\end{align*}
Then
\begin{align*}
\{\varnothing, \{a_{1}\}, \{a_{2}\}, \{a_{3}\}, S_{1}, S_{2}, S_{3}, [n]\}
\end{align*}
forms an induced  subposet  of $\mathcal{P}$ isomorphic to $B(3)$. 

Conversely, suppose that $\mathcal{P}$ contains an induced  subposet  isomorphic to $B(3)$. 
Then there exist $T_{1}, T_{2}, T_{3} \in \mathcal{P}$ such that
\begin{align*}
\{\varnothing, T_{1}, T_{2}, T_{3}, T_{1}\vee T_{2}, T_{1}\vee T_{3}, T_{2}\vee T_{3}, [n]\}
\end{align*}
forms a $B(3)$-lattice. 

Choose
\begin{align*}
a_{1} &\in (T_{1}\vee T_{2}) \setminus T_{3}, \\
a_{2} &\in (T_{1}\vee T_{3}) \setminus T_{2}, \\
a_{3} &\in (T_{2}\vee T_{3}) \setminus T_{1}.
\end{align*}
Then the submatrix of $M$ corresponding to the rows $a_{1},a_{2},a_{3}$ and the columns
\begin{align*}
T_{1}\vee T_{2}, \quad T_{1}\vee T_{3}, \quad T_{2}\vee T_{3}
\end{align*}
is a triangle. 
Thus, $M$ has a triangle.
\end{proof}

 In particular, combining Proposition~\ref{prop:EL-EM} with the equivalences between maximal ASPDs, regular vines, and $(n,3)$-extremal lattices established in Theorems~\ref{thm:vine-ASPD} and~\ref{thm:EL-vine}, we recover the characterization of maximal ASPDs via extremal binary matrices with no triangles obtained by Karpov in Theorem~\ref{thm:Karpov-characterize}.
%********************************************************************************************************

 \bibliographystyle{abbrv}
\bibliography{references} 

@misc{OEIS,
  author = {N. J. A. Sloane and The OEIS Foundation Inc.},
  title = {The On-Line Encyclopedia of Integer Sequences},
  howpublished = {\url{https://oeis.org/}},
  year = {2010}
}

@incollection{kurowicka2010counting-dm,
	author = {Morales-N{\'a}poles, O. },
	editor = {Kurowicka, D.  and Joe, H. },
	title = {Counting {Vines}},
	isbn = {978-981-4299-87-9 978-981-4299-88-6},
	doi = {10.1142/9789814299886_0009},
	booktitle = {Dependence {Modeling}},
	publisher = {WORLD SCIENTIFIC},
	month = dec,
	year = {2010},
	pages = {189--218}
}

@incollection{kurowicka2010regular-dm,
	author = {Joe, H.  and Cooke, R. M.  and Kurowicka, D. },
	editor = {Kurowicka, D.  and Joe, H. },
	title = {Regular {Vines}: {Generation} {Algorithm} and {Number} of {Equivalence} {Classes}},
	isbn = {978-981-4299-87-9 978-981-4299-88-6},
	shorttitle = {Regular {Vines}},
	doi = {10.1142/9789814299886_0010},
	booktitle = {Dependence {Modeling}},
	publisher = {WORLD SCIENTIFIC},
	month = dec,
	year = {2010},
	pages = {219--231}
}

@article{anstee1980properties-joctsa,
	author = {Anstee, R. },
	title = {Properties of (0, 1)-matrices with no triangles},
	volume = {29},
	copyright = {https://www.elsevier.com/tdm/userlicense/1.0/},
	issn = {00973165},
	doi = {10.1016/0097-3165(80)90008-4},
	number = {2},
	journal = {Journal of Combinatorial Theory, Series A},
	month = sep,
	year = {1980},
	pages = {186--198}
}

@article{Chor16,
    AUTHOR = {Chornomaz, Bogdan},
	note = {\url{https://hal.science/hal-01175633v2}},
	title = {Counting extremal lattices},
	year = {HAL preprint, 2016}}

@article {AC17,
    AUTHOR = {Albano, Alexandre and Chornomaz, Bogdan},
     TITLE = {Why concept lattices are large: {E}xtremal theory for
              generators, concepts, and {VC}-dimension},
   JOURNAL = {Int. J. Gen. Syst.},
  FJOURNAL = {International Journal of General Systems},
    VOLUME = {46},
      YEAR = {2017},
    NUMBER = {5},
     PAGES = {440--457},
      ISSN = {0308-1079,1563-5104},
   MRCLASS = {68T30 (68Q32)},
  MRNUMBER = {3687440},
       DOI = {10.1080/03081079.2017.1354798},
       URL = {https://doi.org/10.1080/03081079.2017.1354798},
}

@article {KS23,
    AUTHOR = {Karpov, Alexander and Slinko, Arkadii},
     TITLE = {Constructing large peak-pit {C}ondorcet domains},
   JOURNAL = {Theory and Decision},
  FJOURNAL = {Theory and Decision. An International Journal for
              Multidisciplinary Advances in Decision Science},
    VOLUME = {94},
      YEAR = {2023},
    NUMBER = {1},
     PAGES = {97--120},
      ISSN = {0040-5833,1573-7187},
   MRCLASS = {91B14 (91B06 91B10)},
  MRNUMBER = {4528962},
MRREVIEWER = {Hannu\ J.\ Nurmi},
       DOI = {10.1007/s11238-022-09878-9},
       URL = {https://doi.org/10.1007/s11238-022-09878-9},
}

@article{MRZ24,
    AUTHOR = {Markstr\"om, Klas and Riis, S{\o}ren and Zhou, Bei},
	note = {\url{https://arxiv.org/abs/2401.12547}},
	title = {Arrow's single peaked domains, richness, and domains for plurality and the {B}orda count},
	year = {arXiv preprint, 2024}}

@article{Slinko24,
    AUTHOR = {Slinko, Arkadii},
	note = {\url{https://arxiv.org/abs/2412.05406}},
	title = {A combinatorial representation of {A}rrow's single-peaked domains},
	year = {arXiv preprint, 2024}}

@article{Liver20,
	author = {Liversidge, Georgina},
	note = {\url{https://arxiv.org/abs/2004.00751}},
	title = {Counting Condorcet Domains},
	year = {arXiv preprint, 2020}}

@book {BLL98,
    AUTHOR = {Bergeron, F. and Labelle, G. and Leroux, P.},
     TITLE = {Combinatorial species and tree-like structures},
    SERIES = {Encyclopedia of Mathematics and its Applications},
    VOLUME = {67},
      NOTE = {Translated from the 1994 French original by Margaret Readdy,
              With a foreword by Gian-Carlo Rota},
 PUBLISHER = {Cambridge University Press, Cambridge},
      YEAR = {1998},
     PAGES = {xx+457},
      ISBN = {0-521-57323-8},
   MRCLASS = {05A15 (05C30 05E05)},
  MRNUMBER = {1629341},
MRREVIEWER = {Ira\ Gessel},
}

@article {TTT24,
    AUTHOR = {Tran, Hung Manh and Tran, Tan Nhat and Tsujie, Shuhei},
     TITLE = {Vines and {MAT}-labeled graphs},
   JOURNAL = {Forum Math. Sigma},
  FJOURNAL = {Forum of Mathematics. Sigma},
    VOLUME = {12},
      YEAR = {2024},
     PAGES = {Paper No. e128, 29},
      ISSN = {2050-5094},
   MRCLASS = {06A07 (05C78 52C35)},
  MRNUMBER = {4842405},
       DOI = {10.1017/fms.2024.124},
       URL = {https://doi.org/10.1017/fms.2024.124},
}

@article {Puppe18,
    AUTHOR = {Puppe, Clemens},
     TITLE = {The single-peaked domain revisited: {A} simple global
              characterization},
   JOURNAL = {J. Econom. Theory},
  FJOURNAL = {Journal of Economic Theory},
    VOLUME = {176},
      YEAR = {2018},
     PAGES = {55--80},
      ISSN = {0022-0531,1095-7235},
   MRCLASS = {91B08},
  MRNUMBER = {3851502},
MRREVIEWER = {Vicki\ Knoblauch},
       DOI = {10.1016/j.jet.2018.03.003},
       URL = {https://doi.org/10.1016/j.jet.2018.03.003},
}

@article {Slinko19,
    AUTHOR = {Slinko, Arkadii},
     TITLE = {Condorcet domains satisfying {A}rrow's single-peakedness},
   JOURNAL = {J. Math. Econom.},
  FJOURNAL = {Journal of Mathematical Economics},
    VOLUME = {84},
      YEAR = {2019},
     PAGES = {166--175},
      ISSN = {0304-4068,1873-1538},
   MRCLASS = {91B14},
  MRNUMBER = {3994963},
       DOI = {10.1016/j.jmateco.2019.08.001},
       URL = {https://doi.org/10.1016/j.jmateco.2019.08.001},
}

@book {A63,
    AUTHOR = {Arrow, Kenneth J.},
     TITLE = {Social {C}hoice and {I}ndividual {V}alues},
 PUBLISHER = {second edition John Wiley \& Sons, Inc., New York; Chapman \& Hall, Ltd.,
              London},
      YEAR = {1963},
   MRCLASS = {90.0X},
  MRNUMBER = {39976},
MRREVIEWER = {D.\ Gale},
}

@incollection {M09,
    AUTHOR = {Monjardet, Bernard},
     TITLE = {Acyclic domains of linear orders: {A} survey},
 BOOKTITLE = {The mathematics of preference, choice and order},
    SERIES = {Stud. Choice Welf.},
     PAGES = {139--160},
 PUBLISHER = {Springer, Berlin},
      YEAR = {2009},
      ISBN = {978-3-540-79127-0},
   MRCLASS = {91B14},
  MRNUMBER = {2648300},
       DOI = {10.1007/978-3-540-79128-7\_8},
       URL = {https://doi.org/10.1007/978-3-540-79128-7_8},
}

@article{B48,
	author = {Black, D.},
	journal = {J. Political Economy},
 	pages = {23--34},
	title = {On the rationale of group decision-making},
	volume = {56},
	    NUMBER = {1},
	year = {1948}}

@article{Karpov25,
  author = {Karpov, A.},
  journal = {Dokl. RAN. Math. Inf. Proc. Upr.},
  pages = {102--108},
  title = {Arrow's single-peaked domains},
  volume = {525},
  year = {2025},
  note = {in Russian}
}

@article{I64,
	author = {Inada, K.},
	journal = {Econometrica},
 	pages = {525--531},
	title = {A note on the simple majority decision rule},
	volume = {32},
	year = {1964}}

@article{ER94,
    AUTHOR = {Edelman, Paul H. and Reiner, Victor},
     TITLE = {Free hyperplane arrangements between {$A_{n-1}$} and {$B_n$}},
   JOURNAL = {Math. Z.},
  FJOURNAL = {Mathematische Zeitschrift},
    VOLUME = {215},
      YEAR = {1994},
    NUMBER = {3},
     PAGES = {347--365},
      ISSN = {0025-5874,1432-1823},
   MRCLASS = {52B30},
  MRNUMBER = {1262522},
MRREVIEWER = {P.\ Orlik},
       DOI = {10.1007/BF02571719},
       URL = {https://doi.org/10.1007/BF02571719},
}

@article{ST06,
    AUTHOR = {Sommers, Eric and Tymoczko, Julianna},
     TITLE = {Exponents for {$B$}-stable ideals},
   JOURNAL = {Trans. Amer. Math. Soc.},
  FJOURNAL = {Transactions of the American Mathematical Society},
    VOLUME = {358},
      YEAR = {2006},
    NUMBER = {8},
     PAGES = {3493--3509},
      ISSN = {0002-9947,1088-6850},
   MRCLASS = {17B20 (05E15 14M15 20G05)},
  MRNUMBER = {2218986},
MRREVIEWER = {James\ E.\ Humphreys},
       DOI = {10.1090/S0002-9947-06-04080-3},
       URL = {https://doi.org/10.1090/S0002-9947-06-04080-3},
}

@article {CM20,
    AUTHOR = {Cuntz, Michael and M\"{u}cksch, Paul},
     TITLE = {M{AT}-free reflection arrangements},
   JOURNAL = {Electron. J. Combin.},
  FJOURNAL = {Electronic Journal of Combinatorics},
    VOLUME = {27},
      YEAR = {2020},
    NUMBER = {1},
     PAGES = {Paper No. 1.28, 19},
      ISSN = {1077-8926},
   MRCLASS = {52C35 (05B35 20F55 32S22 51F15)},
  MRNUMBER = {4061070},
MRREVIEWER = {Clement\ Radu\ Popescu},
       DOI = {10.37236/8820},
       URL = {https://doi.org/10.37236/8820},
}

@article{ABCHT16,
    AUTHOR = {Abe, Takuro and Barakat, Mohamed and Cuntz, Michael and Hoge,
              Torsten and Terao, Hiroaki},
     TITLE = {The freeness of ideal subarrangements of {W}eyl arrangements},
   JOURNAL = {J. Eur. Math. Soc. (JEMS)},
  FJOURNAL = {Journal of the European Mathematical Society (JEMS)},
    VOLUME = {18},
      YEAR = {2016},
    NUMBER = {6},
     PAGES = {1339--1348},
      ISSN = {1435-9855,1435-9863},
   MRCLASS = {52C35 (05E15 14N20 17B22)},
  MRNUMBER = {3500838},
MRREVIEWER = {Eric\ N.\ Sommers},
       DOI = {10.4171/JEMS/615},
       URL = {https://doi.org/10.4171/JEMS/615},
}

@article{T80,
    AUTHOR = {Terao, Hiroaki},
     TITLE = {Arrangements of hyperplanes and their freeness. {I}},
   JOURNAL = {J. Fac. Sci. Univ. Tokyo Sect. IA Math.},
  FJOURNAL = {Journal of the Faculty of Science. University of Tokyo.
              Section IA. Mathematics},
    VOLUME = {27},
      YEAR = {1980},
    NUMBER = {2},
     PAGES = {293--312},
      ISSN = {0040-8980},
   MRCLASS = {32C40 (05B25 51F15 52A37)},
  MRNUMBER = {586451},
MRREVIEWER = {P.\ Orlik},
}

@book {OT92,
    AUTHOR = {Orlik, Peter and Terao, Hiroaki},
     TITLE = {Arrangements of hyperplanes},
    SERIES = {Grundlehren der mathematischen Wissenschaften [Fundamental
              Principles of Mathematical Sciences]},
    VOLUME = {300},
 PUBLISHER = {Springer-Verlag, Berlin},
      YEAR = {1992},
     PAGES = {xviii+325},
      ISBN = {3-540-55259-6},
   MRCLASS = {52B30 (14F35 20F36 20F55 32S25 57N65)},
  MRNUMBER = {1217488},
MRREVIEWER = {Michel\ Yves\ Jambu},
       DOI = {10.1007/978-3-662-02772-1},
       URL = {https://doi.org/10.1007/978-3-662-02772-1},
}

@book {KJ11,
     TITLE = {Dependence modeling},
    EDITOR = {Kurowicka, Dorota and Joe, Harry},
      NOTE = {Vine copula handbook},
 PUBLISHER = {World Scientific Publishing Co. Pte. Ltd., Hackensack, NJ},
      YEAR = {2011},
     PAGES = {viii+360},
      ISBN = {978-981-4299-87-9; 981-4299-87-1},
   MRCLASS = {62-06 (62H05 62H20)},
  MRNUMBER = {2849701},
}

@article{TT23,
    AUTHOR = {Tran, Tan N. and Tsujie, Shuhei},
     TITLE = {M{AT}-free graphic arrangements and a characterization of
              strongly chordal graphs by edge-labeling},
   JOURNAL = {Algebr. Comb.},
  FJOURNAL = {Algebraic Combinatorics},
    VOLUME = {6},
      YEAR = {2023},
    NUMBER = {6},
     PAGES = {1447--1467},
      ISSN = {2589-5486},
   MRCLASS = {05C62 (05C75 05C78 52C35)},
  MRNUMBER = {4686122},
       DOI = {10.5802/alco.319},
       URL = {https://doi.org/10.5802/alco.319},
}

@article {ZK22,
    AUTHOR = {Zhu, Kailun and Kurowicka, Dorota},
     TITLE = {Regular vines with strongly chordal pattern of (conditional)
              independence},
   JOURNAL = {Comput. Statist. Data Anal.},
  FJOURNAL = {Computational Statistics \& Data Analysis},
    VOLUME = {172},
      YEAR = {2022},
     PAGES = {Paper No. 107461, 24},
      ISSN = {0167-9473,1872-7352},
   MRCLASS = {99-01},
  MRNUMBER = {4401907},
       DOI = {10.1016/j.csda.2022.107461},
       URL = {https://doi.org/10.1016/j.csda.2022.107461},
}

@article {CKW15,
    AUTHOR = {Cooke, R. M. and Kurowicka, D. and Wilson, K.},
     TITLE = {Sampling, conditionalizing, counting, merging, searching
              regular vines},
   JOURNAL = {J. Multivariate Anal.},
  FJOURNAL = {Journal of Multivariate Analysis},
    VOLUME = {138},
      YEAR = {2015},
     PAGES = {4--18},
      ISSN = {0047-259X,1095-7243},
   MRCLASS = {62H15 (62H05 62H10 62P05)},
  MRNUMBER = {3348830},
       DOI = {10.1016/j.jmva.2015.02.001},
       URL = {https://doi.org/10.1016/j.jmva.2015.02.001},
}

@article {BC02,
    AUTHOR = {Bedford, Tim and Cooke, Roger M.},
     TITLE = {Vines---a new graphical model for dependent random variables},
   JOURNAL = {Ann. Statist.},
  FJOURNAL = {The Annals of Statistics},
    VOLUME = {30},
      YEAR = {2002},
    NUMBER = {4},
     PAGES = {1031--1068},
      ISSN = {0090-5364,2168-8966},
   MRCLASS = {62E10 (60E05 62H20)},
  MRNUMBER = {1926167},
MRREVIEWER = {Friedrich\ Liese},
       DOI = {10.1214/aos/1031689016},
       URL = {https://doi.org/10.1214/aos/1031689016},
}

@article {KC06,
    AUTHOR = {Kurowicka, D. and Cooke, R. M.},
     TITLE = {Completion problem with partial correlation vines},
   JOURNAL = {Linear Algebra Appl.},
  FJOURNAL = {Linear Algebra and its Applications},
    VOLUME = {418},
      YEAR = {2006},
    NUMBER = {1},
     PAGES = {188--200},
      ISSN = {0024-3795,1873-1856},
   MRCLASS = {62H20 (15A99)},
  MRNUMBER = {2257589},
MRREVIEWER = {Tamas\ Lengyel},
       DOI = {10.1016/j.laa.2006.01.031},
       URL = {https://doi.org/10.1016/j.laa.2006.01.031},
}
   %********************************************************************************************************

\makeatletter
\@setaddresses
\global\let\@setaddresses\relax
\makeatother

%\end{document}
\newpage

\section*{Appendix: Catalog}

We give here a catalog of MAT-labeled complete graphs, regular vines, and maximal ASPDs of order $3\le n\le 6$ under the correspondences in Theorems~\ref{thm:graph-vine}, \ref{thm:graph-ASPD}, and~\ref{thm:vine-ASPD}.

\vskip 3em
 
\input{catalogdata3-5.txt}

\newpage
\input{catalogdata6.txt}

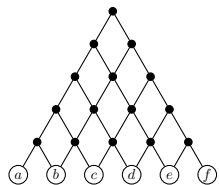
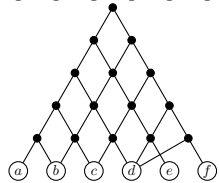
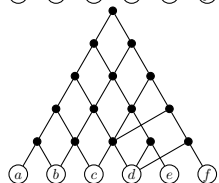
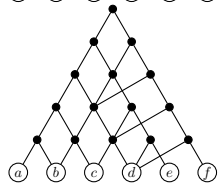
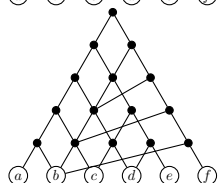
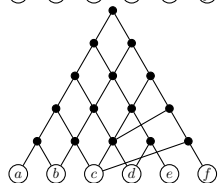
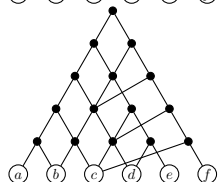
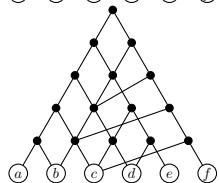
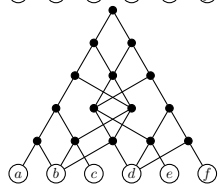
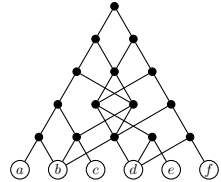
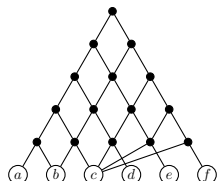
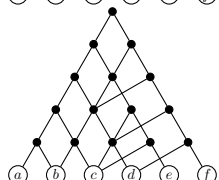
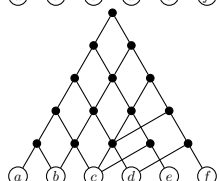
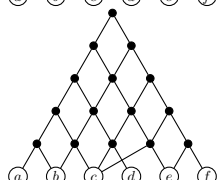
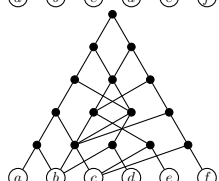
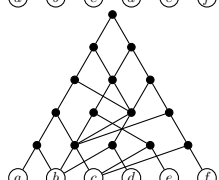
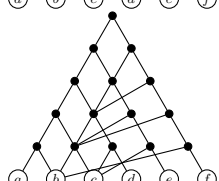
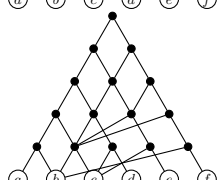
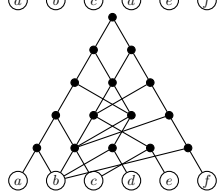
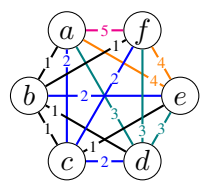
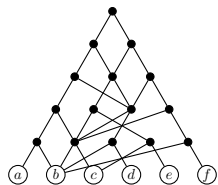
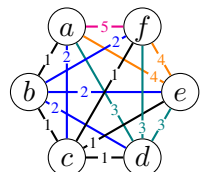
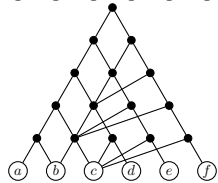
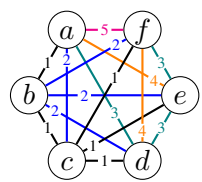
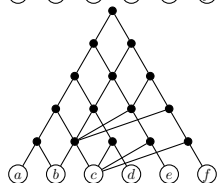
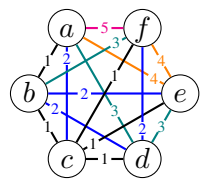
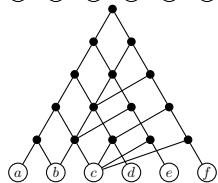
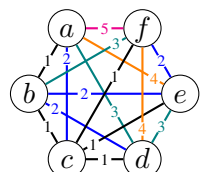
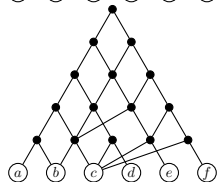
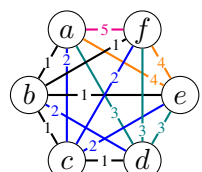
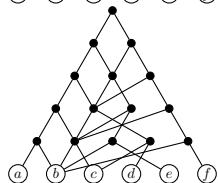
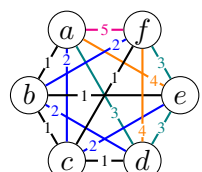
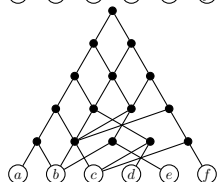
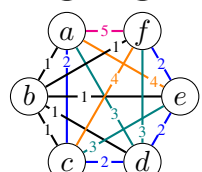
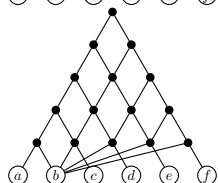
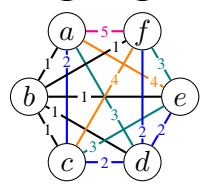
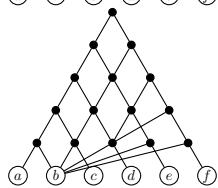
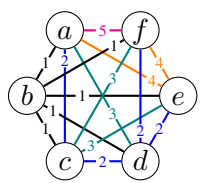
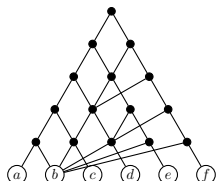
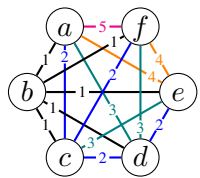
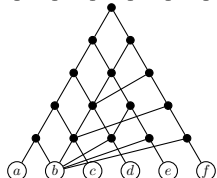
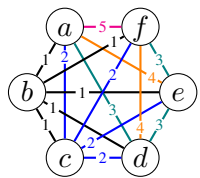
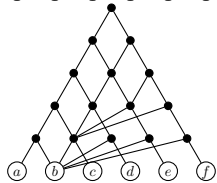
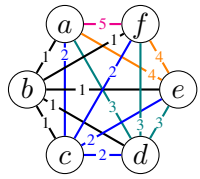
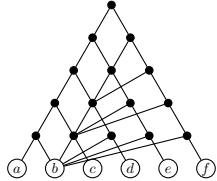
\end{document}